\documentclass[preprint]{imsart}

\arxiv{arXiv:1901.01163}

\startlocaldefs
\usepackage{amsmath,amssymb,amsthm,amsfonts,amsbsy,latexsym,dsfont,tikz,bbm,verbatim}
\usepackage[square,sort,comma,numbers]{natbib}
\RequirePackage[colorlinks,citecolor=blue,urlcolor=blue]{hyperref}

\usepackage{enumerate}

\newcommand{\Cov}[0]{\text{Cov}}
\newcommand{\Var}[0]{\text{Var}}

\newcommand{\R}[0]{\mathbb{R}}

\newcommand{\kindex}{{\overline{S}^{(i_1)}_{2,k}}}
\newcommand{\kindone}{{\overline{S}^{(i_1)}_{2,1}}}
\newcommand{\Pro}{\mathbb{P}}
\newcommand{\Exp}{\mathbb{E}}

\theoremstyle{theorem}
\newtheorem{theorem}{Theorem}[section]
\newtheorem{lemma}[theorem]{Lemma}
\newtheorem{corollary}[theorem]{Corollary}

\theoremstyle{definition}

\newtheorem{remark}[theorem]{Remark}

\newtheorem{ex}[theorem]{Example}
   
\renewcommand{\leq}{\leqslant} 
\renewcommand{\geq}{\geqslant}

\newcommand{\ind}{\mathds{1}}

\def\qed{ \hfill $\blacksquare$}  

\newcommand{\cB}{\mathcal{B}}
\newcommand{\cD}{\mathcal{D}}

\newcommand{\cJ}{\mathcal{J}}

\newcommand{\cR}{\mathcal{R}}
\newcommand{\cS}{\mathcal{S}}
\newcommand{\cV}{\mathcal{V}}
\newcommand{\cY}{\mathcal{Y}}  

\newcommand{\bN}{\mathbb{N}}
\newcommand{\bR}{\mathbb{R}}

\DeclareMathOperator{\E}{\mathds{E}}

\endlocaldefs

\begin{document}

\begin{frontmatter}

\title{Approximating high-dimensional infinite-order $U$-statistics: statistical and computational guarantees}

\runtitle{High-dimensional infinite-order $U$-statistics}

\begin{aug}

\author{\fnms{Yanglei} \snm{Song}\thanksref{a}\corref{}\ead[label=e1]{yanglei.song@queensu.ca}}
\author{\fnms{Xiaohui} \snm{Chen}\thanksref{b}\ead[label=e2]{xhchen@illinois.edu}}
\and
\author{\fnms{Kengo} \snm{Kato}\thanksref{c}\ead[label=e3]{kk976@cornell.edu}}
\address[a]{
Department of Mathematics and Statistics, Queen's University, 48 University Ave, Kingston, ON,  Canada, K7L 3N6
\\
\printead{e1}}
\address[b]{
Department of Statistics, University of Illinois at Urbana-Champaign, 725 S. Wright Street, Champaign, IL 61820
\\
\printead{e2}}
\address[c]{
Department of Statistics and Data Science, Cornell University, 1194 Comstock Hall, Ithaca, NY 14853 
\\
\printead{e3}}

\runauthor{Y. Song, X. Chen, K. Kato}

\affiliation{University of Illinois at Urbana-Champaign and Cornell University}

\end{aug}

\begin{abstract}
We study the problem of distributional approximations to high-dimensional non-degenerate $U$-statistics with random kernels of diverging orders. Infinite-order $U$-statistics (IOUS) are a useful tool for constructing simultaneous prediction intervals that quantify the uncertainty of ensemble methods such as subbagging and random forests. A major obstacle in using the IOUS is their computational intractability when the sample size and/or order are large. In this article, we derive non-asymptotic Gaussian approximation error bounds for an incomplete version of the IOUS with a random kernel. We also study data-driven inferential methods for the incomplete IOUS via bootstraps and develop their  statistical and computational guarantees.
\end{abstract}

\begin{keyword}
\kwd{Infinite-order $U$-statistics}
\kwd{incomplete $U$-statistics}
\kwd{Gaussian approximation}
\kwd{bootstrap}
\kwd{random forests}
\kwd{uncertainty quantification}
\end{keyword}



\end{frontmatter}

\section{Introduction}\label{sec:introduction}
Let $X_{1},\dots,X_{n}$ be independent and identically distributed (i.i.d.) random variables taking value in a measurable space $(S,\cS)$ with common distribution $P$, and let  $h: S^{r} \to \bR^d$ be a symmetric and measurable function with respect to the product space $S^{r}$ equipped with the product $\sigma$-field $\cS^{r} = \cS \otimes \cdots \otimes \cS$ ($r$ times). Assume $\Exp [|h_{j}(X_1,\ldots,X_r)|] < \infty$ for $1 \leq j \leq d$, and consider the statistical inference on the mean vector $\theta = (\theta_1,\ldots, \theta_d)^T= \Exp[h(X_{1},\dots,X_{r})]$. A natural estimator for $\theta$ is the $U$-statistic with kernel $h$:
\begin{equation}
\label{eqn:ustat}
U_n := \frac{1}{|I_{n,r}|} \sum_{\iota \in I_{n,r}} h(X_{i_1},\ldots, X_{i_r}) :=\frac{1}{|I_{n,r}|} \sum_{\iota \in I_{n,r}} h(X_{\iota}),
\end{equation}
where $I_{n,r} := \{\iota = (i_1,\ldots, i_r): 1 \leq i_1 < \ldots < i_r \leq n\}$ is the set of all ordered $r$-tuples of $1,\dots,n$ and $|\cdot|$ denotes the set cardinality. The positive integer $r$ is called the order or degree of the kernel $h$ or the $U$-statistic $U_{n}$. We refer to \cite{lee1990} as an excellent monograph on $U$-statistics.

In the present paper, we are interested in the situation where the order $r$ may be nonneglible relative to the sample size $n$, i.e., $r = r_n \to \infty$ as $n \to \infty$. 
$U$-statistics with divergent orders are called {\it infinite-order $U$-statistic} (IOUS) \citep{Frees1989}. 
IOUS has attracted renewed interests in the recent statistics and machine learning literature in relation to uncertainty quantification for Breiman's bagging \cite{Breiman1996} and random forests \cite{Breiman2001}. In such applications, the tree-based prediction rules can be thought of as $U$-statistics with deterministic and random kernels, respectively, and their order corresponds to the sub-sample size of the training data \cite{mentch2016quantifying}. Statistically, the subsample size $r$ used to build each tree needs to increase with the total sample size $n$ to produce reliable predictions.  As a leading example, we consider  construction of simultaneous prediction intervals for a version of random forests discussed in~\cite{mentch2016quantifying}. 

\begin{ex}[Simultaneous prediction intervals for random forests]\label{ex:random_forest} 
Consider a training dataset of size $n$,
$\{(Y_{1},Z_{1}), \dots, (Y_{n},Z_{n})\} = \{ X_{1},\dots,X_{n} \} = X_1^n$, where $Y_{i} \in \cY$ is a vector of features and $Z_{i} \in \bR$ is a response. Let $h$ be a deterministic prediction rule that takes as input a sub-sample $\{X_{i_{1}}, \dots, X_{i_{r}}\}$ 
and outputs predictions on $d$ testing points $(y_{1}^{*}, \dots, y_{d}^{*})$ in the feature space $\cY$. Then $U_n$ in~\eqref{eqn:ustat} are the overall predictions 
by averaging over all possible sub-samples of size $r$. 

For random forests~\cite{Breiman2001, mentch2016quantifying}, the tree-based prediction rule may be constructed  with additional randomness: in building a tree or multiple trees based on a sub-sample,  the split at each node may only occur on a randomly selected subset of features. 
Thus, let $\{W_{\iota}: \iota \in I_{n,r}\}$ be a collection of i.i.d.~random variables taking value in a measurable space $(S', \cS')$ that  are independent of the  data $X_1^n$, and that determine the potential splits for each sub-sample. Here, each $W_{\iota}$ captures the random mechanism in building a prediction function based on $X_{\iota} = (X_{i_1},\ldots,X_{i_r})$, but are assumed to be independent for different sub-samples. Further, let $H: S^r \times S' \to \bR^d$ be an $\cS^{r} \otimes\cS'$-measurable function, that represents the random forest algorithm, such that 
$\Exp[H(x_1,\ldots,x_r, W)] = h(x_1,\ldots, x_r)$. Then predictions of random forests are given by a $d$-dimensional $U$-statistic 
with random kernel $H$: 
\begin{equation}
\label{eqn:ustat_random_kernel}
\widehat{U}_n := 
|I_{n,r}|^{-1} \sum_{\iota \in I_{n,r}} H(X_{i_1},\ldots,X_{i_r}, W_{\iota}) 
=
|I_{n,r}|^{-1} \sum_{\iota \in I_{n,r}} H(X_{\iota}, W_{\iota}),
 \end{equation}
where the random kernel $H$ varies with $r$. 
\end{ex}

Compared to $U$-statistics with fixed orders (i.e., $r$ being fixed), the analysis of IOUS brings  nontrivial computational and statistical challenges due to increasing orders.  First, even for a moderately large value of $r$, exact computation of all possible ${n \choose r}$ trees is intractable. For diverging $r$, it is not possible to compute $U_n$ in
polynomial-time of $n$.
Second, the variance of the H\'ajek projection (i.e., the first-order term in the Hoeffding decomposition  \cite{hoeffding1948class}) of $U_n - \theta$ tends to zero as $r \to \infty$. To wit, define a function $g: S \to [0,\infty)$ by $g(x_1) = \Exp[h(x_1,X_2,\ldots,X_r)]$, and
$$
\sigma_{g,j}^2 := \Exp[(g_j(X_1) - \theta_j)^2] \text{ for } 1 \leq j \leq d, \quad
\underline{\sigma}^2_{g} := \min_{1 \leq j \leq d} \sigma^2_{g,j}.
$$
Then the H\'ajek projection of $U_n - \theta$ is given by $n^{-1} r \sum_{i=1}^{n}(g(X_i) - \theta)$. By the orthogonality of the projections, we have
$$
\Exp[ (h_j(X_1,\ldots,X_r) - \theta_j)^2] \geq \sum_{1 \leq i \leq r} \Exp[(g_j(X_i)-\theta_j)^2] =
r \sigma^2_{g,j}.
$$
Thus the variances of the kernel $h$ and its associated H\'ajek projection $g$ have different magnitudes. In particular, if the variance of $h_{j}(X_1,\ldots,X_r)$ is bounded by a constant $C > 0$ (which is often assumed for random forests, cf. \cite{mentch2016quantifying}), then ${\sigma}^2_{g,j} \leq C/r$, which vanishes as $r$ diverges. Thus standard Gaussian approximation results in literature are no longer applicable in our setting since they require that the componentwise variances are bounded below from zero to avoid degeneracy, i.e., there is an absolute constant $\underline{\sigma}^2 > 0$ such that $\underline{\sigma}^2_{g} \geq \underline{\sigma}^2$ (cf.~\cite{chernozhukov2013gaussian,chernozhukov2017,chen2018gaussian}).  

In this work, our focus is to derive computationally tractable and statistically valid sub-sampling procedures for making inference on $\theta$ with a class of {\it high-dimensional} random kernels (i.e., large $d$) of {\it diverging} orders (i.e., increasing $r$). To break the computational bottleneck, we consider the incomplete version of $\widehat{U}_{n}$ by sampling (possibly much) fewer terms than $|I_{n,r}|$. In particular, we consider the Bernoulli sampling scheme introduced in~\cite{chen2017randomized}. Given a positive integer $N$, which represents our computational budget, define the sparsity design parameter $p_n := N/|I_{n,r}|$, 
and let $\{Z_{\iota}: \iota \in I_{n,r}\}$ be a collection of i.i.d.~Bernoulli random variables with success probability $p_n$, that are independent of the data $X_1^n$ and  $\{W_{\iota}:\iota \in I_{n,r}\}$. 
Consider the following incomplete $U$-statistic (on the data $X_{1}^{n}$) with random kernel and weights: 
\begin{equation}
\label{eqn:incomplete_ustat_random_kernel}
U'_{n,N} := \widehat{N}^{-1} \sum_{\iota \in I_{n,r}} Z_{\iota} H(X_{\iota}, W_{\iota}), \text{ where }
\widehat{N} := \sum_{\iota \in I_{n,r}} Z_{\iota}.
\end{equation}
Obviously, $U'_{n,N}$ is an unbiased estimator of $\theta$ and it only involves computing $\widehat{N}$ terms, which on average is much smaller than $|I_{n,r}|$ if $p_{n} \ll 1$. 

When the kernel $h$ is both deterministic and of fixed order, finite sample  bounds for the Gaussian and bootstrap approximations of $U'_{n,N} - \theta$ (after a suitable normalization) are established in \cite{chen2017randomized}. Roughly speaking, error bound analysis in \cite{chen2017randomized} has two major steps: $i)$ establish the Gaussian approximation to the H\'ajek projection, and $ii)$ bound the maximum norm of all higher-order degenerate terms. As discussed above, the first-order H\'ajek projection in the Hoeffding decomposition is asymptotically vanishing for the IOUS, and we must control the moments of an increasing number of degenerate terms, which makes the analysis of the incomplete IOUS with random kernels substantially more subtle. 

In Section \ref{sec:GAR_IOUS}, we derive non-asymptotic Gaussian approximation error bounds for approximating the distribution of the incomplete IOUS $U'_{n,N}$ with random kernels subject to sub-exponential moment conditions. Specifically, our rates of convergence for the Gaussian approximation of $U'_{n,N}$ have the explicit dependence on all parameters $(n,N,d,r,\underline{\sigma}_g^2, D_{n})$, where $D_{n}$ is an upper bound for the $\psi_{1}$ norms of the random kernels (for precise statements, see conditions \eqref{h_exp_assumption}, \eqref{exp_cond_distr}, and \eqref{H_exp_moment_assumption} ahead). In particular, asymptotic validity of the Gaussian approximation can be achieved if $\underline{\sigma}_g^{-2}r^{2} D_{n}^{2} \log^{7}(dn) = o(n \wedge N)$. The order of $\underline{\sigma}_g^{-2}$ will be application specific. 
As we shall verify in Section~\ref{sec:apps}, under certain regularity conditions, 
\begin{equation}
\label{eqn:sigma_g_lower_bound_condition}
\underline{\sigma}_g^{-2} = O(r^2).
\end{equation}
It is worth noting that (\ref{eqn:sigma_g_lower_bound_condition}) is sharp in the sense that for the linear kernel $h(x_{1},\cdots,x_{r}) = (x_{1}+\cdots+x_{r})/r$,  we have $\underline{\sigma}_g^{-2} \asymp r^2$ if $c \leq \Var(X_{1j}) \leq C$. If further $D_n = O(1)$, $\log(d) = O(\log(n))$ and $n = O(N)$ (i.e., the computational complexity is at least linear in sample size), then the order of $U'_{n,N}$ is allowed to increase at the rate of $ r = o(n^{1/4-\epsilon})$ for any $\epsilon \in (0, 1/4)$. On the other hand, the dimension may grow exponentially fast in sample size (i.e., $d = O(e^{n^{c}})$ for some constant $c \in (0, 1/7)$) to maintain the asymptotic validity of Gaussian approximations while $r$ is still allowed to increase at a polynomial rate in $n$. 

The proof of our Gaussian approximation results for IOUS builds upon a number of recently developed technical tools such as Gaussian approximation results for sum of independent random vectors and $U$-statistics of fixed orders \cite{chernozhukov2013gaussian,chernozhukov2017,chen2018gaussian,chenkato2017a}, anti-concentration inequality for Gaussian maxima \cite{chernozhukov2015comparison}, and iterative conditioning argument for high-dimensional incomplete $U$-statistics (with the fixed kernel and order) \cite{chen2017randomized}. However, there are three technical innovations in our proof to accommodate the issues of diverging orders and randomness of the kernel. First, we use the {\it iterative renormalization} for each dimension of $g$ and also $H$ by its variance. This simple trick turns out to be the crux to avoid the lower bound assumption for Gaussian approximation in the literature~\cite{chernozhukov2017,chen2017randomized}.  
Second, we derive an order-explicit maximal inequality for the expected supremum of the remainder of the H\'ajek projection of the IOUS (cf. Section \ref{sec:maximal_inequality}). This maximal inequality is new in literature and our main tools include a symmetrization inequality of~\cite{sherman1994maximal} and Bonami inequality~\cite[Theorem 3.2.2]{de2012decoupling} for the Rademacher chaos, both with the explicit dependence on $r$. Third, we develop new tail probability inequalities for $U$-statistics with random kernels by leveraging the independence between $\{W_{\iota}, \iota\in I_{n,r}\}$ and the data $X_1^n$.

In Section \ref{sec:bootstrap}, we derive computationally tractable and fully data-driven inferential methods of $\theta$ based on the incomplete IOUS when the sample size $n$, the dimension $d$, and the order $r$, are all large. We consider a multiplier bootstrap procedure consisting of two partial bootstraps that are conditionally independent given $X_1^n$ and $\{W_{\iota}, Z_{\iota} : \iota\in I_{n,r}\}$: one estimates the covariance matrix of the randomized kernel, and the other estimates the H\'ajek projection. The latter is usually computationally demanding, and we develop a divide and conquer algorithm to maintain the overall computational cost of our multiplier bootstrap procedure at most $O(n^{2}d + B(N+n)d)$, where $B$ denotes the number of bootstrap iterations. Thus the computational cost of the bootstrap to approximate the sampling distribution for incomplete IOUS can be made independent of the order $r$, even though $r$ diverges. 


In Section \ref{sec:apps}, we discuss the key non-degeneracy condition \eqref{eqn:sigma_g_lower_bound_condition} for deriving the validity of Gaussian and bootstrap approximations. We provide a general embedding scheme where a Cram\'er-Rao type lower bound can be established for the minimum $\underline{\sigma}_{g}^{2}$ of the projection variances. Specifically, the lower bound for $r^2 \underline{\sigma}_{g}^{2}$ only involves  the sensitivity of $\Exp[h(X_1,\ldots,X_r)]$ under perturbation and the Fisher information of the embedded family, which in some cases remain constants as $r$ diverges.
In non-parametric regressions, there is a natural embedding of the response variable into a location family such that the sensitivity and Fisher information can be explicitly computed.

\subsection{Connections to the literature}
For univariate $U$-statistics ($d = 1$), the asymptotic
distributions are derived in the seminal paper \cite{hoeffding1948class} for the non-degenerate
case.
\citep{Frees1989} introduced the notion ``infinite-order $U$ statistics" (IOUS) with diverging orders and established the central limit theorem for $U_n$ when $d =1$. For univariate IOUS, asymptotic normality of IOUS can be found in \cite[Chapter 4.6]{borovskikh1996}, and the Berry-Esseen type bounds for IOUS were established by~\cite{friedrich1989berry,vanEs1988,Zwet1984}. Further, \citep{mentch2016quantifying} applied IOUS to construct a prediction interval for one test point. However, $i).$ \citep{mentch2016quantifying} does not address the issue that the variance of the H\'ajek projection is vanishing:
the two conditions in Theorem 1 therein,
$\Exp h_{k_n}(Z_1,\ldots,Z_{k_n}) \leq C <\infty$ and
$\lim \zeta_{1,k_n} \neq 0$, are not compatible based on our previous discussions
; $ii).$ in practice, the size $d$ of a test set  may be comparable to or even much larger than the size $n$ of a training set, and the current work is motivated by such consideration. Limit theorems of the related infinite-order $V$-statistics and the infinite-order $U$-processes were studied in \cite{Shieh1994_SPL,HeiligNolan2001_SS}.  The high-dimensional Gaussian approximation results and bootstrap methods were established in~\cite{chernozhukov2013gaussian,chernozhukov2017} for sum of independent random vectors, and 
in~\cite{chen2018gaussian,chen2017randomized} for $U$-statistics.  We refer readers to these references for extensive literature review. 

Incomplete $U$-statistics were first introduced in \cite{blom1976}, which can be viewed as a special case of weighted $U$-statistics. There is a large literature on limit theorems for weighted $U$-statistics; see \cite{shapirohubert1979,oneilredner1993,major1994,rifiutzet2000}. The asymptotic distributions of incomplete $U$-statistics (for fixed $d$) were derived in \cite{brownkildea1978} and \cite{Janson1984_PTRF}; see also Section 4.3 in \cite{lee1990} for a review on incomplete $U$-statistics. Recently, incomplete U-statistics have gained renewed interests in the statistics and machine learning literatures \cite{ClemenconLugosiVayatis2008_AoS,mentch2016quantifying}. To the best of our knowledge, the current paper is the first work that establishes distributional approximation theorems for incomplete IOUS with random kernels and increasing orders in high dimensions.

The remaining of the paper is organized as follows. We develop Gaussian approximation results for above $U$-statistics in Section~\ref{sec:GAR_IOUS}, and bootstrap methods for the variance of the approximating Gaussian distribution
in Section~\ref{sec:bootstrap}. We apply the theoretical results to several examples in Section~\ref{sec:apps}.
We highlight a maximal inequality 
in Section~\ref{sec:maximal_inequality}, and present  all other proofs  in Appendix \ref{sec:proofs}. 

\subsection{Notation} \label{subsec:notation}
We write $\text{l.h.s.} \lesssim \text{r.h.s.}$ if there exists a finite and positive absolute constant $C$ such that
$\text{l.h.s.} \leq  C\times \text{r.h.s.}$.
We shall use $c, C, C_{1}, C_{2}, \dots$ to denote finite and positive absolute constants, whose value may differ from place to place. We denote $X_i,  \ldots X_{i'}$ by $X_i^{i'}$ for $i \leq i'$. 

For $a,b \in \bR$, let 
$\lfloor a \rfloor$ denote the largest integer that does not exceed $a$,
$a \vee b = \max\{a,b\}$ and 
$a \wedge b = \min\{a,b\}$. For $a,b \in \bR^d$, we write $a \leq b$ if
$a_j \leq b_j$ for $1 \leq j \leq d$, and write $[a,b]$ for the hyperrectangle $\prod_{j=1}^{d}[a_j,b_j]$ if $a \leq b$. 
We denote by $\cR := \{\prod_{j=1}^d [a_j,b_j]: -\infty \leq a_j \leq b_j \leq \infty\}$  the collection of hyperrectangles in $\bR^d$.
Further, for $a \in \bR^d$, $r,t \in \bR$, $ra + t$ is a vector in $\bR^d$ with $j^{th}$ component being $r a_j + t$. For a matrix $A =(a_{ij})$, denote $\|A\|_{\infty} = \max_{i,j} |a_{ij}|$. For a diagonal matrix $\Lambda$ with positive diagonal entries, $\Lambda^{-1/2}$ (resp. $\Lambda^{1/2}$) is the diagonal matrix, with $j$-th diagonal entry being $\Lambda_{jj}^{-1/2}$ (resp. $\Lambda_{jj}^{1/2}$).

For $\beta > 0$, let $\psi_{\beta}: [0,\infty) \to \bR$ be a function defined by
$\psi_{\beta}(x) = e^{x^\beta}-1$, and for any real-valued random variable $\xi$, define
$\|\xi\|_{\psi_{\beta}} = \inf\{C > 0: \Exp[\psi_{\beta}(|\xi|/C)] \leq 1\}$.
Further, we define a family of functions $\{\widetilde{\psi}_{\beta}(\cdot)\}$ on $[0,\infty)$ indexed by $\beta > 0$. For $\beta \geq 1$, define $\widetilde{\psi}_{\beta} = \psi_{\beta}$.
For $\beta \in (0,1)$, define $\tau_{\beta} = (\beta e)^{1/\beta}$, $x_{\beta} = (1/\beta)^{1/\beta}$, and 
$
\widetilde{\psi}_{\beta}(x) = \tau_{\beta} x \mathbbm{1}_{\{x < x_{\beta}\}} + e^{x^\beta} \mathbbm{1}_{ \{ x \geq x_{\beta} \}}
$. 

For a generic random variable $Y$ , let $\Pro_{\vert Y}(\cdot)$ and $\Exp_{\vert Y}[\cdot]$ denote the conditional probability and expectation given $Y$, respectively. 
Further, we write ``a.s." for ``almost surely" and ``w.r.t." for ``with respect to".
Throughout the paper, we assume that $r \geq 2$, $d \geq 3$, $n \geq 4$, $p_n := N/|I_{n,r}|\leq 1/2$. 




\section{Gaussian approximations for IOUS}\label{sec:GAR_IOUS}
In this section, we shall derive non-asymptotic Gaussian approximation error bounds for: (i) the IOUS with random kernel $\widehat{U}_{n}$ in (\ref{eqn:ustat_random_kernel}), which includes the IOUS with deterministic kernel $U_{n}$ in (\ref{eqn:ustat}) as a special case, and (ii)  the incomplete IOUS $U'_{n,N}$ in (\ref{eqn:incomplete_ustat_random_kernel}) under the Bernoulli sampling scheme. 

Recall that $h(x_1^r) = \Exp[H(x_1^r,W)]$, $g(x_1) = \Exp[h(x_1,X^r_2)]$, $\theta = \Exp[g(X_1)]$, 
$\sigma_{g,j}^2 = \Exp[(g_j(X_1) - \theta_j)^2]$ and $\underline{\sigma}^2_{g} = \min_{1 \leq j \leq d} \sigma^2_{g,j}$. Further, define
\begin{align*}
&\Gamma_g :=\Cov(g(X_1)),\quad\Gamma_H := \Cov(H(X^r_1,W)),\;\;\\
&\sigma^2_{H,j} := \Exp[(H_j(X_1^r,W) - \theta_j)^2] \text{ for } 1 \leq j \leq d.
\end{align*} 
Clearly, for $1 \leq j \leq d$, $\sigma^2_{H,j} \geq \sigma^2_{g,j}$  and thus $\underline{\sigma}^2_{H} := \min_{1 \leq j \leq d}  \sigma^2_{H,j}  \geq \underline{\sigma}^2_{g}$.
Define two $d \times d$ diagonal matrices $\Lambda_g$ and $\Lambda_H$ such that
\begin{align}
\label{def:Lambda_g_H}
\Lambda_{g,jj} := \sigma^2_{g,j} \leq \sigma^2_{H,j} :=
\Lambda_{H,jj} \;\; \text{ for } 1 \leq j \leq d.
\end{align}

Let $Y_A$ and $Y_B$ be two independent $d$-dimensional zero mean Gaussian random vectors with variance $\Gamma_g$ and $\Gamma_H$ respectively. We may take $Y_{A}$ and $Y_{B}$ to be independent of any other random variables.
Further, for any two zero mean $d$-dimensional random vectors $U$ and $Y$, 
\begin{align*}
\rho(U,Y) := 
\sup_{R \in \cR}\left\vert
\Pro\left(U \in R \right) - \Pro(Y \in R)
\right\vert,
\end{align*}
where we recall that $\cR := \{\prod_{j=1}^d [a_j,b_j]: -\infty \leq a_j \leq b_j \leq \infty\}$ is the collection of hyperrectangles in $\bR^d$.

Finally, in view of the discussions in the Introduction (Section \ref{sec:introduction}) and to simplify presentation, we assume $\underline{\sigma}_g^2 \leq 1$. Otherwise, the conclusions in this paper hold with $\underline{\sigma}_g$ replaced by $\min\{\underline{\sigma}_g,1\}$.



\subsection{IOUS with random kernel}\label{sec:IOUS_random_kernel}
%
We start with $\widehat{U}_n$. Define for $1 \leq j \leq d$, $q >0$, and
$(x_1,\ldots,x_r) \in S^r$,
\begin{equation}
\label{def:exp_norm_W_cond}
B_{n,j}(x_1,\ldots,x_r) := \| H_j(x_1,\ldots,x_r,W) - h_j(x_1,\ldots,x_r)\|_{{\psi}_{q}}.
\end{equation}
We make following assumptions: 
there exist $D_n \geq 1$ and an absolute constant $q > 0$ such that
\begin{align}
\label{non_degeneracy} & \sigma^2_{g,j} > 0,\;\; \text{ for all } j = 1,\ldots,d, \tag{C1-ND} \\
\label{moments_assumption} &\Exp |g_j(X_1) - \theta_j|^{4} \leq \sigma^2_{g,j} D_n^2,\;\;  \text{ for all } j = 1,\ldots, d, \tag{C2}\\
\label{h_exp_assumption}&\|h_j(X_1^r) - \theta_j\|_{{\psi}_{q}} \leq D_n,\;\; \text{ for all } j = 1,\ldots, d, \tag{C3} \\
\label{exp_cond_distr} &
\| B_{n,j}(X_1^r) \|_{{\psi}_q} \leq D_n
 \;\text{ for all } j = 1,\ldots,d. \tag{C4}
\end{align}
Clearly, if $\left\vert H_j(X_1^r,W)\right\vert \lesssim  D_n$ a.s.~for $1 \leq j \leq d$, then the latter three conditions hold.
Indeed, \eqref{h_exp_assumption} and
\eqref{exp_cond_distr} follow immediately from the definition, and \eqref{moments_assumption} is due to the observation that
$
\Exp |g_j(X_1) - \theta_j|^{4} \lesssim \Exp |g_j(X_1) - \theta_j|^{2} D_n^2 =  \sigma^2_{g,j} D_n^2
$.

\begin{theorem}\label{thrm:IOUS_random}
Assume~\eqref{non_degeneracy},~\eqref{moments_assumption},~\eqref{h_exp_assumption} and~\eqref{exp_cond_distr} hold. 
Then 
\begin{equation*} 
\rho(\sqrt{n} (\widehat{U}_n-\theta), r Y_A)
\lesssim \left( \frac{r^2 D_n^2 \log^{q_{*}}(dn)}{\underline{\sigma}_g^2 \, n} \right)^{1/6}, 
\end{equation*}
where $q_{*} := (6/q+1) \vee 7$, $Y_A \; \sim \; N(0, \Gamma_g)$ and $\lesssim$ means up to a multiplicative constant that only depends on $q$. 

\end{theorem}
\begin{proof}
See Section~\ref{proof:IOUS_random}. We highlight that a key step to establish Theorem \ref{thrm:IOUS_random} is to control the expected supremum of the remainder of the H\'ajek projection of the complete IOUS with deterministic kernel (See Theorem~\ref{thm:exp_sup}). Then the Gaussian approximation result for IOUS follows from Gaussian approximation results for sum of independent random vectors~\cite{chernozhukov2017} and anti-concentration inequality~\cite{chernozhukov2015comparison}, by a similar argument in~\cite{chen2017randomized} with proper normalization.
\end{proof}

Clearly, in the special case of non-random kernel, i.e., $H(x_1,\ldots,x_r,W) = h(x_1,\ldots,x_r)$,
~\eqref{exp_cond_distr} trivially holds. Thus we have the following immediate result for the IOUS with deterministic kernel $U_{n}$ in \eqref{eqn:ustat}. 

\begin{corollary}\label{cor:IOUS}
Assume~\eqref{non_degeneracy},~\eqref{moments_assumption} and~\eqref{h_exp_assumption} hold. Then
\begin{equation*}
\rho(\sqrt{n}(U_n-\theta), r Y_A) \lesssim  \left( \frac{r^2 D_n^2 \log^{q_{*}}(dn)}{\underline{\sigma}_g^2\, n} \right)^{1/6}.
\end{equation*}
where 
$q_{*} := (6/q+1) \vee 7$, and 
$\lesssim$ means up to a multiplicative constant that only depends on $q$. 
\end{corollary}

\begin{remark}[Comparisons with existing results for $d=1$]
For the univariate IOUS with non-random kernels, asymptotic normality and its rate of convergence are well understood in literature; see \cite{borovskikh1996} for a survey of results in this direction. In \cite{Zwet1984}, a Berry-Esseen bound is derived for symmetric statistics, which include IOUS (with non-random kernels) as a special case. In particular, applying Corollary 4.1 in \cite{Zwet1984} to IOUS,
the rate of convergence to normality is of order $O(r^{2} n^{-1/2} \sigma_H^2/ \sigma_g^2)$ for a bounded kernel, which implies that asymptotic normality requires (at least) $r=o(n^{1/6})$. A related Berry-Esseen bound is given in \cite{friedrich1989berry}. In both papers, the rates of convergence are suboptimal. For elementary symmetric polynomials (which are $U$-statistics corresponding to the product kernel $h(x_{1}, \dots, x_{r}) = x_{1} \cdots x_{r}$), it is shown in \cite{vanEs1988} that the sharp rate of convergence to normality is of order $O(r n^{-1/2})$, provided that $\E[X_{1}] \neq 0, \Var(X_{1}) \in (0, \infty)$, $\E[|X_{1}|^{3}] < \infty$ and $r = O((\log{n})^{-1}(\log_{2}(n))^{-1} n^{1/2})$. This result implies that asymptotic normality for the IOUS with the product kernel is achieved when $r=O(\log^{-2}(n) n^{1/2})$. If $\underline{\sigma}_g^{-2} = O(r^{2})$, which holds under regularity conditions in Lemma~\ref{lemma:lower_sigma_g}, our Corollary~\ref{cor:IOUS} with $q=1$ implies that the rate of convergence for high-dimensional IOUS is $O((r^{4} \log^{7}(dn) n^{-1})^{1/6})$ (with suitably bounded moments). In particular, Gaussian approximation is asymptotically valid if $\log{d} = O(\log{n})$ and $r=o(n^{1/4-\epsilon})$ for any $\epsilon \in (0, 1/4)$. Even though our result is valid for a smaller range of $r$ and the rate is slower than the optimal rate in the case $d=1$,  Corollary~\ref{cor:IOUS} does allow the dimension to grow sub-exponentially fast in sample size, which is a useful feature for high-dimensional statistical inference. In addition, to the best of our knowledge, the validity of bootstrap procedures proposed in Section~\ref{sec:bootstrap} to approximate the sampling distribution of IOUS (on hyperrectangles in $\R^{d}$) are new in literature. 
\end{remark}



\subsection{Incomplete IOUS with random kernel}\label{sec:incomplete_IOUS_random}
%

Now we consider $U_{n,N}'$, where we recall that $N$ is some given computational budget.
We will assume the following conditions: for $q >0$,
\begin{align}
\label{H_exp_moment_assumption} &\|H_j(X_1^r,W) - \theta_j\|_{{\psi}_{q}} \leq D_n,\; \text{ for all } j = 1,\ldots, d, \tag{C3'} \\
\label{H_moments_assumption} &\Exp |H_j(X_1^{r}) - \theta_j|^{4} \leq \sigma^2_{H,j} D_n^2,\;\;  \text{ for all } j = 1,\ldots, d \tag{C5}.
\end{align}
Clearly,  \eqref{exp_cond_distr} and~\eqref{H_exp_moment_assumption} implies~\eqref{h_exp_assumption} up to a multiplicative constant. Further,~\eqref{H_exp_moment_assumption} and~\eqref{H_moments_assumption} hold if $\left\vert H_j(X_1^r,W) \right\vert \lesssim D_n$ a.s.~for $ 1 \leq j \leq d$.


\begin{theorem}\label{thrm:IOUS_incomplete}
Assume~\eqref{non_degeneracy},~\eqref{moments_assumption},~\eqref{exp_cond_distr},~\eqref{H_exp_moment_assumption} 
and~\eqref{H_moments_assumption}
hold.
Then 
\begin{equation*}
\rho\left( \sqrt{n} ({U}'_{n,N} - \theta),\; 
rY_A +\alpha_n^{1/2}Y_B
\right)
\lesssim  \varpi_n, \text{ where } \varpi_n := \left( \frac{r^{q_1} D_n^2 \log^{q_{*}}(dn)}{\underline{\sigma}_{g}^2\; (n \wedge N)} \right)^{1/6},
\end{equation*}
where  $\alpha_n := n/N$, $q_1 := 2 \vee (2/q)$, $q_{*} := (6/q+1) \vee 7$, $\lesssim$ means up to a multiplicative constant that only depends on $q$, and we recall that $Y_A  \sim  N(0, \Gamma_g)$,
$Y_B \sim  N(0,\Gamma_H)$ and $Y_A,Y_B$ are independent.
\end{theorem}
\begin{proof}
See Section~\ref{proof:incomplete_IOUS_random}.
\end{proof}
\vspace{0.2cm}

\begin{remark}\label{remark:ALL}
If $q \geq 1$, then $q_1 = 2$ and $q_* = 7$. Since $\|\xi\|_{\psi_1} \lesssim \|\xi\|_{\psi_{q}}$ for any random variable $\xi$ and $q\geq 1$, we may assume without loss of generality that $q \leq 1$  in the proof. When $r$ is fixed, $q = 1$, the kernel is deterministic, and there exists some absolute constant $\sigma^2 > 0$ such that 
$\underline{\sigma}_g^2 \geq \sigma^2$, then  the above Theorem recovers Theorem 3.1 from \cite{chen2017randomized}.

Further, by first conditioning on $X_1^n$, we have
\begin{align*}
\Gamma_H &= \Cov  \left( H(X_1^r,W)\right)
\succeq	 \Cov  \left( h(X_1^r)\right) := \Gamma_h,
\end{align*}
where for two square matrices, $A \succeq	 B$ means $A-B$ is positive semi-definite. 
Thus the random kernel $H(\cdot)$ increases the variance of the approximating Gaussian distribution compared to the associated deterministic kernel $h(\cdot)$.
\end{remark}
\vspace{0.2cm}

\section{Bootstrap approximations}\label{sec:bootstrap}
In Section~\ref{sec:incomplete_IOUS_random}, we have seen that the incomplete $U$-statistic with random kernel is approximated by a Gaussian distribution $ N(0, r^2 \Gamma_g + \alpha_n \Gamma_H)$. However, the covariance term is typically unknown in practice. In this section, we will estimate  $\Gamma_g$ and $\Gamma_H$ by bootstrap methods. 

\subsection{Bootstrap for $\Gamma_H$}
Let $\cD_{n} := \{X_1,\ldots,X_n\} \cup \{W_{\iota}, Z_{\iota}: \iota \in I_{n,r}\}$ be the data involved in the definition of $U'_{n,N}$, and take a collection of independent $N(0,1)$ random variables $\{\xi'_{\iota}: \iota \in I_{n,r}\}$ that is independent of the data $\cD_{n}$. Define the following bootstrap distribution:
\begin{equation}
\label{U_bootstrap_B}
U_{n,B}^{\#} := \frac{1}{\sqrt{\widehat{N}}} \sum_{\iota \in I_{n,r}} \xi'_{\iota} \sqrt{Z_{\iota}} \left(
H(X_{\iota},W_{\iota}) - U'_{n,N}
\right).
\end{equation}

The next theorem establishes the validity of $U_{n,B}^{\#}$.

\begin{theorem}\label{thrm:U_B_valid}
Assume the conditions~\eqref{non_degeneracy}~\eqref{moments_assumption},~\eqref{exp_cond_distr}, \eqref{H_exp_moment_assumption} and~\eqref{H_moments_assumption} hold. If
\begin{equation}
\label{U_B_assumption}
\frac{r^{q_1} D_n^2  \log^{q_2}(dn)}{(\underline{\sigma}_H^2 \wedge 1) \;(n \wedge N)} \leq C_1 n^{-\zeta},
\end{equation}
for $q_1 := 2 \vee (2/q)$, $q_{2} := (4/q+1) \vee 5$, some constants $C_1 > 0$ and $\zeta \in (0,1)$, then there exists a constant $C$ depending only on  $q$, $C_1$ and $\zeta$ such that with probability at least $1 - C/n$,
\begin{equation*}
\sup_{R \in \cR}
\left| 
\Pro_{\vert \cD_{n}} \left( U^{\#}_{n,B} \in R \right)
- \Pro(Y_B \in R)
\right| \; \leq \; C n^{-\zeta/6}.
\end{equation*}
\end{theorem}
\begin{proof}
See Section~\ref{proof:U_B}.
\end{proof}

\subsection{Bootstrap for the approximating Gaussian distribution}
Let $S_1 \subset \{1,\ldots,n\}$, and $n_1 = |S_1|$. Further, consider a collection of $\cD_n$-measurable $\bR^d$-valued random vectors $\{G_{i_1}: i_1 \in S_1\}$, where $G_{i_1}$ is some ``good" estimator of $g(X_{i_1})$, and its form is specified later. We use the following quantity to measure the quality of $G_{i_1}$ as an estimator of $g(X_{i_1})$
\begin{align}\label{Delta_A_1}
\widehat{\Delta}_{A,1} :=  \max_{1 \leq j \leq d} \frac{1}{n_1 \sigma_{g,j}^2}\sum_{i_1 \in S_1} \left(
G_{i_1,j} - g_j(X_{i_1})
\right)^2.
\end{align}

Define $\overline{G} := \frac{1}{n_1} \sum_{i_1 \in S_1} G_{i_1}$ and consider the following bootstrap distribution for $N(0,\Gamma_g)$:
\begin{equation}
\label{U_A_bootstrap}
U_{n_1,A}^{\#} := \frac{1}{\sqrt{n_1}} \sum_{i_1 \in S_1} \xi_{i_1}\left(G_{i_1} - \overline{G} \right),
\end{equation}
where $\{\xi_{i_1}: i_1 \in S_1\}$ is a collection of independent $N(0,1)$ random variables that is independent of $\cD_{n}$ and $\{\xi'_\iota: \iota \in I_{n,r}\}$.

\begin{lemma}\label{lemma:appr_Y_A}
Assume the conditions~\eqref{non_degeneracy},~\eqref{moments_assumption} and~\eqref{H_exp_moment_assumption} hold. If
\begin{align}\label{aux_Y_A}
\frac{ D_n^2 \log^{q_2}(dn)}{\underline{\sigma}_g^2\; n_1 }  \leq C_1 n^{-\zeta_1},
\text{ and }
\Pro\left(\widehat{\Delta}_{A,1} \log^4(d) > C_1 n^{-\zeta_2} \right) \leq C_1 n^{-1},
\end{align}
for $q_2 := (4/q+1) \vee 5$, some constants $C_1$, and $\zeta_1,\zeta_2 \in (0,1)$. Then there exists a constant $C$ depending only on $q$, $C_1$  and $\zeta_1$ such that with probability at least $1 - C/n$,
\begin{equation*}
\sup_{R \in \cR}
\left| 
\Pro_{\vert \cD_{n}} \left( U^{\#}_{n_1,A} \in R \right)
- \Pro( Y_A  \in R)
\right| \; \leq \; C n^{-(\zeta_1 \wedge \zeta_2)/6},
\end{equation*}
where we recall that $Y_A \;\sim\; N(0,\Gamma_g)$.
\end{lemma}
\begin{proof}
See Subsection~\ref{proof:appr_Y_A}.
\end{proof}

Hereafter we consider a special case of the divide and conquer bootstrap algorithm in~\cite{chen2017randomized} to estimate $\Gamma_g$.  For each $i_1 \in S_1$, partition the remaining indexes, $\{1,\ldots,n\}\setminus \{i_1\}$, into disjoint subsets $\{S_{2,k}^{(i_1)}: k = 1,\ldots,K\}$, each of size $L = r-1$,
where  $K = \lfloor (n-1)/(r-1)\rfloor$.

Now define for each $i_1 \in S_1$ and $k =1,\ldots,K$, 
\begin{equation*}
\kindex := \{i_1\} \cup S_{2,k}^{(i_1)},\quad
G_{i_1} := \frac{1}{K} \sum_{k=1}^{K} 
H( X_{\kindex}, W_{\kindex}).
\end{equation*}

Finally, define
$$
U_{n,n_1}^{\#} := r U_{n_1,A}^{\#} + \alpha_n^{1/2} U_{n,B}^{\#}.
$$

\begin{theorem}\label{thrm:appr_U_full}
Assume the conditions~\eqref{non_degeneracy}~\eqref{moments_assumption},~\eqref{exp_cond_distr}~\eqref{H_exp_moment_assumption} and~\eqref{H_moments_assumption} hold. If
\begin{equation}
\label{U_full_assumption}
\frac{r^{q_1} D_n^2\log^{q_{*}}(dn)}{\underline{\sigma}_g^2\; (n_1 \wedge N)} 
\leq  C_1 n^{-\zeta},
\end{equation}
for $q_1 := 2 \vee (2/q)$, $q_{*} := (6/q+1) \vee 7$, some constants $C_1 > 0$, $\zeta \in (0,1)$.
For any  $\nu \in \left(\max\{7/6,1/\zeta\}, \infty \right)$,
there exists a constant $C$ depending only on $q$, $\zeta$, $\nu$ and $C_1$ such that with probability at least $1 - C/n$,
\begin{equation*}
\sup_{R \in \cR}
\left| 
\Pro_{\vert \cD_{n}} \left( U^{\#}_{n,n_1} \in R \right)
- \Pro(r Y_A + \alpha_n^{1/2} Y_B \in R)
\right| \; \leq \; C n^{-(\zeta-1/\nu)/6}.
\end{equation*}
\end{theorem}
\begin{proof}
See Subsection~\ref{proof:appr_U_full}.
\end{proof}

\subsection{Simultaneous confidence intervals}\label{sec:spi}

We first combine the Gaussian approximation result  with the bootstrap result.

\begin{corollary}\label{cor:comb_nonnormalize}
Assume~\eqref{non_degeneracy},~\eqref{moments_assumption}~\eqref{exp_cond_distr}~\eqref{H_exp_moment_assumption}
and~\eqref{H_moments_assumption} hold. Further, assume that for some constants $C_1 > 0$, $\zeta \in (0,1)$,
~\eqref{U_full_assumption} holds. Then there exists a constant $C$ depending only on $q$, $C_1$ and $\zeta$ such that with probability at least $1 - C/n$,
\begin{equation*}
\sup_{R \in \cR}
\left| 
\Pro \left( \sqrt{n} \left( U_{n,N}' - \theta \right)  \in R \right)
-
\Pro_{\vert \cD_{n}} \left( U^{\#}_{n,n_1} \in R \right)
\right| \; \leq \; C n^{-\zeta/7}.
\end{equation*}
\end{corollary}
\begin{proof}
It follows from Theorem \ref{thrm:IOUS_incomplete} and
Theorem \ref{thrm:appr_U_full} (with $\nu = 7/\zeta$).
\end{proof}

In simultaneous confidence interval construction, it is sometimes desirable to normalize the variance of each dimension, so that if we use maximum-type statistics, the critical value is not dominated by terms with large variance. 
Define  for $1 \leq j \leq d$,
\[
\widehat{\sigma}^2_{g,j} := \frac{1}{n_1} \sum_{i_1 \in S_1}
\left(G_{i_1,j} - \overline{G}_j \right)^2, \;\;
\widehat{\sigma}^2_{H,j} := \frac{1}{\widehat{N}} \sum_{\iota \in I_{n,r}} Z_{\iota} 
\left(
H_j(X_{\iota},W_{\iota}) - U'_{n,N,j}
\right)^2,
\]
which are the diagonal elements in the conditional covariance matrices of 
$U_{n,A}^{\#}$~\eqref{U_A_bootstrap} and $U_{n,B}^{\#}$~\eqref{U_bootstrap_B} respectively. Further, define a $d \times d$ diagonal matrix $\widehat{\Lambda}$ with
$$
\widehat{\Lambda}_{j,j} = r^2 \widehat{\sigma}^2_{g,j} + \alpha_n \widehat{\sigma}^2_{H,j}, \; \text{ for each }
1 \leq j \leq d.
$$

\begin{corollary}\label{cor:comb_normalize}
Assume the conditions in Corollary~\ref{cor:comb_nonnormalize}. Then there exists a constant $C$ depending only on $q$, $C_1$ and $\zeta$ such that with probability at least $1 - C/n$,
\begin{equation*}
\sup_{R \in \cR}
\left| 
\Pro \left( \sqrt{n} \widehat{\Lambda}^{-1/2}\left( U_{n,N}' - \theta \right)  \in R \right)
-
\Pro_{\vert \cD_{n}} \left( \widehat{\Lambda}^{-1/2} U^{\#}_{n,n_1} \in R \right)
\right| \; \leq \; C n^{-\zeta/7}.
\end{equation*}
Consequently,
\begin{equation*}
\sup_{t > 0}
\left| 
\Pro \left( \| \sqrt{n} \widehat{\Lambda}^{-1/2}(U_{n,N}' - \theta) \|_{\infty} \leq t  \right)
-
\Pro_{\vert \cD_{n}} \left( \|\widehat{\Lambda}^{-1/2} U^{\#}_{n,n_1}\|_{\infty} \leq t\right)
\right| \; \leq \; C n^{-\zeta/7}.
\end{equation*}
\end{corollary}
\begin{proof}
See Subsection~\ref{proof:comb_normalize}.
\end{proof}

\begin{remark}
From Corollary~\ref{cor:comb_normalize}, we can immediately construct confidence intervals for $\theta$ in a data-dependent way. 
Specifically, let $\widehat{q}_{1-\alpha}$ be a $(1-\alpha)^{th}$ quantile of the conditional distribution of 
$\|\widehat{\Lambda}^{-1/2} U^{\#}_{n,n_1}\|_{\infty}$ given $\cD_n$. Then one way to construct simultaneous confidence intervals with confidence level $(1-\alpha)$ is as follows: for $1 \leq j \leq d$,
$U_{n,N,j}' \; \pm \; \widehat{q}_{1-\alpha} \; n^{-1/2}\widehat{\Lambda}^{1/2}_{j,j}$.
\end{remark}

\section{Applications}\label{sec:apps}

In many applications, $g(x) =\Exp[h(x,X_2,\ldots,X_r)]$ does not admit an explicit form, and thus it is usually hard to compute $\underline{\sigma}_g$ in conditions \eqref{non_degeneracy} and \eqref{U_full_assumption} directly. When the kernel $h$ has special structures, we can establish a lower bound on $\underline{\sigma}_g$ with explicit dependence on $r$, which can be applied to Example~\ref{ex:random_forest}. We shall give additional examples in Section~\ref{ex:SimpleEx} and~\ref{exmp:kde} to illustrate the usefulness of $U$-statistics as a tool to estimate and make inference of certain statistical functionals of $X_{1},\dots,X_{r}$. In Section~\ref{ex:SimpleEx} for the expected maximum and log-mean functionals, we also establish a lower bound on $\underline{\sigma}_g$ with explicit dependence on $r$. In Section~\ref{exmp:kde} for the kernel density estimation problem, $r$ is assumed to be fixed, but we allow the diameter of the design points to diverge.

For simplicity of the presentation, in this section, we assume that all involved derivatives and integrals exist and are finite, and that the order of integrals and the order of integral and differentiation can be exchanged. These assumptions can be justified under standard smoothness and moment conditions.
For illustration, we use $q=1$ in~\eqref{exp_cond_distr} and~\eqref{H_exp_moment_assumption}.

\subsection{Lower bound for $\underline{\sigma}_g$}
Suppose that the distribution $P$ of $X_1$ has a density function $f_0$ with respect to some $\sigma$-finite (reference) measure $\mu$, i.e.,  
$$
P(A) = \int_A f_{0}(x) \mu(dx)\;\;\text{ for any } A \in \cS.
$$
We first embed $f_{0}$ into a family of densities $\{f_{\beta}: \beta \in B \subset \bR^{\ell}\}$, where $B$ is an open neighborhood of $0 \in \bR^{\ell}$. Such embeddings always exist and below are some examples for $S = \bR^{\ell}$.
\begin{enumerate}
\item {\it Location and scale family.} 
If $\mu$ is the Lebesgue measure on $\bR^{\ell}$, we may consider the following location or scaling families: for $x \in \bR^{\ell}$,
\begin{align*}
f_{\beta}(x) =& f_0(x - \beta) \text{ with }  \beta \in \bR^{\ell}, \\ 
\text{ or } \quad f_{\beta}(x) =& (1+\beta)f_0((1+\beta)x) \text{ with } \beta \in (-1,1).
\end{align*}

\item {\it Exponential family.} If $\phi(\beta) := \log\left( \int f_0(x) e^{\beta^Tx} \mu(dx)\right)<\infty$ for $\beta \in B$, then we may consider the exponential family: 
\[
f_{\beta}(x) = f_0(x) \exp(\beta^T x - \phi(\beta)), \text{ for } x \in \bR^{\ell}, \beta \in B.
\]

\item {\it Additive noise model.} Let $\Upsilon$ be a $\bR^{\ell}$-dimensional random vector independent of $X_1$, whose distribution is absolutely continuous w.r.t.~$\mu$, then $X_1 + \beta \Upsilon$ has a density $f_{\beta}$ given by the convolution of those of $X_{1}$ and $\beta \Upsilon$. 
\end{enumerate}

For $\beta \in B$, define the following perturbed expectation 
\[
\theta(\beta) := \int h(x_1,\ldots,x_r) \prod_{i=1}^{r} f_{\beta}(x_i) \mu(d x_i) :=
\Exp_{\beta}[h(X_1,\ldots,X_r)],
\]
where $\Exp_{\beta}$ denotes the expectation when $X_1,\ldots,X_r$ have density $f_{\beta}$. Further, define
$$
\Psi(\beta) :=  \sum_{i=1}^{r} \nabla\ln f_{\beta}(X_i),\quad
\cJ(\beta) := r^{-1}\Var_{\beta}\left( \Psi(\beta)\right),
$$
where $\nabla$ denotes the gradient (or derivative when $\beta$ is a scalar) with respect to $\beta$ and  $\Var_{\beta}$ denotes the covariance matrix when $X_1,\ldots,X_r$ have the density $f_{\beta}$. Thus $\Psi(\beta)$ is the score function and $\cJ(\beta)$ is the Fisher-information for a single observation. 


\begin{lemma}\label{lemma:lower_sigma_g}
If we assume  $\cJ(0)$ is positive definite, then
\begin{equation}
\label{eqn:lower_sigma_g}
{\sigma}_{g,j}^2 \; \geq \; r^{-2} (\nabla \theta_j(0))^{T} \cJ^{-1}(0) \nabla \theta_j(0), \;\; \text{ for } 1 \leq j \leq d.
\end{equation}
In particular, if there exists an absolute positive constant $c$ such that
$$
(\nabla \theta_j(0))^{T} \cJ^{-1}(0) \nabla \theta_j(0) \; \geq \; c \;\; \text{ for } 1 \leq j \leq d,
$$
then $\underline{\sigma}_g^{2} \geq cr^{-2}$.
\end{lemma}
\begin{proof}
See Subsection~\ref{proof:lower_sigma_g}.  
\end{proof}

\subsection{Simultaneous prediction intervals for random forests}
Consider the Example \ref{ex:random_forest} and assume that 
$(Y_1,Z_1)$ has density $q(y)p(z;y)$ w.r.t.~the product measure $\nu(dy) \otimes dz$ on $\cY\times \bR$, i.e., for  $A_1 \in \cB(\cY), A_2 \in \cB(\bR)$,
\[
\Pro(Y_1 \in A_1, Z_1 \in A_2) \; = \; \int_{A_1 \times A_2} q(y) p(z;y) \nu(dy) dz.
\] 
That is, the feature $Y_1$ has the density $q(y)$ w.r.t.~some $\sigma$-finite measure $\nu$ on $\cY$, and thus is allowed to have both continuous and discrete components. The response $Z_1$ given $Y_1 = y$ has a conditional density $p(z;y)$ w.r.t.~the Lebesgue measure. 

For many regression algorithms such as tree based methods, if we fix the features and increase the responses of training samples by $\beta \in \bR$, the prediction at any test point will increase by $\beta$, i.e., $\text{ for } 1 \leq j \leq d$,
\[
H_j\left((y_1,z_1+\beta),\ldots,(y_r,z_r+\beta),w\right) = H_j\left((y_1,z_1),\ldots,(y_r,z_r),w\right)  + \beta, 
\]
which implies that $h\left((y_1,z_1+\beta),\ldots,(y_r,z_r+\beta)\right) = h\left((y_1,z_1),\ldots,(y_r,z_r)\right)  + \beta$. Now we consider the embedding into the ``location" family $\{q(y)p(z-\beta;y): \beta \in \bR\}$. Observe that
\[
\theta_j(\beta) = \Exp_{\beta}[h_j(X_1,\ldots,X_r)] = \theta_j(0) + \beta, \text{ for } 1 \leq j \leq d,
\]
which implies that $\theta_j'(0) = 1$. In addition,
\[
\cJ(\beta) = \Var_{\beta}\left(\frac{d}{d\beta} \ln(q(Y)p(Z-\beta;Y)) \right) =
\Exp_{\beta}\left[\left(\frac{\partial_z p(Z-\beta;Y)}{p(Z-\beta;Y)}\right)^2
\right].
\]
Thus if we assume that there exists $c$ such that
\begin{equation}\label{aux:rf_cond}
\cJ(0) = \Exp\left[\left(\frac{\partial_z p(Z;Y)}{p(Z;Y)}\right)^2
\right] \leq c^{-1},
\end{equation}
then (\ref{eqn:lower_sigma_g}) reduces to $\underline{\sigma}_g^2 \geq c r^{-2}$. If further we assume that $H_j(X_1^r,W) \leq C$ a.s.~for some constant $C$ and each $1 \leq j \leq d$ (this holds for example when the response is bounded a.s.),
then the conditions \eqref{moments_assumption}, \eqref{h_exp_assumption}, \eqref{exp_cond_distr}  and  \eqref{H_moments_assumption} hold with $D_n = \ln^{-2}(2) C$. With these assumptions,  the condition~\eqref{U_full_assumption} in Corollary~\ref{cor:comb_normalize} simplifies as 
\[
\frac{r^4 \log^7(dn)}{ n_1 \wedge N} 
\leq  C_1 n^{-\zeta}.
\]
Thus if $r = O(n^{1/4-\epsilon})$ for some $\epsilon > 0$, $\log(d) = O(\log(n))$, and $n = O(n_1 \wedge N)$, 
then Corollary~\ref{cor:comb_normalize} can be used to construct asymptotically valid simultaneous prediction intervals with the error of approximation decaying polynomially fast in $n$.

\begin{remark}[Fisher information in nonparametric regressions]
Let us take a closer look at the condition \eqref{aux:rf_cond}. Consider the nonparametric regression model 
\[
Z_i = \kappa(Y_i) + \epsilon_i, \text{ for } 1 \leq i \leq n,
\]
where $\kappa:\cY \to \bR$ is a deterministic measurable function, and $\epsilon_1,\ldots\epsilon_n$  are i.i.d.~with some density $f$ with respect to the Lebesgue measure. Then $p(z;y) = f(z-\kappa(y))$ and thus
\[
\cJ(0) = \int \left(\frac{f'(z-\kappa(y))}{f(z-\kappa(y))}\right)^2 q(y)f(z-\kappa(y)) \nu(dy) dz 
= \int \frac{\left(f'(z)\right)^2}{f(z)} dz,
\]
where for the last equality, we first perform integration w.r.t.~$dz$ and apply a change-of-variable. Thus $\cJ(0)$ only depends the density of the noise.
\end{remark}

\subsection{Expected maximum and log-mean functionals}

Next we compute the lower bounds on $\underline{\sigma}_{g}^{2}$ for two additional statistical functionals. 

\begin{ex} \label{ex:SimpleEx}
Let $S = \bR^d$ and consider the following two kernels: for $1 \leq j \leq d$,
\[
h_j(x_1,\ldots,x_r) = \max_{1 \leq i \leq r} x_{ij}, \quad
\text{ and } \quad
h_j(x_1,\ldots,x_r) = \log\left(\frac{1}{r}\sum_{i=1}^{r} x_{ij}\right).
\]
In the former case, we are interested in estimating the expectation for the coordinate-wise maxima of $r$ independent random vectors, $\{\Exp[\max_{1\leq i \leq r} X_{ij}]: 1 \leq j \leq d\}$. In the latter, we assume $X_{1j} > 0$ for $1 \leq j \leq d$ and are interested in estimating $\{\Exp[\log(r^{-1}\sum_{i=1}^{r} X_{ij})]: 1 \leq j \leq d\}$. In both cases, the coordinates of $X_1$ can have arbitrary dependence, and we allow $r \to \infty$.
\end{ex}

Consider the first kernel in Example~\ref{ex:SimpleEx}, where $S = \bR^{d}$, and $h_j(x_1,\ldots,x_r)  =  \max_{1 \leq i \leq r} x_{ij}$ for $1 \leq j \leq d$. Assume $X_{1j}$ has a density $f_j$ w.r.t.~the Lebesgue measure on $\bR$ for $1 \leq j \leq d$, and we consider the following embedding $\{f_j(\cdot - \beta): \beta \in \bR\}$. 
As in the previous example,  for $\beta \in \bR$
\[
\theta_j'(\beta) = 1, \quad \cJ_j(\beta) = \Var_{\beta}\left(\frac{d}{d\,\beta} \ln f_j(X_{1j}-\beta) \right)
= \int \frac{(f_j'(x-\beta))^2}{f_j(x - \beta)} d\,x.
\]
Thus, by Lemma~\ref{lemma:lower_sigma_g}, if we assume for some absolute positive constant $c$ 
\[
\int \frac{(f_j'(x))^2}{f_j(x)} d\,x \; \leq \; c^{-1}, \; 1 \leq j \leq d,
\]
we have $\underline{\sigma}_g^{2}  \geq c r^{-2}$. Further, if we assume that there exists a positive constant $C$ such that
\[
\|X_{1j}\|_{\psi_1} \leq C, \quad 1 \leq j \leq d, 
\]
then by maximal inequality (e.g., see~\cite[Lemma 2.2.2]{van1996weak}), $\|\max_{1 \leq i \leq r} X_{ij}\|_{\psi_1} \lesssim \log(r)$. Then if we select $D_n = C'\underline{\sigma}_g^{-1} \log^2(r)$ , the conditions \eqref{moments_assumption}, \eqref{h_exp_assumption} and  \eqref{H_moments_assumption} hold. Further,  \eqref{exp_cond_distr}  trivially holds for non-random kernels.  With above assumptions and selection of $D_n$,  the condition~\eqref{U_full_assumption} in Corollary~\ref{cor:comb_normalize} simplifies as 
$(n_1 \wedge N)^{-1}{r^6 \log^{4}(r)\log^7(dn)}
\leq  C_1 n^{-\zeta}$.\\


Now consider the second kernel in Example~\ref{ex:SimpleEx}, where  $h_j(x_1,\ldots,x_r)= \log\left(r^{-1}\sum_{i=1}^{r}  x_{ij} \right)$ and $X_{1j} > 0$ for $1 \leq j \leq d$. Assume $X_{1j}$ has a density $f_j$ w.r.t.~the Lebesgue measure on $\bR$ for $1 \leq j \leq d$, and  consider the following embedding $\{(1+\beta)f_j((1+\beta)\cdot ): \beta \in (-1,1)\}$. As before, it is easy to see that for $ 1 \leq j \leq d$,
\[
\theta_j'(0) = 1, \quad \text{ and } \quad
\cJ_j(0) = \int \frac{(xf_j'(x) + f_j(x))^2}{f_j(x)} d\,x.
\]
Thus if there exists a constant $c$ such that $\max_{1 \leq j \leq d}\cJ_j(0) \leq c^{-1}$, then $\underline{\sigma}_g^{2} \geq c r^{-2}$. Further, if there exists a constant $C > 0$ such that
$$
\Pro(0 < X_{1j} \leq C) = 1, \; 1 \leq j \leq d,
$$ 
then the conditions \eqref{moments_assumption}, \eqref{h_exp_assumption}, \eqref{exp_cond_distr}  and  \eqref{H_moments_assumption} hold with $D_n = \ln^{-1}(2) \log(C)$. With these assumptions,  the condition~\eqref{U_full_assumption} in Corollary~\ref{cor:comb_normalize} simplifies as 
$(n_1 \wedge N)^{-1}{r^4 \log^7(dn)} 
\leq  C_1 n^{-\zeta}$.

\subsection{Kernel density estimation}


\begin{ex}[Kernel density estimation]\label{exmp:kde} 
Let $\tau : S^r  \to  \bR^\ell$ be a measurable function that is symmetric in its $r$ arguments, and $\{t_j:1\leq j \leq d\} \subset \bR^\ell$ be $d$ design points. 
\cite{frees1994estimating,gine2007local} used $U_n$ 
 as a kernel density estimator (KDE) for the density of $\tau(X_1,\ldots,X_r)$ at the given design points
with
\[
h_j(x_1,\ldots,x_r) = \frac{1}{b_n^{\ell}} \kappa\left( t_j - \tau(x_1,\ldots,x_r) \over b_n \right),\;\; 1 \leq j \leq d,
\]
where $b_n > 0$ is a bandwidth parameter, and $\kappa(\cdot)$ is the density estimation kernel with $\int \kappa(z) dz = 1$, which should not be confused with the $U$-statistic kernel $h$. 
For this example, we will assume $r$ fixed and the bandwidth $b_n \to 0$, but allow the diameter of the design points, $\max_{1 \leq j \leq d} \|t_j\|$, to grow, where $\|\cdot\|$ denotes the usual Euclidean norm. 
\end{ex}

Assume that given $X_1 = x_1$, $\tau(x_1,X_2^r)$ has a density $f(z;x_1)$ w.r.t.~the Lebesgue measure on $\bR^{\ell}$, i.e., $\Pro\left(\,\tau(x_1,X_2,\ldots,X_r) \in A \right) = \int_A f(z;x_1) dz$ for any $A \in \cB(\bR^{\ell})$. 
Then by definition, for $1 \leq j \leq d$,
\[g_j(x_1) = \Exp[h_j(x_1,X_2^r)] = \int \frac{1}{b_n^{\ell}}\, \kappa\left( {t_j - z \over b_n} \right) f(z;x_1)  dz 
= \int \kappa(z) f(t_j - b_n z; x_1) dz.
\]
For $t \in \bR^{\ell}$, denote 
\[
\cV_n(t) := \Var\left(\int \kappa(z) f(t - b_n z; X_1) dz\right),\quad
\cV(t) := \Var(f(t; X_1)).
\]
As in~\cite{frees1994estimating}, if $\int \kappa^2(z) dz < \infty$ and 
$\sup_{t} \Exp[f^2(t; X_1)] < \infty$, then $\lim_{n \to \infty} \cV_n(t) = \cV(t)$ for any fixed $t$.
If there exists some $R > 0$ such that $\max_{1 \leq j \leq d} \|t_j\| \leq R$ for any $d \in \bN$ and $\inf_{t \in \bR^{\ell}: |t| \leq R} \cV(t) > 0$, under mild continuity assumptions (e.g.~the equicontinuty of  $\cV_n(t)$), there exists an absolute constant $c > 0$ such that $\underline{\sigma}_g^2 \geq c$ for large $n$. Then we can apply the result in~\cite{chen2017randomized},
which does not allow $\underline{\sigma}_g^2$ to vanish.

In this work, we allow $\underline{\sigma}_g^2$ to vanish, and thus allow the diameter of the design points to grow as $n$ becomes large. Specifically, if we assume $\kappa(\cdot)$ is bounded by some constant $C$, we can select $D_n = \ln^{-1}(2) C b_n^{-1}$ in conditions \eqref{moments_assumption}, \eqref{h_exp_assumption}, \eqref{exp_cond_distr}  and  \eqref{H_moments_assumption}. Then   the condition~\eqref{U_full_assumption} in Corollary~\ref{cor:comb_normalize} simplifies as 
\[
\frac{ \log^7(dn) }{\underline{\sigma}_g^2 b_n^2 (n_1 \wedge N)} 
\leq  C_1 n^{-\zeta}.
\]
Thus if $\log(d) = O(\log(n))$ and $n = O(n_1 \wedge N)$, to apply Corollary~\ref{cor:comb_normalize}, we require that
$\underline{\sigma}_g^{-2} =  O(b_n^2 n^{1-\epsilon})$ for any $\epsilon > 0$.

\begin{remark}
\cite{frees1994estimating} considers the case  $d = 1$, and shows the $\sqrt{n}$-convergence rate of the KDE. The same discussion applies here. \cite{gine2007local} constructs confidence bands (without computational considerations and bootstrap results) for the density of $\tau(X_1^r)$, under the additional assumptions required to establish the convergence of empirical processes.
\end{remark}


\section{Maximal inequality}\label{sec:maximal_inequality}

In this section, we derive an upper bound on the expected supremum of the remainder of the H\'ajek projection of the complete IOUS with deterministic kernel. This maximal inequality (with the explicit dependence on $r$) serves as a key step to establish the Gaussian approximation result for the incomplete IOUS with random kernel. 

\begin{theorem}\label{thm:exp_sup}
Assume~\eqref{h_exp_assumption} hold. 
Then there exist constants $c,C$, depending only on $q$,  such that
if ${r^2 \log(d)}/{n} \leq c$, then
\begin{equation*} 
\mathbb{E} \left[ \max_{1 \leq j \leq d} \left\vert (U_{n,j} - \theta_j) - \frac{r}{n}\sum_{i=1}^{n}( g_j(X_{i}) -\theta_j) \right\vert   \right]
\leq C \frac{r^2 \log^{1+1/q}(d) D_n}{n}.
\end{equation*}
\end{theorem}

The proof of Theorem \ref{thm:exp_sup} is quite involved: we need to develop a number of technical tools such as the symmetrization inequality and Bonami inequality (i.e., exponential moment bound) for the Rademacher chaos, all with the explicit dependence on $r$. 

We start with some notation. Let $X' := (X_1',\ldots,X_n')$ be an independent copy of $X := (X_1,\ldots,X_n)$, and $\epsilon := (\epsilon_1,\ldots,\epsilon_n)$ be i.i.d.~Rademacher random variables, i.e., $\Pro(\epsilon_1 = 1) = \Pro(\epsilon_1 = -1)=1/2$, that are independent of $X$ and $X'$.
If all involved random variables are independent, we write  $\Exp_{\epsilon}$ (resp. $\Exp_{X'}$) for expectation only w.r.t.~$\epsilon$ (resp. $X'$).

For a given probability space $(X , \mathcal{A}, Q)$,  a measurable function
$f$ on X and $x \in X$, we use the notation $Qf = \int f dQ$ whenever the latter integral is well-defined, and 
denote $\delta_x$ the Dirac measure on $X$, i.e., $\delta_x(A) = \ind{\{x \in A\}}$ for any $A \in \mathcal{A}$
. For a measurable symmetric function $f$ on $S^r$ and $k = 0,1,\ldots, r$, let 
$P^{r-k}f$ denote the function on $S^{k}$ defined by
$$
P^{r-k}f(x_1,\ldots,x_k) := \Exp\left[f(x_1,\ldots,x_k, X_{k+1},\ldots,X_{r})\right],
$$
whenever it is well defined. To prove Theorem \ref{thm:exp_sup}, without loss of generality, we may assume
$$
\theta = P^r h = 0,
$$
since we can always consider $h(\cdot) - \theta$ instead. 
For $0 \leq k \leq r$, define
\begin{equation}\label{def:k_level}
\begin{split}
&\widetilde{\pi}_k h(x_1,\ldots,x_k) := P^{r-k} h,\\
&\pi_k h(x_1,\ldots,x_k) := (\delta_{x_1} - P) \times \cdots \times (\delta_{x_k} - P) \times P^{r-k} h.
\end{split}
\end{equation}
Clearly $\pi_k$ is  degenerate of order $k$ 
with respect to the distribution $P$ in the sense of~\eqref{def:canonical} below. For any $\iota = (i_1,\ldots, i_k) \in I_{n,k}$, and $J =(j_1,\ldots, j_\ell) \in I_{k,\ell}$ where $0 \leq \ell \leq k$, define
$$
\iota_J := (i_{j_1},\ldots, i_{j_{\ell}}) \in I_{n,\ell}.
$$
Then 
$
\pi_k h(x_{\iota}) = \Exp_{X'} \left[
\sum_{\ell=0}^{k} (-1)^{k-\ell}\sum_{J \in I_{k,\ell}} \widetilde{\pi}_{k}h (x_{\iota_J}, X'_{\iota \setminus \iota_{J}})
\right] \text{ for all } \iota \in I_{n,k}.
$

Further, the Hoeffding decomposition \cite{hoeffding1948class} for the $U$-statistic (with $\theta =0$) is as follows:
\begin{equation*}
\begin{split}
   U_n =  \frac{1}{|I_{n,r}|}\sum_{\iota \in I_{n,r}} h(X_{\iota})
&=\sum_{k=1}^{r} {\binom{n}{r}}^{-1}{\binom{n-k}{r-k}}
\sum_{\iota \in I_{n,k}} \pi_k h(X_{\iota}).   
   \\
&=\sum_{k=1}^{r} {\binom{r}{k}}{\binom{n}{k}}^{-1}
\sum_{\iota \in I_{n,k}} \pi_k h(X_{\iota})
=: \sum_{k=1}^{r} \binom{r}{k} U_n^{(k)}(\pi_k h).
\end{split}
\end{equation*}
%
Finally, for any $1 \leq k \leq r$, define the envelope function
$$
F_k(x_1,\ldots,x_k) := \max_{1 \leq j \leq d} \left\vert \widetilde{\pi}_{k} h_j(x_1,\ldots,x_k) \right\vert.
$$

\subsection{Symmetrization inequality}
For each integer $k$, consider a symmetric kernel $f: S^{k} \to \bR^{d}$. We say that $f$ is {\it degenerate} of order $k$ 
with respect to the distribution $P$ if
\begin{equation}\label{def:canonical}
    \Exp_{X_1}[f_j(X_1,X_2,\ldots,X_k)] = 0 \;\; \text{a.s.}, \text{ for any }  1 \leq j \leq d.
\end{equation}

The following result is essentially due to~\cite[Section 3, Symmetrization inequality]{sherman1994maximal} in the $U$-process setting. We provide a self-contained (and perhaps more transparent) proof for completeness.

\begin{theorem}[Symmetrization inequality]\label{thm:sym}
Assume~\eqref{def:canonical} holds.
\begin{equation*}
\Exp\left[ 
\max_{1\leq j \leq d} \left\vert \sum_{\iota \in I_{n,k}} f_j(X_{i_1},\ldots,X_{i_k})\right\vert
\right]
\leq 
2^k
\Exp\left[ 
\max_{1\leq j \leq d} \left\vert \sum_{\iota \in I_{n,k}} \epsilon_{i_1}\cdots \epsilon_{i_k} f_j(X_{i_1},\ldots,X_{i_k})
\right\vert
\right].
\end{equation*}
\end{theorem}

\begin{remark}
In Theorem~\ref{thm:sym}, the symmetrization costs a multiplicative factor of $2^k$ for a degenerate kernel of order $k$. Standard symmetrization argument for such degenerate $U$-statistics (cf. \cite[Theorem 3.5.3]{de2012decoupling}) together with the decoupling inequalities (cf. \cite[Theorem 3.1.1]{de2012decoupling}) in literature yield that 
\[
\Exp\left[ 
\max_{1\leq j \leq d} \left\vert \sum_{\iota \in I_{n,k}} f_j(X_{i_1},\ldots,X_{i_k})\right\vert
\right]
\leq 
C_{k}
\Exp\left[ 
\max_{1\leq j \leq d} \left\vert \sum_{\iota \in I_{n,k}} \epsilon_{i_1}\cdots \epsilon_{i_k} f_j(X_{i_1},\ldots,X_{i_k})
\right\vert
\right], 
\]
where $C_{k} = 2^{4k-2} (k-1)! (k^{k}-1) ((k-1)^{k-1}-1) \times \cdots \times (2^{2}-1)$. Since $2^{k} \ll C_{k}$, improvement of the constant to the exponential growth in $k$ turns out to be crucial to obtain the maximal inequality for the IOUS in Theorem~\ref{thm:exp_sup}. The major component for the super-exponential behavior of $C_{k}$ is due to the step for applying the decoupling inequality in \cite[Theorem 3.1.1]{de2012decoupling}, which is valid for any (measurable) symmetric kernel. If the kernel $f$ is degenerate of order $k$, then symmetrization can be directly done without the decoupling inequality (cf. the proof of Theorem~\ref{thm:sym} below). \end{remark}

\begin{proof}[Proof of Theorem~\ref{thm:sym}]
Define a new sequence of random variables $\{Z_i: 1 \leq i \leq n\}$:
$$
Z_i = X_{i} \mathbbm{1}_{\{\epsilon_i = 1\}} + X'_{i} \mathbbm{1}_{\{\epsilon_i = -1\}}.
$$
Further, for each $\iota =\{i_1,\ldots,i_k\}\in I_{n,k}$, define
$$
\widetilde{f}_{j,\iota} = {2^{k}} \Exp_{\epsilon}\left[
f_j(Z_{i_1},\ldots,Z_{i_k}) \epsilon_{i_1} \cdots \epsilon_{i_k}
\right].
$$
Due to degeneracy, we have
\begin{align*}
\Exp_{X'}\left[
\widetilde{f}_{j,\iota}
\right]
&= {2^{k}} \Exp_{\epsilon}  \Exp_{X'}\left[
f_j(Z_{i_1},\ldots,Z_{i_k}) \epsilon_{i_1} \cdots \epsilon_{i_k}
\right] \\
&= {2^{k}} \Exp_{\epsilon}  \left[
f_j(X_{i_1},\ldots,X_{i_k}) 
\mathbbm{1}_{\{
\epsilon_{i_1} =1, \ldots, \epsilon_{i_k} = 1 \}}
\right] \\
&=f_j(X_{i_1},\ldots, X_{i_k}),
\end{align*}
where the first and third equalities follow from definitions and Fubini Theorem, and the second follows from the degeneracy. To wit, on the event that $\{\epsilon_{i_\ell} = -1\}$ for some $1 \leq \ell \leq k$,  
$$
\Exp_{X'_{i_\ell}}\left[
f_j(Z_{i_1},\ldots, Z_{i_{\ell-1}}, X'_{i_\ell}, Z_{i_{\ell+1}},\ldots,Z_{i_k}) \epsilon_{i_1} \cdots \epsilon_{i_k}
\right] = 0.
$$
The rest of the argument is standard: by Jensen's inequality,
\begin{align*}
\max_{1\leq j \leq d} \left\vert \sum_{\iota \in I_{n,k}} f_j(X_{i_1},\ldots,X_{i_k})\right\vert
&=
\max_{1\leq j \leq d} \left\vert \sum_{\iota \in I_{n,k}} \Exp_{X'}\left[
\widetilde{f}_{j,\iota}
\right]\right\vert \\
&\leq 
{2^{k}} \Exp_{\epsilon,X'}  \max_{1\leq j \leq d} \left\vert \sum_{\iota \in I_{n,k}}
f_j(Z_{i_1},\ldots,Z_{i_k}) \epsilon_{i_1} \cdots \epsilon_{i_k}
\right\vert.
\end{align*}
Since $(X_1,\ldots,X_n, \epsilon_1,\ldots,\epsilon_n)$  and 
$(Z_1,\ldots,Z_n, \epsilon_1,\ldots,\epsilon_n)$ have the same distribution, 
taking expectation on both sides
completes the proof.
\end{proof}

\subsection{Maximal inequality}

We start with a lemma, whose proof is elementary and thus omitted.
Recall the definition of $\widetilde{\psi}_{\beta}$ in Subsection~\ref{subsec:notation}.

\begin{lemma}\label{psi_prop}
For any $\beta > 0$, $\widetilde{\psi}_{\beta}(\cdot)$ is strictly increasing, convex, and $\widetilde{\psi}_{\beta}(0) = 0$. Further, for any $\beta > 0$,
$$
\widetilde{\psi}_{\beta}(x) \leq e^{x^\beta} \leq \widetilde{\psi}_{\beta}(x) + e^{1/\beta},
$$
and consequently 
$$
\widetilde{\psi}_{\beta}^{-1}(m) \leq \log^{1/\beta}\left(
m + e^{1/\beta}
\right).
$$
\end{lemma}

Now we state the maximal inequality with explicit constants.

\begin{lemma}\label{max_inequ}
Fix $\beta \in (0,1]$. Consider a sequence of \textit{non-negative} random variables $\{Z_{j}: 1\leq j \leq d\}$,
and assume that there exists some real number $\Delta > 0$ such that
$\Exp [ \widetilde{\psi}_{\beta}\left({Z_j}/{\Delta} \right) ] \leq 2, \text{ for } 1 \leq j \leq d$.
Then 
\begin{equation*}
\Exp \left[ \max_{1 \leq j \leq d} Z_j \right] \leq \Delta \log^{1/\beta}(2d + e^{1/\beta}).
\end{equation*}
\end{lemma}
\begin{proof}
By monotonicity and convexity,
\begin{align*}
\widetilde{\psi}_{\beta} \left(\Exp \left[ \max_{1 \leq j \leq d} (Z_j/{\Delta}) \right] \right)
&\leq \Exp\left[
\widetilde{\psi}_{\beta} \left( \Delta^{-1}\max_{1 \leq j \leq d} Z_j \right)
\right] \\
& = \Exp\left[
\max_{1 \leq j \leq d}  \widetilde{\psi}_{\beta}( Z_j/\Delta) 
\right]  \leq \sum_{1 \leq j \leq d} \Exp\left[
 \widetilde{\psi}_{\beta}( Z_j/\Delta) 
\right] = 2d. 
\end{align*}
Then the proof is complete by Lemma~\ref{psi_prop}.
\end{proof}

\subsection{Exponential moment of Rademacher chaos}
The goal is to establish an exponential moment bound (i.e., Bonami inequality) of Rademacher chaos of order $k$. Based on the well-known hyper-contractivity of Rademacher chaos variables in literature (cf. \cite[Corollary 3.2.6]{de2012decoupling}), our Lemma~\ref{exp_rade} below provides an exponential moment bound with an explicit dependence on the order. 

\begin{lemma}[Exponential moment of Rademacher chaos]
\label{exp_rade}
Fix $k \geq 2$, $\beta = 2/k$ and  let $\{x_{\iota}: \iota \in I_{n,k}\}$ be a collection of real numbers. Consider the following homogeneous chaos of order $k$:
$$
Z = \sum_{\iota \in I_{n,k}} x_{\iota} \epsilon_{i_1}\cdots \epsilon_{i_k}, 
$$
where $\epsilon_1,\ldots,\epsilon_n$ are i.i.d.~Rademacher random variables. Then
$$
\Exp  \left[ \widetilde{\psi}_{\beta}\left({|Z|}/{\Delta_n}\right) \right] \leq 2, \;\text{ where }
\Delta_n = 7^{k/2} \sqrt{\sum_{\iota \in I_{n,k}} x_{\iota}^2}.
$$
\end{lemma}

\begin{proof}
Denote $\kappa = \sqrt{\Exp[Z^2]}$, $c = \sqrt{7}$ and thus $\Delta_n = c^{k} \kappa$.
Observe that $\beta \leq 1$ and $\beta k = 2$. From~\cite[Theorem 3.2.2]{de2012decoupling}, we have for any $q > 0$
$$
\Exp |Z|^q  \leq \left( q^{q/\beta} \bigvee 1\right) \kappa^q
\leq (q^{q/\beta} + 1) \kappa^q
. 
$$
Here, the first inequality clearly holds for $q \leq 2$, and we use \cite[Theorem 3.2.2]{de2012decoupling} for $q > 2$.
Then using the fact that $e^x \leq 1 + \sum_{\ell = 1}^{\infty} |x|^{\ell}/{\ell !}$ and by Lemma~\ref{psi_prop}, we have
\begin{align*}
\Exp \widetilde{\psi}_{\beta}\left({|Z|}/{\Delta_n}\right) 
&\leq \Exp \exp\left((|Z|/\Delta_n)^\beta \right) 
\leq 1 + \sum_{\ell = 1}^{\infty} \Exp|Z|^{\beta \ell}/{(\ell ! \Delta_n^{\beta \ell})} \\
&\leq 
\sum_{\ell=1}^{\infty} \frac{(\beta \ell)^{\ell} \kappa^{\beta \ell}}{\ell ! \Delta_n^{\beta \ell}} 
+
\sum_{\ell=0}^{\infty} \frac{ \kappa^{\beta \ell}}{\ell ! \Delta_n^{\beta \ell}}  
= \sum_{\ell=1}^{\infty} \frac{\beta^{\ell} \ell^{\ell} }{\ell ! c^{2 \ell}} 
+
\sum_{\ell=0}^{\infty} \frac{1}{\ell ! c^{2 \ell}}.
\end{align*}
Using the fact that $\ell^{\ell} \leq e^{\ell} \ell !$, we have
\begin{align*}
\Exp \widetilde{\psi}_{\beta}\left({|Z|}/{\Delta_n}\right) 
& \leq \sum_{\ell=1}^{\infty} \left(\frac{\beta e }{ c^{2}} \right)^{\ell}  
+
\sum_{\ell=0}^{\infty} \frac{1}{\ell ! c^{2 \ell}} 
\leq \sum_{\ell=1}^{\infty} \left( \frac{e}{ c^{2}} \right)^{\ell}  
+
\sum_{\ell=0}^{\infty} \frac{1}{\ell ! c^{2 \ell}}. 
\end{align*}
Since $c^2  = 7 > e$, we have
$$
\Exp \widetilde{\psi}_{\beta}\left({|Z|}/{\Delta_n}\right) 
\leq \frac{e}{c^2 - e} + e^{c^{-2}} < 2,
$$
which completes the proof.

\end{proof}

\subsection{Proof of Theorem~\ref{thm:exp_sup}}
Now we are in position to prove Theorem~\ref{thm:exp_sup}. Recall that we assume $\theta = 0$. First, for each $2 \leq k \leq r$ and $1 \leq j \leq d$, define
$$
Z_{k,j} = 
\Exp_{\epsilon}\left[ 
\left\vert \sum_{\iota \in I_{n,k}} \epsilon_{i_1}\cdots \epsilon_{i_k} \pi_k h_{j}(X_{i_1},\ldots,X_{i_k})
\right\vert
\right],
$$
where $\pi_k h$ is defined in~\eqref{def:k_level},
and $\epsilon_1,\ldots,\epsilon_n$ are i.i.d.~Rademacher random variables. 
Define
\begin{align*}
\Delta_{k,j}^2
&= \sum_{\iota \in I_{n,k}} \left(\pi_k h_{j}(X_\iota) \right)^2
 = 
\sum_{\iota \in I_{n,k} }
\left(
\Exp_{X'} \left[ \sum_{\ell=0}^{k} (-1)^{k-\ell}\sum_{J \in I_{k,\ell}} \widetilde{\pi}_{k} h_j(X_{\iota_J}, X'_{\iota \setminus \iota_{J}})\right]
\right)^2.
\end{align*} 
By Jensen's inequality and the fact that $(\sum_{i=1}^{n} z_n)^2 \leq n \sum_{i=1}^{n} z_n^2$, we have for any $1 \leq j \leq d$,
\begin{align*}
\Delta_{k,j}^2 &\leq 2^{k} \Exp_{X'} \left[ \sum_{\iota \in I_{n,k}}  
\sum_{\ell=0}^{k} \sum_{J \in I_{k,\ell}} \left( \widetilde{\pi}_{k} h_j(X_{\iota_J}, X'_{\iota \setminus \iota_{J}}) \right)^2\right] \\
&\leq 2^{k} \Exp_{X'} \left[ \sum_{\iota \in I_{n,k}}  
\sum_{\ell=0}^{k} \sum_{J \in I_{k,\ell}} F_{k}^2(X_{\iota_J}, X'_{\iota \setminus \iota_{J}})\right].
\end{align*}
Then by Lemma~\ref{exp_rade}, 
$$
\Exp_{\epsilon}\left[\widetilde{\psi}_{2/k}\left(
\frac{|Z_{k,j}|}{7^{k/2} \Delta_{k,j}}
\right) \right]
\leq  2.
$$
Further, by Lemma~\ref{max_inequ} with $\beta = 2/k$, we have
\begin{align*}
&\Exp_{\epsilon} \max_{1 \leq j \leq d} 
\left\vert \sum_{\iota \in I_{n,k}} \epsilon_{i_1}\cdots \epsilon_{i_k} \pi_k h_j(X_{i_1},\ldots,X_{i_k})
\right\vert
\leq 7^{k/2} \max_{1 \leq j \leq d}(\Delta_{k,j}) \log^{k/2}(2d + e^{k/2}) \\
&\leq 
 14^{k/2} \log^{k/2}(2d + e^{k/2})
\sqrt{ 
\Exp_{X'} \left[ \sum_{\iota \in I_{n,k}}  
\sum_{\ell=0}^{k} \sum_{J \in I_{k,\ell}} F_{k}^2(X_{\iota_J}, X'_{\iota \setminus \iota_{J}})\right]
}.
\end{align*}
Then by Lemma~\ref{thm:sym} and Jensen's inequality, we have
\begin{align*}
\Exp\left[ 
\max_{1\leq j \leq d} \left\vert \sum_{\iota \in I_{n,k}} \pi_k h_{j}(X_{\iota})\right\vert
\right]
&\leq 56^{k/2}  \log^{k/2}(2d + e^{k/2})
\Exp\sqrt{ 
\Exp_{X'} \left[ \sum_{\iota \in I_{n,k}}  
\sum_{\ell=0}^{k} \sum_{J \in I_{k,\ell}} F_{k}^2(X_{\iota_J}, X'_{\iota \setminus \iota_{J}})\right]
} \\
&
\leq 56^{k/2}  \log^{k/2}(2d + e^{k/2})
\sqrt{ 
\Exp \left[ \sum_{\iota \in I_{n,k}}  
\sum_{\ell=0}^{k} \sum_{J \in I_{k,\ell}} F_{k}^2(X_{\iota_J}, X'_{\iota \setminus \iota_{J}})\right]
} 
\\
& = 56^{k/2}  \log^{k/2}(2d + e^{k/2})
\sqrt{ \binom{n}{k} 2^k \Exp[ F_k^2(X_1,\ldots,X_k)].
} 
\end{align*}
Now we bound $\Exp[ F_k^2(X_1,\ldots,X_k)]$. By the definition of $ \widetilde{\pi}_{k} h_j$, condition~\eqref{h_exp_assumption}, Lemma \ref{psi_prop} and Jensen's inequality, we have
\begin{align*}
&\Exp\left[\widetilde{\psi}_{q} \left( |\widetilde{\pi}_{k} h_j(X_1,\ldots,X_k)| / D_n \right) \right]\\
=\;& \Exp\left[ \widetilde{\psi}_{q}(|\Exp_{X'}[h_{j}(X_1,\ldots, X_k, X'_{k+1}, \ldots, X'_{r})]|/ D_n) \right] \\
\leq\;& \Exp\left[ \widetilde{\psi}_{q}(|h_{j}(X_1,\ldots, X_k, X'_{k+1}, \ldots, X'_{r})|/ D_n) \right] \\
\leq\;& \Exp\left[ {\psi}_{q}(|h_{j}(X_1,\ldots, X_k, X_{k+1}, \ldots, X_{r})|/ D_n) \right] + 1 \leq 2.
\end{align*}
Since $\widetilde{\psi}_q(0) = 0$, by Jensen's inequality, we have
$\|\widetilde{\pi}_{k} h_j(X_1,\ldots,X_k)|\|_{\widetilde{\psi}_q} \leq 2D_n$.
Then by the standard maximal inequality (e.g., see~\cite[Lemma 2.2.2]{van1996weak}), there exists a constant $C$, depending only on $q$, such that for $1 \leq k \leq r$, 
$$
\sqrt{ \Exp|F_k(X_1,\ldots, X_k)|^2 } \leq C \log^{1/q}(d) D_{n}.
$$
Thus we obtain that 
\begin{align*}
\Exp \left[ \max_{1 \leq j \leq d}  \left\vert U_{n,j} - \frac{r}{n}\sum_{i=1}^{n} g_j(X_{i}) \right\vert  \right]
&\leq  \sum_{k=2}^{r}\binom{r}{k}\Exp\left[
\max_{1 \leq j \leq d} \left\vert
U_{n}^{(k)}(\pi_k h_{j}) \right\vert
\right]      \\
&\leq
\sum_{k=2}^{r}
\frac{\binom{r}{k}}{\sqrt{\binom{n}{k}}}(112)^{k/2}  \log^{k/2}(2d + e^{k/2})
\sqrt{ 
\Exp F_{k}^2(X_1,\ldots,X_k)
}
\\
& \leq C
\log^{1/q}(d) D_n \sum_{k=2}^{r}
\frac{\binom{r}{k}}{\sqrt{\binom{n}{k}}}(112)^{k/2}  \log^{k/2}(2d + e^{k/2}).
\end{align*}
Observe that if $r^2\leq n$, we have for any $1 \leq i \leq r$
$$
\frac{r-i}{\sqrt{n-i}} \leq \frac{r}{\sqrt{n}} \quad
\Rightarrow \quad
\frac{\binom{r}{k}}{\sqrt{\binom{n}{k}}} \leq \frac{1}{\sqrt{k!}}\left(\frac{r^2}{n} \right)^{k/2}.
$$
Further, for any $x,y \geq 2$, $\log^{k/2}(x+y) \leq 2^{k/2}(\log^{k/2}(x) + \log^{k/2}(y))$. 
Now, take $c = 1/500$, and in particular $r^2 \leq n$. Then
\begin{align*}
\Exp \left[ \max_{1 \leq j \leq d}  \left\vert U_{n,j} - \frac{r}{n}\sum_{i=1}^{n} g_j(X_{i}) \right\vert   \right]
\leq C
\log^{1/q}(d) D_n 
\sum_{k=2}^{r} 
\left(224
\frac{r^2}{n}
\right)^{k/2}
(\log^{k/2}(2d) + \frac{1}{\sqrt{k!}} (k/2)^{k/2}).
\end{align*}
For the first term, by geometric series formula,
$$
I  = C \log^{1/q}(d) D_n 
\sum_{k=2}^{r} 
\left(224
\frac{r^2 \log(2d)}{n}
\right)^{k/2}
\leq C
\frac{r^2 \log^{1+1/q}(d) D_n}{n}.
$$
For the second term, since for any $\ell \geq 1$, $\ell^{\ell} \leq e^{\ell} \ell!$, we have
\begin{align*}
II = C \log^{1/q}(d) D_n 
\sum_{k=2}^{r} 
\left(112e
\frac{r^2 }{n}
\right)^{k/2}
\leq C \frac{r^2 \log^{1/q}(d) D_n}{n},
\end{align*}
which completes the proof of Theorem~\ref{thm:exp_sup}. 
\qed


\appendix

\section{Proofs}\label{sec:proofs}

\subsection{Tail probabilities}\label{subsec:tail_probabilities}
In this section, we collect and prove some results regarding tail probabilities for sum of independent random vectors, $U$-statistics, and $U$-statistics with random kernels. For each type of statistics, we present two versions, one for non-negative random variables and the other for general cases.

These inequalities are used in bounding the effects due to sampling (Subsection \ref{subsec:effect of sampling}), and also in controlling the $\|\cdot\|_{\infty}$ distance between the bootstrap covariance matrices  and their targets (Section \ref{sec:proof_bootstrap}).

\subsubsection{Tail probabilities for sum of independent random vectors}
In this subsection, $m,n,d \geq 2$ are all integers.

\begin{lemma}\label{tail_nonneg_sum}
Let $Z_1,\ldots,Z_m$ be independent $\bR^d$-valued random vectors and $\beta \in (0,1]$.
Assume that
$$
Z_{ij} \geq 0,\;\;
\|Z_{ij}\|_{\psi_{\beta}} \leq u_n, \text{ for all } i = 1,\ldots,m, \text{ and } j = 1,\ldots, d. 
$$
Then there exists some constant $C$ that only depends on $\beta$ such that
\begin{align*}
\Pro\left(
\max_{1\leq j \leq  d} \sum_{i=1}^{m} Z_{ij}
\geq C\left(
\max_{1\leq j \leq d} \Exp\left[\sum_{i=1}^{m} Z_{ij} \right]
+  u_n \log^{1/\beta}(dm)\left(\log(dm)+ \log^{1/\beta}(n)\right)
\right)
\right) 
\leq 3/n.
\end{align*} 
\end{lemma}
\begin{proof}
See Subsection~\ref{proof:tail_nonneg_sum}.
\end{proof}
\vspace{0.2cm}
\begin{lemma}\label{tail_general_sum_ind}
Let $Z_1,\ldots,Z_m$ be independent $\bR^d$-valued random vectors and $\beta \in (0,1]$.
Assume that
$$
\Exp[Z_{ij}] = 0,\;\;
\|Z_{ij}\|_{\psi_{\beta}} \leq u_n, \text{ for all } i = 1,\ldots,m, \text{ and } j = 1,\ldots, d. 
$$
Then there exists some constant $C$ that only depends on $\beta$ such that
\begin{align*}
\Pro\left(
\max_{1\leq j \leq  d} 
\left\vert
\sum_{i=1}^{m} Z_{ij}
\right\vert
\geq C\left( \sigma \log^{1/2}(dn)
+  u_n \log^{1/\beta}(dm)\left( \log(dm) + \log^{1/\beta}(n) \right)
\right)
\right) 
\leq 4/n,
\end{align*}
where $\sigma^2 := \max_{1 \leq j \leq d} \sum_{i=1}^{m} \Exp[Z_{ij}^2]$. 
\end{lemma}
\begin{proof}
See Subsection~\ref{proof:tail_general_sum_ind}
\end{proof}

\vspace{0.2cm}
\begin{lemma}\label{tail_Bernoulli}
Let $Z_1,\ldots,Z_m$ be independent and identical distributed Bernoulli random variables with success probability $p_n$, i.e., $\Pro(Z_i = 1) = 1 - \Pro(Z_i = 0) = p_n$  for $1 \leq i \leq m$. Further, let 
$a_1,\ldots,a_m$ be deterministic $\bR^d$ vectors. Then there exists an absolute constant $C$  such that
\begin{align*}
\Pro\left(
\max_{1\leq j \leq  d} 
\left\vert
\sum_{i=1}^{m} (Z_{i} - p_n) a_{ij}
\right\vert
\geq C\left( \sqrt{p_n(1-p_n)} \sigma \log^{1/2}(dn)
+  M \log(dn)
\right)
\right) 
\leq 4/n,
\end{align*}
where $\sigma^2 :=  \max_{1 \leq j \leq d} \sum_{i=1}^{m} a_{ij}^2$ and $M = \max_{1 \leq i \leq m, 1\leq j \leq d} |a_{ij}|$. 
\end{lemma}
\begin{proof}
See Subsection~\ref{proof:tail_Bernoulli}
\end{proof}

\vspace{0.2cm}
\subsubsection{Tail probabilities for $U$-statistics}
\begin{lemma}\label{tail U-non-negative}
Let $X_1,\ldots,X_n$ be i.i.d.~random variables  taking value in $(S,\cS)$ and fix $\beta \in (0,1]$.
Let $f : (S^r,\cS^r) \to \bR^d$ be a measurable, symmetric function such that 
$\text{ for all } j = 1,\ldots,d$,
$$
f_j(X_1,\ldots,X_r) \geq 0 \; \text{ a.s.}, \quad
\Exp[f_j(X_1,\ldots,X_r)] \leq v_n,
\quad \|f_j(X_1,\ldots,X_r)\|_{\psi_{\beta}} \leq u_n. 
$$
Define $U_n := |I_{n,r}|^{-1}\sum_{\iota \in I_{n,r}} f(X_{\iota})$.
Then there exists a constant $C$ that only depends on $\beta$ such that 
$$
\Pro\left( \max_{1\leq j\leq d} U_{n,j} \geq C \left(v_n
+ n^{-1} r  \log^{1/\beta+1}(dn)\log^{1/\beta-1}(n) u_n\right)
\right)\leq \frac{3}{n}.
$$
Clearly, we can replace $v_n$ by $u_n$.
\end{lemma}
\begin{proof}
See Subsection~\ref{proof:tail_nonneg}.
\end{proof}
\vspace{0.2cm}

\begin{lemma}\label{tail general U}
Let $X_1,\ldots,X_n$ be i.i.d.~random variables  taking value in $(S,\cS)$ and fix $\beta \in (0,1]$.
Let $f : (S^r,\cS^r) \to \bR^d$ be a measurable, symmetric function such that 
$$
\Exp \left[ f_j(X_1,\ldots,X_r) \right] = 0, \quad \|f_j(X_1,\ldots,X_r)\|_{\psi_{\beta}} \leq u_n \text{ for all } j = 1,\ldots,d.
$$
Define $U_n := |I_{n,r}|^{-1}\sum_{\iota \in I_{n,r}} f(X_{\iota})$
and $\sigma^2 := \max_{1 \leq j \leq d} \Exp[f_j^2(X_1^r)]$.
Then there exists a constant $C$ that only depends on $\beta$ such that 
$$
\Pro\left( \max_{1\leq j\leq d} |U_{n,j}| \geq C \left( n^{-1/2} r^{1/2} \log^{1/2}(dn)\sigma
+ n^{-1} r \log^{1/\beta+1}(dn)\log^{1/\beta-1}(n) u_n \right)
\right)\leq \frac{4}{n}.
$$
Clearly, we can replace $\sigma$ by $u_n$.
\end{lemma}
\begin{proof}
See subsection~\ref{proof:tail_general}.
\end{proof}
\vspace{0.2cm}

\subsubsection{Tail probabilities for $U$-statistics with random kernel}
Let $X_1,\ldots,X_n$ be i.i.d.~random variables  taking value in $(S,\cS)$ and 
$W, \{W_{\iota}, \iota \in I_{n,r}\}$  be i.i.d.~random variables  taking value in $(S',\cS')$, that are independent of $X_1^{n}$. In this subsection, we consider a measurable function $F: S^r \times S' \to \bR^{d}$ that is symmetric in the first $r$ variables, and  fix some $\beta \in (0,1]$. Further, define
\begin{align*}
&f(x_1,\ldots,x_r) := \Exp[F(x_1,\ldots,x_r, W)], \\
&b_j(x_1,\ldots,x_r) := \|F_j(x_1,\ldots,x_r,W) - f_j(x_1,\ldots,x_r) \|_{\psi_{\beta}} \; \text{ for all } j = 1,\ldots,d.
\end{align*}


We first consider the non-negative random kernels.

\begin{lemma}\label{lemma:nonneg_tail_random_kernel}
Consider
$
Z := \max_{1\leq j \leq d} {|I_{n,r}|}^{-1}\sum_{\iota \in I_{n,r}}  F_j(X_{\iota},W_{\iota}) .
$
Assume  that for all $j = 1,\ldots, d$, $F_{j}(\cdot) \geq 0$, and
that there exists $u_n \geq 1$ such that
$$
\|b_j(X_1^r)\|_{\psi_{\beta}} \leq  u_n,\quad
\|f_j(X_1^r)\|_{\psi_{\beta}} \leq u_n, \text{ for all } j = 1,\ldots,d.
$$ 
Then there exists some constant $C$ that only depends on $\beta$ such that
with probability at least $1 - 8/n$, 
\begin{align*}
Z \; \leq \;&C  \max_{1\leq j \leq d} \Exp\left[f_j(X_1^r)\right]
+  C n^{-1}r\log^{1/\beta+1}(dn)\log^{1/\beta-1}(n) u_n \\
&+  C |I_{n,r}|^{-1} r^{3/\beta} \log^{2/\beta+1}(dn) \log^{2/\beta-1}(n) u_n.
\end{align*}
\end{lemma}
\begin{proof}
See subsection~\ref{proof:nonneg_tail_random_kernel}.
\end{proof}
\vspace{0.2cm}

Next, we consider centered random kernels.
\begin{lemma}\label{lemma:tail_random_kernel}
Consider
$
Z := \max_{1\leq j \leq d} \left \vert {|I_{n,r}|}^{-1}\sum_{\iota \in I_{n,r}} \left( F_j(X_{\iota},W_{\iota}) -f_j(X_{\iota}) \right) \right\vert.
$
 Assume there exists $u_n \geq 1$ such that for all $j = 1,\ldots, d$,
\begin{align*}
\|b_j(X_1,\ldots,X_r)\|_{\psi_{\beta}} \leq u_n.
\end{align*}
Then there exists some constant $C$ that only depends on $\beta$ such that
with probability at least $1 - 9/n$, 
\begin{align*}
Z \; \leq \;&C u_n
|I_{n,r}|^{-1/2} r^{1/2} \log^{1/2}(dn)\left(1 +
n^{-1/2} r^{1/2} \log^{1/\beta+1/2}(dn) \log^{1/\beta-1/2}(n)\right) \\
&+ C u_n |I_{n,r}|^{-1} r^{3/\beta} \log^{2/\beta+1}(dn)\log^{2/\beta-1}(n).
\end{align*}
\end{lemma}
\begin{proof}
See subsection~\ref{proof:tail_random_kernel}.
\end{proof}
\vspace{0.2cm}

\subsection{Additional lemmas}
The following Lemma concerns Gaussian approximation for sum of independent vectors. It replaces the $\|\cdot\|_{\psi_1}$ condition in Proposition 2.1 of \cite{chernozhukov2017}  by $\|\cdot\|_{\psi_q}$. 

\begin{lemma}\label{lemma:GAR_q}
Let $Z_1,\ldots,Z_n$ be independent $\bR^d$-valued random vectors. Assume that for some absolute constant $\underline{\sigma}^2 > 0$, and $q > 0$,
\begin{align*}
&n^{-1}\sum_{i=1}^{n} \Exp\left[ Z_{ij}^2 \right] \geq \underline{\sigma}^2,\quad
n^{-1}\sum_{i=1}^{n} \Exp\left[ Z_{ij}^{2+k} \right] \leq D_n^k\; \text{ for } j = 1,\ldots,d,\;\; k = 1,2,\\
&\|Z_{ij}\|_{\psi_q} \leq D_n,\;\; \text{ for } i=1,\ldots,n, \;\; j = 1,\ldots,d.
\end{align*}
Then  there exists some constant $C$ that only depends on $q$  and $\underline{\sigma}^2$ such that
\begin{align*}
\rho(n^{-1/2}\sum_{i=1}^{n}\left(Z_{i} - \Exp[Z_{i}]\right),\; Y) \leq C 
\left( 
\frac{D_n^2 \log^{q_{*}}(dn)}{n}
\right)^{1/6},
\end{align*} 
where $q_{*} = (6/q+1) \vee 7$, $Y  \sim  N\left(0, \Sigma\right)$, and $\Sigma := n^{-1}\sum_{i=1}^{n} \Exp[Z_i Z_i']$.
\end{lemma}
\begin{proof}
See Subsection \ref{additional_proofs}.
\end{proof}

The following lemmas are elementary, but used repeatedly.

\begin{lemma}\label{Inr_large}
Let $\beta > 0$. There exits a constant $C$, only depending on $\beta$, such that 
for any positive integers $r,n$ such that $2 \leq r \leq \sqrt{n}$,
\[
n^2 r^{\beta} \leq C \|I_{n,r}\|.
\]
\end{lemma}
\begin{proof}
Fix $\beta$. If $r \to \infty$, $n^2 r^{\beta}/\|I_{n,r}\| \to 0$. Thus there exits $M$ such that if $r \geq M$, $n^2 r^{\beta} \leq \|I_{n,r}\|$. For $r < M$, the inequality holds with $C = M^{\beta}$. 
\end{proof}

\begin{lemma}\label{psi_norms}
Let $\beta, k > 0$. For any random variable $X$,
\[
\|X^k\|_{\psi_{\beta}} = \|X\|^{k}_{\psi_{k\beta}}.
\]
\end{lemma}
\begin{proof}
Observe that
\[
\Exp\left[\exp\left(|X|^k/\|X\|^{k}_{\psi_{k\beta}} \right)^{\beta} \right]
= \Exp\left[\exp\left(|X|/\|X\|_{\psi_{k\beta}} \right)^{k\beta} \right] \leq 2,
\]
which implies that $\|X^k\|_{\psi_{\beta}} \leq \|X\|^{k}_{\psi_{k\beta}}$. The reverse direction is similar.
\end{proof}

For $\beta < 1$, $\|\cdot\|_{\psi_\beta}$ is not a norm , but the usual triangle inequality and maximal inequality hold up to a multiplicative constant.

\begin{lemma}\label{triangle_max_inequ}
Fix $\beta \in (0,1)$. 
\begin{enumerate}
\item[(i)] For any random variables $X$ and $Y$,
\[
\|X+Y\|_{\psi_\beta} \;\; \leq \;\; 2^{1+1/\beta} \left( \|X\|_{\psi_\beta} +\|Y\|_{\psi_\beta} \right).
\]
\item[(ii)] Let $\xi_1,\ldots,\xi_n$ be a sequence of random variables such that $\|\xi_i\|_{\psi_\beta} \leq D$ for $1 \leq i \leq n$, and $n \geq 2$. Then there exists a constant $C$  depending only on $\beta$ such that
\[
\|\max_{1 \leq i \leq n} \xi_i\|_{\psi_\beta} \; \leq C \log^{1/\beta}(n) D.
\]
\end{enumerate}
\end{lemma}
\begin{proof}
See Subsection \ref{additional_proofs}.
\end{proof}

\subsection{Proofs in Section~\ref{sec:IOUS_random_kernel}}\label{proof:IOUS_random}

We first prove Corollary~\ref{cor:IOUS} and then prove Theorem~\ref{thrm:IOUS_random}. 

\begin{proof}[Proof of Corollary~\ref{cor:IOUS}]
Let $c$ be the constant in Theorem~\ref{thm:exp_sup}. Without loss of generality, we assume
\begin{align}\label{wlog_c_IOUS}
\frac{r^2 D_n^2 \log^{q_{*}}(dn)}{\underline{\sigma}_g^2\; n} \leq c, \quad \text{ and } \quad \theta = 0 ,
\end{align}
since $\rho(\cdot,\cdot) \leq 1$ and we can always consider $h(\cdot) - \theta$ instead. Recall that
$q_{*} = (6/q+1) \vee 7$. 




Fix any rectangle $R = [a,b] \in \cR$, where $a,b \in \bR^d$ and $a \leq b$. Define
\begin{align*}
\tilde{a} = r^{-1}\Lambda_g^{-1/2} a, \;\; \tilde{b} =  r^{-1}\Lambda_g^{-1/2} b,\;\;
\tilde{U}_n = r^{-1}\Lambda_g^{-1/2} U_n, \;\;
\tilde{G}_i = \Lambda_g^{-1/2} g(X_i).
\end{align*}
Denote
$$
\xi_n := \max_{1 \leq j \leq d} \left\vert \tilde{U}_{n,j} - \frac{1}{n}\sum_{i=1}^{n} \tilde{G}_{i,j} \right\vert.
$$
Then by Theorem~\ref{thm:exp_sup}, 
$$\Exp[\xi_n] \;\leq \;  r^{-1}\underline{\sigma}_g^{-1}
\max_{1 \leq j \leq d} \left\vert  U_{n,j} - \frac{r}{n}\sum_{i=1}^{n} g_j(X_{i}) \right\vert \; \lesssim \;
\underline{\sigma}_g^{-1} n^{-1} r \log^{1+1/q}(d) D_n.
$$
For any $t > 0$, by Markov inequality and definition,
\begin{align*}
&\Pro(\sqrt{n} U_n \in R) = \Pro(-\sqrt{n} U_n \leq -a \; \cap \; \sqrt{n} U_n \leq b)
= \Pro(-\sqrt{n} \tilde{U}_n \leq - \tilde{a} \; \cap \; \sqrt{n} \tilde{U}_n \leq \tilde{b})\\
& \leq \Pro(-\sqrt{n} \tilde{U}_n \leq -\tilde{a} \;\cap\; \sqrt{n} \tilde{U}_n \leq \tilde{b} \;\cap\; \sqrt{n} \xi_n \leq t)
+ \Pro(\sqrt{n} \xi_n > t) \\
& \leq \Pro(- \frac{1}{\sqrt{n}} \sum_{i=1}^{n} \tilde{G}_i \leq -\tilde{a} + t \;\cap\; \- \frac{1}{\sqrt{n}} \sum_{i=1}^{n} \tilde{G}_i \leq \tilde{b} +t) + C t^{-1} \underline{\sigma}_g^{-1} n^{-1/2} r \log^{1+1/q}(d) D_n.
\end{align*}

Due to assumptions~\eqref{moments_assumption},~\eqref{h_exp_assumption} and Cauchy-Schwarz inequality,
\begin{align*}
&\Exp[\tilde{G}_{i,j}^2] = 1, \text{ for } 1 \leq i \leq n, 1 \leq j \leq d, \\
&\Exp[\tilde{G}_{i,j}^{4}] \leq (\sigma_{g,j}^{-1} D_n)^{2}
\leq (\underline{\sigma}_{g}^{-1} D_n)^{2} , \text{ for } 1 \leq i \leq n, 1 \leq j \leq d, \\
&\Exp[|\tilde{G}_{i,j}|^{3}] \leq \sqrt{\Exp[\tilde{G}_{i,j}^{2}] \Exp[\tilde{G}_{i,j}^{4}]}
\leq \underline{\sigma}_{g}^{-1} D_n, \text{ for } 1 \leq i \leq n, 1 \leq j \leq d, \\
& \| \tilde{G}_{i,j}\|_{\psi_q} \leq \sigma_{g,j}^{-1} D_n
\leq \underline{\sigma}_{g}^{-1} D_n , \text{ for } 1 \leq i \leq n, 1 \leq j \leq d.
\end{align*}

Then due to Lemma \ref{lemma:GAR_q},  we have
\begin{align*}
\Pro(\sqrt{n} U_n \in R) \leq \Pro&(- \Lambda_g^{-1/2} Y_A \leq -\tilde{a} + t \;\cap\; \Lambda_g^{-1/2} Y_A \leq \tilde{b} +t)  \\
& + C \left(\underline{\sigma}_g^{-2} n^{-1}D_n^2 \log^{q_{*}}(dn) \right)^{1/6} +
 C t^{-1}\underline{\sigma}_g^{-1} n^{-1/2} r \log^{1+1/q}(d) D_n.
\end{align*}
Further, by anti-concentration inequality~\cite[Lemma A.1]{chernozhukov2017},
\begin{align*}
\Pro(\sqrt{n} U_n \in R) \leq \Pro&(- \Lambda_g^{-1/2} Y_A \leq -\tilde{a}  \;\cap\; \Lambda_g^{-1/2} Y_A \leq \tilde{b} ) + 
C t \sqrt{\log(d)}  
\\
& + C \left(\underline{\sigma}_{g}^{-2} n^{-1}D_n^2 \log^{q_{*}}(dn) \right)^{1/6} +
 C t^{-1}\underline{\sigma}_g^{-1} n^{-1/2} r \log^{1+1/q}(d) D_n.
\end{align*}
Finally, taking $t = \left(\underline{\sigma}_g^{-2} n^{-1} r^{2}\log^{1+2/q}(d) D_n^{2}\right)^{1/4}$ and due to convention~\eqref{wlog_c_IOUS}, we have
\begin{align*}
\Pro(\sqrt{n} U_n \in R) &\leq \Pro(r Y_A \in R)
+  C \left(\underline{\sigma}_{g}^{-2} n^{-1} D_n^2 \log^{q_{*}}(dn) \right)^{1/6} 
+  C \left(\underline{\sigma}_{g}^{-2}n^{-1} r^{2}\log^{3+2/q}(d) D_n^{2}\right)^{1/4} \\
& \leq \Pro(r Y_A \in R)
+ C \left(\underline{\sigma}_{g}^{-2} n^{-1} D_n^2 \log^{q_{*}}(dn) \right)^{1/6} 
+  C\left(\underline{\sigma}_{g}^{-2} n^{-1} r^{2}\log^{q_{*}}(d) D_n^{2}\right)^{1/6} \\
&\leq \Pro(r Y_A \in R)
+  C \left(\underline{\sigma}_{g}^{-2} n^{-1} r^{2} D_n^2 \log^{q_{*}}(dn) \right)^{1/6}.
\end{align*}
Likewise, we can show the lower inequality 
$$
\Pro(\sqrt{n} U_n \in R) \geq \Pro(r Y_A \in R)
- C \left(\underline{\sigma}_{g}^{-2} n^{-1} r^{2} D_n^2 \log^{q_{*}}(dn) \right)^{1/6},
$$
which completes the proof.
\end{proof}


\begin{proof}[Proof of Theorem~\ref{thrm:IOUS_random}]
As before, without loss of generality, we assume
\begin{align}\label{wlog_c_IOUS_random}
\theta = 0, \quad \text{ and } \quad \frac{r^2 D_n^2 \log^{q_{*}}(dn)}{\underline{\sigma}_g^2 n} \leq c_1,
\end{align}
for some sufficiently small $c_1 \in (0,1)$. Define for each $\iota =(i_1,\ldots,i_r) \in I_{n,r}$,
$$
\overline{H}_{\iota} := H(X_{i_1},\ldots,X_{i_r}, W_{\iota}) - h(X_{i_1},\ldots,X_{i_r})
:= H(X_{\iota}, W_{\iota}) - h(X_{\iota}).
$$
Then by definition,
$$
\widehat{U}_n = \mathrm{R}_n+ U_n, \; \text{ where }
\mathrm{R}_n := {|I_{n,r}|}^{-1}\sum_{\iota \in I_{n,r}}\overline{H}_{\iota}.
$$

\underline{\textit{Step 1}}.  We first show that 
\begin{equation}\label{bound_on_H_bar}
\Exp \left[
\max_{1\leq j \leq d} \left\vert
\mathrm{R}_{n,j} 
\right\vert
\right] \lesssim \frac{D_n \log^{1/2+1/q}(dn)}{n}.
\end{equation}

Note that conditional on $X_1^n$, $\mathrm{R}_n$ is an average of independent random vectors. 
Thus by~\cite[Lemma 8]{chernozhukov2015comparison},
\begin{align*}
&\Exp_{\vert X^n_1} \left[
\max_{1\leq j \leq d} |I_{n,r}|\left\vert
\mathrm{R}_{n,j} 
\right\vert
\right]
\lesssim \\
&\sqrt{\log(d) \max_{1 \leq j \leq d} \sum_{\iota \in I_{n,r}} \Exp_{\vert X^n_1}\left[ \overline{H}_j^2(X_\iota, W_{\iota})\right]} + \log(d) \sqrt{\Exp_{\vert X^n_1} \left[ 
\max_{\iota \in I_{n,r}} \max_{1 \leq j \leq d} \overline{H}_j^2(X_\iota, W_\iota)
\right]}.
\end{align*}
By definition~\eqref{def:exp_norm_W_cond} and maximal inequality (\cite[Lemma 2.2.2]{van1996weak} and Lemma \ref{triangle_max_inequ}),
\begin{align*}
&\Exp_{\vert X^n_1}\left[ \overline{H}_j^2(X_\iota, W_{\iota})\right] 
\leq B_{n,j}^2(X_\iota) \text{ for all } \iota \in I_{n,r}, \\
& \sqrt{\Exp_{\vert X^n_1} \left[ 
\max_{\iota \in I_{n,r}} \max_{1 \leq j \leq d} \overline{H}_j^2(X_\iota, W_\iota) \right]}
\leq r^{1/q}\log^{1/q}(dn) \max_{\iota \in I_{n,r}} \max_{1 \leq j \leq d} B_{n,j}(X_{\iota}).
\end{align*}
Define 
$$
Z_1 := \max_{1 \leq j \leq d} \frac{1}{|I_{n,r}|}\sum_{\iota \in I_{n,r}} B_{n,j}^2(X_{\iota})
\leq \max_{\iota \in I_{n,r}} \max_{1 \leq j \leq d} B_{n,j}^2(X_{\iota}) := M_1^2.
$$
Under the assumption~\eqref{exp_cond_distr} and again  maximal inequality (\cite[Lemma 2.2.2]{van1996weak} and Lemma \ref{triangle_max_inequ}), we have
$$
\| M_1 \|_{\psi_q}  \leq r^{1/q}\log^{1/q}(dn) D_n.
$$
Then, we have
\begin{align*}
\Exp \left[
\max_{1\leq j \leq d} \left\vert
\mathrm{R}_{n,j} 
\right\vert
\right]
&\lesssim 
 \sqrt{\frac{\log(d)}{|I_{n,r}|} } \Exp[M_1]
+ \frac{ r^{1/q} \log^{1+1/q}(dn)}{|I_{n,r}|} \Exp[M_1] \\
&\lesssim  
\left( 
\sqrt{\frac{\log(d)}{|I_{n,r}|} } 
+ \frac{ r^{1/q} \log^{1+1/q}(dn)}{|I_{n,r}|} 
\right) r^{1/q}\log^{1/q}(dn) D_n.
\end{align*}
Then due to Lemma \ref{Inr_large} and~\eqref{wlog_c_IOUS_random},  we have
\begin{align*}
\Exp \left[
\max_{1\leq j \leq d} \left\vert
\mathrm{R}_{n,j} 
\right\vert
\right]
\lesssim \frac{D_n \log^{1/2+1/q}(dn)}{n}.
\end{align*}
\vspace{0.2cm}

\underline{\textit{Step 2}}.  We finish the proof by a similar argument as in the proof of Corollary~\ref{cor:IOUS}. 

Fix any rectangle $R = [a,b] \in \cR$, where $a,b \in \bR^d$ and $a \leq b$.  Define
\begin{align*}
\tilde{a} = r^{-1}\Lambda_g^{-1/2} a, \;\; \tilde{b} =  r^{-1}\Lambda_g^{-1/2} b,\;\;
\tilde{Y}_A = \Lambda_g^{-1/2} Y_A.
\end{align*}
where we recall that $\Lambda_g$ is defined in~\eqref{def:Lambda_g_H}. Recall that $\widehat{U}_n = U_n + \mathrm{R}_n$.
For any $t > 0$, by Markov inequality, the result from Step 1, and Corollary~\ref{cor:IOUS},
\begin{align*}
&\Pro(\sqrt{n} \widehat{U}_n \in R) = \Pro(-\sqrt{n} \widehat{U}_n \leq -a \; \cap \; \sqrt{n} \widehat{U}_n \leq b) \\
 \leq \; & \Pro(-\sqrt{n} \widehat{U}_n \leq -a \;\cap\; \sqrt{n} \widehat{U}_n \leq b \;\cap\; \sqrt{n} \|\mathrm{R}_n\|_{\infty} \leq t) + \Pro(\sqrt{n} \|\mathrm{R}_n\|_{\infty} > t) \\
 \leq &\Pro(-\sqrt{n} {U}_n \leq -a +t \;\cap\; \sqrt{n} {U}_n \leq b +t)
+ C t^{-1} n^{-1/2} D_n \log^{1/2+1/q}(dn) \\
\leq \; &\Pro(- rY_A \leq -a +t \;\cap\; r Y_A \leq b +t)
+   C \left( \frac{r^2 D_n^2 \log^{q_{*}}(dn)}{\underline{\sigma}^2_{g}\,n} \right)^{1/6}
+  C t^{-1} n^{-1/2} D_n \log^{1/2+1/q}(dn)\\
\leq \; &\Pro(- \tilde{Y}_A \leq - \tilde{a} +t r^{-1}\underline{\sigma}^{-1}_{g} \;\cap\;  \tilde{Y}_A \leq \tilde{b} + t r^{-1} \underline{\sigma}^{-1}_{g})
+   C \left( \frac{r^2 D_n^2 \log^{q_{*}}(dn)}{\underline{\sigma}^2_{g}\,n} \right)^{1/6} \\
&+  C t^{-1} n^{-1/2} D_n \log^{1/2+1/q}(dn).
\end{align*}
Observe that $\Exp[\tilde{Y}_{A,j}^2] = 1$ for $1 \leq j \leq d$. By anti-concentration inequality~\cite[Lemma A.1]{chernozhukov2017},
\begin{align*}
\Pro(\sqrt{n} \widehat{U}_n \in R)  \; \leq \;
&\Pro(- \tilde{Y}_A \leq - \tilde{a} \;\cap\;  \tilde{Y}_A \leq \tilde{b}) +   C \left( \frac{r^2 D_n^2 \log^{q_{*}}(dn)}{\underline{\sigma}^2_{g}\,n} \right)^{1/6} \\
&+C t r^{-1} \underline{\sigma}^{-1}_{g}\sqrt{\log(d)}
+ C t^{-1} n^{-1/2} D_n \log^{1/2+1/q}(dn).
\end{align*}

Finally, taking $t = \left(\underline{\sigma}^{2}_{g} n^{-1}r^2\log^{2/q}(dn) D_n^2\right)^{1/4}$ and due to convention~\eqref{wlog_c_IOUS_random}, we have
\begin{align*}
\Pro(\sqrt{n} \widehat{U}_n \in R) \leq \Pro( rY_A \in R) 
+  C \left( \frac{r^2 D_n^2 \log^{q_{*}}(dn)}{\underline{\sigma}^{2}_{g}\, n} \right)^{1/6}.
\end{align*}
By a similar argument, we can show
\begin{align*}
\Pro(\sqrt{n} \widehat{U}_n \in R) \geq \Pro( rY_A \in R) 
-  C \left( \frac{r^2 D_n^2 \log^{q_{*}}(dn)}{\underline{\sigma}^{2}_{g}\, n} \right)^{1/6}.
\end{align*}
which completes the proof.
\end{proof}

\subsection{Proofs in Section~\ref{sec:incomplete_IOUS_random}}
In this subsection, without loss of generality, we assume $\theta = 0$.  Recall the definition $\Lambda_H$ in~\eqref{def:Lambda_g_H}.
Further, define a function $\tilde{H}: S^r*S' \to \bR^d$ by 
$\tilde{H}(x_1^r,w) = \Lambda_H^{-1/2}H(x_1^r,w) \text{ for any } x_1^r \in S^r, w \in S'$,
and 
\begin{equation}\label{def:est_cov}
\Gamma_{\tilde{H}} := \text{Cov}(\tilde{H}(X_1^r,W)) = \Lambda_H^{-1/2} \Gamma_H \Lambda_H^{-1/2}, \quad
\widehat{\Gamma}_{\tilde{H}} := \frac{1}{|I_{n,r}|} \sum_{\iota \in I_{n,r}} \tilde{H}(X_{\iota},W_{\iota}) \tilde{H}(X_{\iota},W_{\iota})^T.
\end{equation}
Clearly, if~\eqref{H_moments_assumption} holds, then
\begin{align}\label{tilde_H_assumptions}
\Exp|\tilde{H}_{j}(X_1^r,W)|^{2+k} \leq (\sigma_{H,j}^{-1} D_n)^{k} \leq (\underline{\sigma}_H^{-1} D_n)^k, \;\; \text{ for }
1 \leq j \leq d, k= 1,2,
\end{align}
where again we applied Cauchy–Schwarz inequality for $k = 1$.


\subsubsection{Bounding $\widehat{N}/N$}
The following lemma follows from an application of Bernstein's inequality and is proved in the Step 5 of the  proof of~\cite[Theorem 3.1]{chen2017randomized}. It is included here for easy reference.

\begin{lemma}\label{lemma:N_hat}
Assume $\sqrt{\log(n)/N} \leq 1/4$. Then 
$$
\Pro\left(|\widehat{N}/N - 1| > 2\sqrt{\log(n)/N} \right) \leq 2n^{-1},\quad
\Pro\left(|N/\widehat{N} - 1| > 4\sqrt{\log(n)/N} \right) \leq 2n^{-1}.
$$
\end{lemma}

\subsubsection{Bounding the normalized covariance estimator}
\begin{lemma}\label{lemma:convariance_est}
Assume~\eqref{H_exp_moment_assumption},~\eqref{exp_cond_distr} and~\eqref{H_moments_assumption} hold. Then there exists a constant $C$, depending only on $q$,  such that with probability at least $1-13/n$,
\begin{align*}
&\|\widehat{\Gamma}_{\tilde{H}} - \Gamma_{\tilde{H}}\|_{\infty} \;\leq \;\\
&C \left(\underline{\sigma}_H^{-1} n^{-1/2} r^{1/2} \log^{1/2}(dn) D_n
+ \underline{\sigma}_H^{-2} n^{-1} r \log^{2/q+1}(dn)\log^{2/q-1}(n) D_n^2 \right) \\ 
+&C 
\underline{\sigma}_H^{-2} D_n^2 |I_{n,r}|^{-1/2} r^{1/2} \log^{1/2}(dn)\left(1 +
n^{-1/2} r^{1/2} \log^{2/q+1/2}(dn)\log^{2/q-1/2}(n) \right)\\
 + &C \underline{\sigma}_H^{-2} D_n^2 |I_{n,r}|^{-1} r^{6/q} \log^{4/q+1}(dn)\log^{4/q-1}(n).
\end{align*}
\end{lemma}

\begin{proof}
Define $v(x_1^r) := \Exp[\tilde{H}(x_1^r, W) \tilde{H}(x_1^r,W)^T],\;\widehat{V} := |I_{n,r}|^{-1}\sum_{\iota \in I_{n,r}} v(X_{\iota})$.  Observe that
$$ \|\widehat{\Gamma}_{\tilde{H}} - \Gamma_{\tilde{H}}\|_{\infty} \;\leq\; 
\|\widehat{\Gamma}_{\tilde{H}} - \widehat{V}\|_{\infty} +
\|\widehat{V} - \Gamma_{\tilde{H}}\|_{\infty}.
$$
We will bound these two terms separately. 
\vspace{0.5cm}

\noindent{\underline{Step 0}}. We first make a few observations. Clearly, 
$\Exp[v(X_1^r)] = \Gamma_{\tilde{H}}$, and  for all $1 \leq j,k \leq d$, by Jensen's inequality for conditional expectation and~\eqref{tilde_H_assumptions},
\begin{align}\label{v_fourth}
\Exp\left\vert v_{jk}(X_1^r)  \right\vert^2 &\leq
\Exp[\tilde{H}_j^2(X_1^r,W) \tilde{H}_k^2(X_1^r,W)] 
\leq
\Exp[\tilde{H}_j^4(X_1^r,W)] +
\Exp[\tilde{H}_k^4(X_1^r,W)]   \lesssim \underline{\sigma}_{H}^{-2} D_n^2.
\end{align}
Further, by definition
\begin{align*}
\left\vert v_{jk}(x_1^r) \right\vert &\leq
\Exp[\tilde{H}_j^2(x_1^r,W)] +
\Exp[\tilde{H}_k^2(x_1^r,W)]  \\
& \lesssim \underline{\sigma}_{H}^{-2}\left(B_{n,j}^2(x_1^r) +h_j^2(x_1^r) + B_{n,k}^{2}(x_1^r)+h_k^2(x_1^r) \right).
\end{align*}
As a result,  by the assumptions~\eqref{exp_cond_distr} and~\eqref{H_exp_moment_assumption}, and Lemma~\ref{psi_norms},
\begin{align}
\begin{split}
\max_{1 \leq j,k \leq d}\|v_{jk}(X_1^r)\|_{\psi_{q/2}} &\lesssim
\underline{\sigma}_{H}^{-2}\max_{1\leq j\leq d} 
\left(
\|B_{n,j}^2(X_1^r)\|_{\psi_{q/2}} +
\|h_{j}^2(X_1^r)\|_{\psi_{q/2}} 
\right) \\
&= \underline{\sigma}_{H}^{-2}\max_{1\leq j \leq d} 
\left(
\|B_{n,j}(X_1^r)\|_{\psi_{q}}^2 +
\|h_{j}(X_1^r)\|_{\psi_{q}}^2 
\right) \lesssim (\underline{\sigma}_{H}^{-1}D_n)^2.
\end{split}\label{v_exp}
\end{align}
\vspace{0.2cm}

\noindent\underline{\textit{Step 1}}. We bound $\|\widehat{\Gamma}_{\tilde{H}} - \widehat{V}\|_{\infty}$
using Lemma~\ref{lemma:tail_random_kernel} with $F = \tilde{H}\tilde{H}^T$ and $\psi_{q/2}$.
For $1\leq j,k \leq d$, define 
$$b_{jk}(x_1^r) :=\|\tilde{H}_j(x_1^r,W) \tilde{H}_k(x_1^r,W) - v_{jk}(x_1^r)\|_{\psi_{q/2}}.
$$
Observe that due to Lemma \ref{psi_norms} and \ref{triangle_max_inequ},
\begin{align*}
b_{jk}(x_1^r) &\lesssim 
\|\tilde{H}_j^2(x_1^r,W)\|_{\psi_{q/2}} + \|\tilde{H}_k^2(x_1^r,W)\|_{\psi_{q/2}} +v_{jk}(x_1^r)\\
&=\|\tilde{H}_j(x_1^r,W)\|_{\psi_{q}}^2 + \|\tilde{H}_k(x_1^r,W)\|_{\psi_{q}}^2 +v_{jk}(x_1^r)\\
&\lesssim \underline{\sigma}_{H}^{-2}(
h_j^2(x_1^r) +B_{n,j}^2(x_1^r)
+
h_k^2(x_1^r) +B_{n,k}^2(x_1^r)) +v_{jk}(x_1^r).
\end{align*}
Then due to~\eqref{v_exp} and the assumptions~\eqref{exp_cond_distr} and~\eqref{H_exp_moment_assumption}, 
\begin{align*}
\|b_{jk}(X_1^r)\|_{\psi_{q/2}} \lesssim (\underline{\sigma}_H^{-1} D_n)^2, \text{ for all }  1 \leq j,k \leq d.
\end{align*}
Now we apply Lemma~\ref{lemma:tail_random_kernel}, with probability at least $1- 9/n$,
\begin{align*}
\|\widehat{\Gamma}_{\tilde{H}} - \widehat{V}\|_{\infty} \lesssim
\;\;&\underline{\sigma}_H^{-2}D_n^2 |I_{n,r}|^{-1/2} r^{1/2} \log^{1/2}(dn)\left(1 +
n^{-1/2} r^{1/2} \log^{2/q+ 1/2}(dn)\log^{2/q-1/2}(n) \right) \\
 + & \underline{\sigma}_H^{-2} D_n^2 |I_{n,r}|^{-1} r^{6/q} \log^{4/q+1}(dn)\log^{4/q-1}(n).
\end{align*}
\vspace{0.2cm}

\noindent{\underline{\textit{Step 2}}}. We bound $\|\widehat{V} - \Gamma_{\tilde{H}}\|_{\infty}$ using Lemma~\ref{tail general U} with $\psi_{q/2}$. By~\eqref{v_fourth} and~\eqref{v_exp}, with probability at least $1- 4/n$,
$$
\|\widehat{V} - \Gamma_{\tilde{H}	}\|_{\infty}
\lesssim 
n^{-1/2} r^{1/2} \log^{1/2}(dn) \underline{\sigma}_H^{-1}D_n
+ n^{-1} r \log^{2/q+1}(dn)\log^{2/q-1}(n) \underline{\sigma}_H^{-2}D_n^2.
$$
Then the proof is complete by combining step 1 and 2.
\end{proof}

\subsubsection{Bounding the effect of sampling}\label{subsec:effect of sampling}
The following quantity will appear in the proof of Theorem~\ref{thrm:IOUS_incomplete}:
\begin{align}\label{zeta_n}
\sqrt{N} \zeta_n :=  \frac{1}{\sqrt{|I_{n,r}|}}\sum_{\iota \in I_{n,r}} \frac{Z_{\iota} - p_n}{\sqrt{p_n(1-p_n)}} \tilde{H}(X_{\iota}, W_{\iota})
:=  \frac{1}{\sqrt{|I_{n,r}|}}\sum_{\iota \in I_{n,r}} \widetilde{Z}_{\iota},
\end{align}
The next lemma establishes conditional Gaussian approximation for $\sqrt{N} \zeta_n$. 

\begin{lemma}\label{lemma:effect_of_sampling}
Suppose the assumptions in Theorem~\ref{thrm:IOUS_incomplete} hold. There exists a  constant $C$, depending on $q$, such that
with probability at least $1 - C/n$,
$$
\rho^{\cR}_{\vert X,W}(\sqrt{N} \zeta_n, \Lambda^{-1/2}_H {Y}_{B})
:= \sup_{R \in \cR}
\left\vert
\Pro_{\vert X,W}\left(\sqrt{N}\zeta_n \in R \right)
- \Pro(\Lambda^{-1/2}_H {Y}_B \in R)
\right\vert \leq C \varpi_n,
$$
where we recall that $Y_B \sim N(0,\Gamma_H)$, and
we abbreviate $\Pro_{\vert X,W}$ for $\Pro_{\vert X_1^{n}, \{W_{\iota}: \iota \in I_{n,r}\}}$.
\end{lemma}
\begin{proof} Consider conditionally independent (conditioned on $X,W$) $\bR^d$-valued random vectors $\{\widehat{Y}_{\iota}: \iota \in I_{n,r}\}$ such that
$$
\widehat{Y}_{\iota} \vert X,W
\;\sim \; N(0, \tilde{H}(X_{\iota}, W_{\iota}) \tilde{H}(X_{\iota}, W_{\iota})^T),\quad
\widehat{Y} := |I_{n,r}|^{-1/2} \sum_{\iota \in I_{n,r}} \widehat{Y}_{\iota}.$$
Clearly, $\widehat{Y}\vert X,W \;\sim\; N(0,\widehat{\Gamma}_{\tilde{H}})$. Further, define
\begin{align*}
\rho^{\cR}_{\vert X,W}(\sqrt{N} \zeta_n, \widehat{Y})
&:= \sup_{R \in \cR}
\left\vert
\Pro_{\vert X,W}\left(\sqrt{N}\zeta_n \in R \right)
- \Pro_{\vert X,W}(\widehat{Y}\in R)
\right\vert, \\
\rho^{\cR}_{\vert X,W}(\widehat{Y}, \Lambda^{-1/2}_H Y_B)
&:= \sup_{R \in \cR}
\left\vert
\Pro_{\vert X,W}\left(\widehat{Y} \in R \right)
- \Pro(\Lambda^{-1/2}_H {Y}_B\in R)
\right\vert.
\end{align*}
By triangle inequality, it then suffices to show that each of the following events happens with probability at least $1-C/n$,
\begin{align}\label{aux_to_show}
\rho^{\cR}_{\vert X,W}(\sqrt{N} \zeta_n, \widehat{Y}) \leq C \varpi_n,\quad
\rho^{\cR}_{\vert X,W}(\widehat{Y}, \Lambda^{-1/2}_H Y_B) \leq C \varpi_n,
\end{align}
on which we now focus. Without loss of generality, since $\underline{\sigma}_g \leq 1$, we assume
\begin{equation}
\label{aux_wlog}
\frac{r^{q_1} D_n^2 \log^{q_{*}}(dn)}{\underline{\sigma}_g^{2} \; n \wedge N} \leq c_1, \quad \text{ and } \quad 
\frac{r^{q_1} D_n^2 \log^{q_{*}}(dn)}{ n \wedge N} \leq c_1.
\end{equation}
for some sufficiently small constant $c_1 \in (0,1)$ that is to be determined.
Recall that $q_1 = 2 \vee (2/q)$ and $q_{*} = (6/q+1) \vee 7$.

\vspace{0.3cm}

\noindent{\underline{\textit{Step 0}}}.  
By Lemma~\ref{lemma:convariance_est} and \ref{Inr_large},
\begin{align}\label{Gamma_H_hat_inf}
\Pro\left(
\|\widehat{\Gamma}_{\tilde{H}} - \Gamma_{\tilde{H}}\|_{\infty} \;\leq \;
C \left( \frac{r\log^{1 \vee (2/q-1)}(dn) D_n^2}{\underline{\sigma}_H^2\; n}\right)^{1/2}
\right)
\; \geq \;
1 - \frac{13}{n}.
\end{align}
In particular,
since $\Gamma_{\tilde{H},jj}  = 1$, if we take $c_1$ small enough such that
$C c_1^{1/2} \leq 1/2$,
then $\Pro\left(\min_{1 \leq j \leq d} \widehat{\Gamma}_{\tilde{H},jj} \geq 1/2 \right) \geq 1-13/n$.

\vspace{0.3cm}
\noindent{\underline{\textit{Step 1}}}.  The goal is to show that 
the first event in~\eqref{aux_to_show}, $\rho^{\cR}_{\vert X,W}(\sqrt{N} \zeta_n, \widehat{Y}) \leq C \varpi_n$, holds with probability at least $1 - C/n$.

\vspace{0.2cm}
\underline{\textit{Step 1.1.}} Define 
\begin{align}\label{def:Lhat}
\widehat{L}_n :=
\max_{1 \leq j \leq d} |I_{n,r}|^{-1}
\sum_{\iota \in I_{n,r}} \Exp_{\vert X,W}\left[|\widetilde{Z}_{\iota,j}|^3 \right].
\end{align}
Further,  $\widehat{M}_n(\phi) := \widehat{M}_{n,X}(\phi) + \widehat{M}_{n,Y}(\phi)$, where
\begin{equation}
\label{def:M_XY}
\begin{split}
&\widehat{M}_{n,X}(\phi) := |I_{n,r}|^{-1}\sum_{\iota \in I_{n,r}}
\Exp_{\vert X,W}\left[\max_{1 \leq j \leq d} |\widetilde{Z}_{\iota,j}|^3; \max_{1 \leq j \leq d} |\widetilde{Z}_{\iota,j}| > 
\frac{\sqrt{|I_{n,r}|}}{4\phi \log d}
\right],\\
&\widehat{M}_{n,Y}(\phi) := |I_{n,r}|^{-1}\sum_{\iota \in I_{n,r}}
\Exp_{\vert X,W}\left[\max_{1 \leq j \leq d} |\widehat{Y}_{\iota,j}|^3; \max_{1 \leq j \leq d} |\widehat{Y}_{\iota,j}| > 
\frac{\sqrt{|I_{n,r}|}}{4\phi \log d}
\right],
\end{split}
\end{equation}

By Theorem 2.1 in~\cite{chernozhukov2017},  there exist absolute constants $K_1$ and $K_2$ such that for
any real numbers $\overline{L}_n$ and $\overline{M}_n$, we have
\begin{align*}
\rho^{\cR}_{\vert X,W}(\sqrt{N} \zeta_n, \widehat{Y})
\leq K_1 \left( \left(
\frac{\overline{L}_n^2 \log^7(d)}{|I_{n,r}|}
\right)^{1/6} + \frac{\overline{M}_n}{\overline{L_n}} \right) \;\; \text{ with } \;\;
\phi_n := K_2 \left( \frac{\overline{L}_n^2 \log^4(d)}{|I_{n,r}|}\right)^{-1/6},
\end{align*}
on the event $\mathcal{E}_n := \{\widehat{L}_n \leq \overline{L}_n\} \cap \{\widehat{M}_n(\phi_n) \leq \overline{M}_n\} \cap \{\min_{1 \leq j \leq d} \widehat{\Gamma}_{\tilde{H},jj} \geq 1/2\}$.

In Step 0, we have shown $\Pro\left(\min_{1 \leq j \leq d} \widehat{\Gamma}_{\tilde{H},jj} \geq 1/2 \right) \geq 1-13/n$.
In Step 1.2-1.4, we select proper $\overline{L}_n$ and $\overline{M}_n$ such that the first two events happen with probability at least $1 - C/n$. In Step 1.5, we plug in these values.

%
%
%

\vspace{0.3cm}
\underline{\textit{Step 1.2: Select $\overline{L}_n$.}}\hspace{0.1cm}  
Since $p_n \leq 1/2$, $\Exp|Z_{\iota} - p_n|^3 \leq C p_n$, and thus
$$\widehat{L}_n \leq C p_n^{-1/2} Z_1, \;\text{ where } Z_1 := \max_{1 \leq j \leq d}\frac{1}{|I_{n,r}|}\sum_{\iota \in I_{n,r}} \left\vert
\tilde{H}_j(X_{\iota},W_{\iota})\right\vert^3
.$$
We will apply Lemma~\ref{lemma:nonneg_tail_random_kernel} with $F(\cdot) = |\tilde{H}(\cdot)|^3$ and $\beta = q/3$. Thus for $1 \leq j \leq d$, define
$$
f_j(x_1^r) := \Exp\left[ \left|\tilde{H}_j(x_1^r,W)\right|^3\right],\;\;
b_j(x_1^r) := \left\| \left|\tilde{H}_j(x_1^r,W)\right|^3 - f_j(x_1^r) \right\|_{\psi_{q/3}}.
$$
First, by iterated expectation and due to~\eqref{tilde_H_assumptions},
$$
\Exp\left[ f_j(X_1^r) \right] = \Exp\left[ \left|\tilde{H}_j(X_1^r,W)\right|^3\right] \leq \underline{\sigma}_H^{-1}D_n, \;\;\text{ for } 1 \leq j \leq d.
$$
Second, observe that $\sigma_{H,j}^3 f_j(x_1^r) \lesssim \Exp\left[ |H_j(x_1^r,W) - h_j(x_1^r)|^3 \right]  +|h_j(x_1^r)|^3
\lesssim B_{n,j}^3(x_1^r) + |h_j(x_1^r)|^3$, and thus due to \eqref{h_exp_assumption},~\eqref{exp_cond_distr} and Lemma \ref{psi_norms} and \ref{triangle_max_inequ}, 
\begin{align*}
\|f_j(X_1^r)\|_{\psi_{q/3}} &\;\lesssim\; \sigma_{H,j}^{-3}\left(
\|B_{n,j}^3(X_1^r)\|_{\psi_{q/3}}  +
\|h_j^3(X_1^r)\|_{\psi_{q/3}} \right) \\
&\;=\;\sigma_{H,j}^{-3}\left( \|B_{n,j}(X_1^r)\|_{\psi_{q}}^3 
+\|h_j(X_1^r)\|^3_{\psi_{q}} \right)
\;\lesssim\; (\underline{\sigma}_{H}^{-1}D_n )^3.
\end{align*}
Further, observe that by Lemma \ref{triangle_max_inequ},
\begin{align*}
\sigma_{H,j}^3 b_j(x_1^r) &\lesssim \| \left|H_j(x_1^r,W) - h_j(x_1^r)\right|^3 \|_{\psi_{q/3}} + |h_j^3(x_1^r)| + \sigma_{H,j}^3f_j(x_1^r) \\
&= B_{n,j}^3(x_1^r)+ |h_j^3(x_1^r)|  + \sigma_{H,j}^3f_j(x_1^r).
\end{align*}
Thus by the same argument, $\|b_j(X_1^r)\|_{\psi_{q/3}} \;\lesssim\; (\underline{\sigma}_H^{-1}D_n)^3$.
Then by Lemma~\ref{lemma:nonneg_tail_random_kernel}, with probability at least $1-8n^{-1}$,
\begin{align*}
 Z_1 \; \leq \;C \left( \underline{\sigma}_H^{-1} D_n
+  n^{-1}r\log^{6/q}(dn) \underline{\sigma}_H^{-3} D_n^3 +   |I_{n,r}|^{-1} r^{9/q} \log^{12/q}(dn)\underline{\sigma}_H^{-3} D_n^3
\right).
\end{align*}
Due to Lemma \ref{Inr_large} and assumption~\eqref{aux_wlog},
$\Pro( \widehat{L}_n \; \leq \;C \underline{\sigma}_H^{-1} p_n^{-1/2} D_n ) \geq 1 - 8/n$.
Thus there is a constant $C_1$, depending on $q$, such that if  
\begin{align}\label{overline_L_sel}
\overline{L}_n := C_1 \underline{\sigma}_H^{-1} p_n^{-1/2} r^{1/q} D_n,
\end{align}
then $\Pro(\widehat{L}_n \leq \overline{L}_n) \geq 1-8/n$.
\vspace{0.2cm}

\underline{\textit{Step 1.3: bounding $\widehat{M}_{n,X}(\phi_n)$.}}\hspace{0.1cm}  
Since $Z_{\iota}$ is a Bernoulli random variable, it is clear that $\widehat{M}_{n,X}(\phi_n) = 0$ on the event
$$
M := \max_{\iota \in I_{n,r}} \max_{1 \leq j \leq d} |\tilde{H}_j(X_{\iota},W_{\iota})| \leq \frac{\sqrt{N}}{4 \phi_n \log(d)}
= 4^{-1}K_2^{-1}C_1^{1/3} \left(\frac{r^{1/q} D_n N}{\underline{\sigma}_H \log(d)}\right)^{1/3},
$$
where we use the value~\eqref{overline_L_sel} for $\overline{L}_n$.

By assumption~\eqref{H_exp_moment_assumption} and Lemma \ref{triangle_max_inequ}, 
$$\|M\|_{\psi_q} \leq C' \underline{\sigma}_H^{-1} r^{1/q} D_n \log^{1/q}(dn) \; \Rightarrow\;
\Pro\left(M \leq C' \underline{\sigma}_H^{-1}  r^{1/q} D_n \log^{2/q}(dn) \right) \geq 1 - 2/n.
$$
Due to~\eqref{aux_wlog},
\begin{align*}
&\left(\frac{r^{1/q} D_n N}{\underline{\sigma}_H\log(d)}\right)^{1/3} \geq c_1^{-1/3}\underline{\sigma}_H^{-1}  r^{1/q} D_n \log^{2/q}(dn), \;\;\\
&\phi_n^{-1} = K_2^{-1} C_1^{1/3} \left(\frac{r^{2/q} D_n^2 \log^4(d)}{\underline{\sigma}_H^2 N} \right)^{1/6}
\leq K_2^{-1}C_1^{1/3} c_1^{1/6}.
\end{align*}

Thus if  we take $c_1$ in~\eqref{aux_wlog} to be sufficiently small such that 
$$
c_1^{-1/3} 4^{-1}K_2^{-1}C_1^{1/3} \geq C' \;\;\text{ and }\;\; K_2^{-1}C_1^{1/3} c_1^{1/6} \leq 1.
$$
then $\Pro(\widehat{M}_{n,X}(\phi_n) = 0) \geq 1-2/n$ and $\phi_n \geq 1$.
\vspace{0.3cm}

\underline{\textit{Step 1.4: select $\overline{M}_n$.}}\hspace{0.1cm}  
From Step 1.3, we have shown that
$$\Pro\left(\mathcal{E'}_n \right) \geq 1 - 2/n,
\text{ where }
\mathcal{E}'_n := \left\{M:= \max_{\iota \in I_{n,r}} \max_{1 \leq j \leq d} |\tilde{H}_j(X_{\iota},W_{\iota})| \leq C'\underline{\sigma}_H^{-1} r^{1/q} D_n \log^{2/q}(dn) \right\}.
$$
Then by the same argument as in Step 1.4 of the proof of~\cite[Theorem 3.1]{chen2017randomized}
and due to~\eqref{aux_wlog} and $\phi_n \geq 1$,
on the event $\mathcal{E}'_n$,  for any $\iota \in I_{n,r}$,
\begin{align*}
&\Exp_{\vert X,W}\left[\max_{1 \leq j \leq d} |\widehat{Y}_{\iota,j}|^3; \max_{1 \leq j \leq d} |\widehat{Y}_{j,\iota}| > 
\frac{\sqrt{|I_{n,r}|}}{4\phi_n \log d}
\right] \\
\leq \;\;&
C \left(\frac{\sqrt{|I_{n,r}|}}{4\phi_n \log d} + C M \log^{1/2}(d) \right)^3 \exp\left(-
\frac{\sqrt{|I_{n,r}|}}{ C M \phi_n \log^{3/2} d} 
\right) \\
\leq \;\;& C n^{3r/2} \exp\left(-
\frac{{|I_{n,r}|^{1/3}}}{ C \underline{\sigma}_H^{-2/3} r^{2/3q}D_n^{2/3} \log^{2/q+5/6}(dn)} 
\right) 
\leq \; C n^{3r/2} \exp\left(-|I_{n,r}|^{11/84}/C
\right).
\end{align*}
Thus there exists an absolute constant $C_2$ such that if we set
\begin{align}
\label{overline_M_sel}
\overline{M}_n := C_2  n^{3r/2} \exp\left(- |I_{n,r}|^{11/84}/C_2 \right),
\end{align}
then $\Pro(\widehat{M}_{n,Y}(\phi_n) \leq \overline{M}_{n}) \geq 1 -2/n$.
\vspace{0.3cm}


\underline{\textit{Step 1.5: plug in $\overline{L}_n$ and $\overline{M}_n$.}}\hspace{0.1cm}  
Recall the definition $\overline{L}_n$ and $\overline{M}_n$ in~\eqref{overline_L_sel} and~\eqref{overline_M_sel}. With these selections, we have shown  that $\Pro(\mathcal{E}_n) \geq 1 - C/n$, where we recall that $\mathcal{E}_n := \{\widehat{L}_n \leq \overline{L}_n\} \cap \{\widehat{M}_n(\phi_n) \leq \overline{M}_n\} \cap \{\min_{1 \leq j \leq d} \widehat{\Gamma}_{\tilde{H},jj} \geq 1/2\}$. Further, on the event $\mathcal{E}_n$,
\begin{align*}
&\rho^{\cR}_{\vert X,W}(\sqrt{N} \zeta_n, \widehat{Y})
\lesssim \left(
\frac{\overline{L}_n^2 \log^7(d)}{|I_{n,r}|}
\right)^{1/6} + \frac{\overline{M}_n}{\overline{L_n}} \\
&\lesssim
\left(\frac{r^{2/q} D_n^2 \log^7(d)}{\underline{\sigma}_H^{2} N}
\right)^{1/6} +
\frac{p_n^{1/2} \underline{\sigma}_H}{D_n r^{1/q}} n^{3r/2} \exp\left(-|I_{n,r}|^{11/84}/C_2 \right)
\lesssim C \varpi_n,
\end{align*}
which completes the proof of Step 1.

\vspace{0.4cm}
\underline{\textit{Step 2.}}\hspace{0.1cm}  
The goal is to show that 
the second event in~\eqref{aux_to_show}, $\rho^{\cR}_{\vert X,W}(\widehat{Y}, \Lambda^{-1/2}_H Y_B) \leq C \varpi_n$, holds with probability at least $1 - C/n$.

Observe that $\Cov( \Lambda^{-1/2}_H Y_B) = \Gamma_{\tilde{H}}$ and $\Gamma_{\tilde{H},jj} = 1$ for $1 \leq j \leq d$.
By the Gaussian comparison inequality~\citep[Lemma C.5] {chen2017randomized},
$$
\rho^{\cR}_{\vert X,W}(\widehat{Y},  \Lambda^{-1/2}_H Y_B) \lesssim \overline{\Delta}^{1/3} \log^{2/3}(d),
$$
on the event that $\{\|\widehat{\Gamma}_{\tilde{H}}-\Gamma_{\tilde{H}}\|_{\infty} \leq \overline{\Delta}\}$.
From~\eqref{Gamma_H_hat_inf} in Step 0,
\begin{align*}
\Pro\left(
\|\widehat{\Gamma}_{\tilde{H}} - \Gamma_{\tilde{H}}\|_{\infty} \;\leq \;
C  (\underline{\sigma}_H^{-2} n^{-1}r\log^{1 \vee(2/q-1)}(dn) D_n^2)^{1/2}
\right)\; \geq \;
1 - 13/n.
\end{align*}
Thus if we set $\overline{\Delta} =C (\underline{\sigma}_H^{-2} n^{-1}r\log^{1 \vee(2/q-1)}(dn) D_n^2)^{1/2}$, then with probability at least $1-C/n$,
\begin{align*}
\rho^{\cR}_{\vert X,W}(\widehat{Y}, Y_B) \leq C 
\left(\frac{r\log^{5 \vee(2/q+3)}(dn) D_n^2}{\underline{\sigma}_H^{2}\; n} \right)^{1/6} \leq C \varpi_n.
\end{align*}
\end{proof}

\vspace{0.2cm}
\subsubsection{Proof of Theorem~\ref{thrm:IOUS_incomplete}}\label{proof:incomplete_IOUS_random}
Without loss of generality, we assume that 
\begin{equation}
\label{wlog_c_incomplete}
\frac{r^{q_1} D_n^2 \log^{q_*}(dn)}{\underline{\sigma}_g^2\; n \wedge N} \leq \frac{1}{16}.
\end{equation}
Observe that
\begin{align*}
 U'_{n,N} &= \frac{N}{\widehat{N}}\left( 
\frac{1}{N}\sum_{\iota \in I_{n,r}} (Z_{\iota} - p_n) \Lambda_H^{1/2}\tilde{H}(X_{\iota}, W_{\iota})
+
\frac{1}{|I_{n,r}|}\sum_{\iota \in I_{n,r}} {H}(X_{\iota},W_{\iota})
\right) \\
&= \frac{N}{\widehat{N}}\left( \sqrt{1-p_n} \Lambda_H^{1/2}\zeta_n + \widehat{U}_n\right)
:= \frac{N}{\widehat{N}} \Phi_n,
\end{align*}
where we recall that $\widehat{U}_n$ and $\zeta_n$ is defined in Section~\ref{sec:IOUS_random_kernel}
and in~\eqref{zeta_n} respectively. Denote $Y := rY_A +\alpha_n^{1/2}Y_B$. 

\vspace{0.2cm}
\underline{Step 1:} the goal is to show that
$$
\rho\left(\sqrt{n} \Phi_n, \;\; rY_A +\alpha_n^{1/2}Y_B \right)
\lesssim \varpi_n.
$$
For any rectangle $R \in \cR$, observe that
\begin{align*}
&\Pro(\sqrt{n}\left(\widehat{U}_n + \sqrt{1-p_n} \Lambda_H^{1/2} \zeta_n \right) \in R)\\
= \;\; &\Exp\left[
\Pro_{\vert X,W}\left(
\sqrt{N} \zeta_n \in \left( \frac{1}{\sqrt{\alpha_n (1-p_n)}} \Lambda_H^{-1/2} R-
\sqrt{\frac{N}{1-p_n}}\Lambda_H^{-1/2} \widehat{U}_n
\right)
\right)
\right].
\end{align*}
By Lemma~\ref{lemma:effect_of_sampling}, since $n^{-1} \lesssim \varpi_n$, we have
\begin{align*}
&\Pro(\sqrt{n}\left( \widehat{U}_n + \sqrt{1-p_n} \Lambda_H^{1/2}\zeta_n \right) \in R)\\
\leq  \;\; &\Exp\left[
\Pro_{\vert X,W}\left(
\Lambda_H^{-1/2} Y_B \in \left( \frac{1}{\sqrt{\alpha_n (1-p_n)}} \Lambda_H^{-1/2} R-
\sqrt{\frac{N}{1-p_n}}\Lambda_H^{-1/2} \widehat{U}_n
\right)
\right)
\right] + C \varpi_n \\
= \;\; &
\Pro\left(\sqrt{n} \widehat{U}_n \in 
\left[   R- \sqrt{\alpha_n(1-p_n)}Y_B
\right] \right)+ C \varpi_n,
\end{align*}
where we recall that $Y_B$ is independent of all other random variables.
Further, by Theorem~\ref{thrm:IOUS_random},
\begin{align*}
&\Pro(\sqrt{n}\left( \widehat{U}_n + \sqrt{1-p_n} \Lambda_H^{1/2}\zeta_n \right) \in R)\\
\leq \;\; &
\Exp\left[
\Pro_{\vert Y_B}\left(\sqrt{n} \widehat{U}_n \in 
\left[  R- \sqrt{\alpha_n(1-p_n)} Y_B
\right] \right)
\right]+ C \varpi_n \\
\leq \;\; &
\Exp\left[
\Pro_{\vert Y_B}\left(r Y_A \in 
\left[   R- \sqrt{\alpha_n(1-p_n)} Y_B
\right] \right)
\right]+ C \varpi_n, \\
= \;\; &
\Pro\left( \Lambda_g^{-1/2}(r Y_A + \sqrt{\alpha_n(1-p_n)} Y_B) \in \Lambda_g^{-1/2} R
\right) + C \varpi_n.
\end{align*}

Observe that $\Exp[(\sigma_{g,j}^{-1}r Y_{A,j})^2] = r^2 \geq 1$ for any $1 \leq j \leq d$, $\|\Gamma_H\|_{\infty} \lesssim D_n^2$ due to~\eqref{H_exp_moment_assumption}, and  $\alpha_n p_n = n/|I_{n,r}| \lesssim n^{-1}$. Then by the Gaussian comparison inequality~\citep[Lemma C.5] {chen2017randomized} and due to~\eqref{wlog_c_incomplete}
\begin{align*}
\Pro(\sqrt{n} \Phi_n \in R) &\; \leq \; 
\Pro\left( \Lambda_g^{-1/2}(r Y_A + \sqrt{\alpha_n} Y_B) \in \Lambda_g^{-1/2} R
\right) + C \varpi_n + C \left(\frac{D_n^2 \log^2(d)}{ \underline{\sigma}_g^2 n} \right)^{1/3} \\
&\;\leq \;
\Pro\left( r Y_A + \sqrt{\alpha_n} Y_B \in R
\right) + C \varpi_n.
\end{align*}
Similarly, we can show $\Pro(\sqrt{n}\Phi_n \in R)
\;\geq \; \Pro\left( r Y_A + \sqrt{\alpha_n} Y_B \in R
\right) - C \varpi_n$. Thus the proof of Step 1 is complete.

\vspace{0.3cm}
\underline{Step 2:} we show that with probability at least $1 - C \varpi_n$,
 $$\|(\frac{N}{\widehat{N}}-1)\sqrt{N} \Phi_n \|_{\infty} \leq C \nu_n, \;\;\; \text{ where }
 \nu_n := \sqrt{\frac{\log^3(dn) r^2 D_n^2}{n \wedge N}}.
 $$
Clearly, $\Exp[Y_j^2] = r^2 \sigma_{g,j}^2 + \alpha_n \sigma_{H,j}^2$. Then due to~\eqref{H_exp_moment_assumption}, $\Exp[Y_j^2]  \leq (r^2 + \alpha_n) D_n^2$. Since $Y$ is a multivariate Gaussian, 
$\max_{1\leq j \leq d} \|Y_j\|_{\psi_2} \leq \sqrt{(r^2 + \alpha_n) D_n^2}$. Then by the maximal inequality~\cite[Lemma 2.2.2]{van1996weak} $\|\max_{1\leq j \leq d} |Y_j|\|_{\psi_2} \leq C \sqrt{(r^2 + \alpha_n) D_n^2 \log(d)}$,
which further implies that 
$$\Pro\left(\max_{1\leq j \leq d} |Y_j| \geq C \sqrt{(r^2 + \alpha_n) D_n^2 \log(d)\log(n)} \right) \leq 2n^{-1}.$$
Since $n^{-1} \lesssim \varpi_n$,  and from the result in Step 1, we have
$$\Pro\left(\|\sqrt{n} \Phi_n\|_{\infty} \geq C \sqrt{(r^2 + \alpha_n) D_n^2 \log(d)\log(n)} \right) \leq C \varpi_n.$$
Finally, due to Lemma~\ref{lemma:N_hat} and~\eqref{wlog_c_incomplete}, we have with probability at least $1 - C \varpi_n$,
\begin{align*}
 \|({N}{/\widehat{N}}-1)\sqrt{N} \Phi_n \|_{\infty} \leq  C
\sqrt{(r^2 + \alpha_n) D_n^2 \log(d)\log^2(n) N^{-1} \alpha_n^{-1}}.   
 \end{align*} 
Since $(r^2 + \alpha_n)N^{-1} \alpha_n^{-1} = r^2n^{-1} + N^{-1} \leq 2r^2 (n\wedge N)^{-1}$, the proof is complete.
\vspace{0.3cm}

\underline{Step 3: final step.}  Recall that  $\sqrt{N} U_{n,N}' = \sqrt{N} \Phi_n + (N/\widehat{N}-1)\sqrt{N} \Phi_n$
and $\nu_n$ is defined in Step 2. For any rectangle $R = [a,b]$ with $a \leq b$, by Step 2,
\begin{align*}
\Pro\left( \sqrt{N} U_{n,N}' \in R\right)
&\leq \Pro\left( \sqrt{N} U_{n,N}' \in R \;\cap\; \|(N/\widehat{N}-1)\sqrt{N} \Phi_n\|_{\infty} \leq C \nu_n\right)
+ C \varpi_n \\
&\leq \Pro\left( \sqrt{N} \Phi_n \leq - a + C\nu_n \;\cap\; \sqrt{N} \Phi_n \leq b + C\nu_n\right)
+ C \varpi_n.
\end{align*}
Then by the result in Step 1, we have
\begin{align*}
\Pro\left( \sqrt{N} U_{n,N}' \in R\right)
&\leq \Pro\left( \alpha_n^{-1/2} Y \leq - a + C\nu_n \;\cap\; \alpha_n^{-1/2}Y  \leq b + C\nu_n\right)
+ C \varpi_n \\
&\leq \Pro\left( \alpha_n^{-1/2} \tilde{Y} \leq - \tilde{a} + C\underline{\sigma}_H^{-1} \nu_n \;\cap\; \alpha_n^{-1/2} \tilde{Y}  \leq \tilde{b} + C\underline{\sigma}_H^{-1} \nu_n\right) + C \varpi_n,
\end{align*}
where $\tilde{Y} = \Lambda_H^{-1} Y$, $\tilde{a} = \Lambda_H^{-1} a$ and $\tilde{b} = \Lambda_H^{-1} b$. Observe that
$\Exp[ (\alpha_n^{-1/2} \tilde{Y}_j)^2] \geq \Exp[(\sigma_{H,j}^{-1}Y_{B,j})^2]
= 1$ for $1 \leq j \leq d$, and thus by anti-concentration inequality~\cite[Lemma A.1]{chernozhukov2017},
\begin{align*}
\Pro\left( \sqrt{N} U_{n,N}' \in R\right)
&\leq \Pro\left( \alpha_n^{-1/2} \tilde{Y} \leq - \tilde{a}  \;\cap\; \alpha_n^{-1/2}Y  \leq \tilde{b} \right) + 
C\underline{\sigma}_H^{-1} \nu_n \log^{1/2}(d) +
C \varpi_n\\
&=  \Pro\left( \alpha_n^{-1/2} Y \in R\right) + 
\sqrt{ \frac{\log(d) \log^{3}(dn)\; r^2D_n^2}{\underline{\sigma}_H^2 n \wedge N}} +
C \varpi_n \\
&\leq  \Pro\left( \alpha_n^{-1/2} Y \in R\right)  +
C \varpi_n,
\end{align*}
where the last inequality is due to~\eqref{wlog_c_incomplete}. Similarly, we can show
$\Pro\left( \sqrt{N} U_{n,N}' \in R\right) \geq  \Pro\left( \alpha_n^{-1/2} Y \in R\right)  -
C \varpi_n$, and thus
$$
\rho(\sqrt{N} U_{n,N}', \;\;\alpha_n^{-1/2} Y) \lesssim \varpi_n,
$$
which completes the proof.
 

\vspace{0.2cm}
\subsection{Proofs in Section~\ref{sec:bootstrap}} \label{sec:proof_bootstrap}
In this subsection, without loss of generality, we assume $q \leq 1$ (see Remark \ref{remark:ALL}).

\subsubsection{Proof of Theorem~\ref{thrm:U_B_valid}}\label{proof:U_B}

\begin{proof} Without loss of generality, we can assume $\theta = \Exp[H(X_1^{r},W)] = 0$, since otherwise we can center $H$ first. Recall the definition of $\Lambda_H$ in~\eqref{def:Lambda_g_H},  $\tilde{H}(\cdot) = \Lambda_H^{-1/2}H(\cdot)$, and 
$\Gamma_{\tilde{H}}, \;\widehat{\Gamma}_{\tilde{H}}$ in~\eqref{def:est_cov}. Observe that for any integer $k$,
there exists some constant $C$ that depends only on $k$ and $\zeta$ such that
\begin{equation}\label{log_exp_bounded}
\log^{k}(n) n^{-\zeta} \leq C.
\end{equation}
\vspace{0.2cm}

\noindent{\underline{Step 0.}} Define $\tilde{U}_{n,N}' := \Lambda_H^{-1/2} U_{n,N}'$ and 
$$
\widehat{\Delta}_B := \left\|
\frac{1}{\widehat{N}}\sum_{\iota \in I_{n,r}} Z_{\iota} 
\left(\tilde{H}(X_{\iota},W_{\iota}) - \tilde{U}_{n,N}'\right)
\left(\tilde{H}( X_{\iota},W_{\iota}) -\tilde{U}_{n,N}'\right)^{T}
- \Gamma_{\tilde{H}}
\right\|_{\infty}.
$$
Since $\Gamma_{\tilde{H},jj} = 1$ for $1 \leq j \leq d$,  by Gaussian comparison inequality~\cite[C.5]{chen2017randomized},
\begin{align*}
&\sup_{R \in \cR}
\left| 
\Pro_{\vert \cD_{n}} \left( U^{\#}_{n,B} \in R \right)
- \Pro(Y_B \in R)
\right| \\
= \;&\sup_{R \in \cR}
\left| 
\Pro_{\vert \cD_{n}} \left( \Lambda_H^{-1/2} U^{\#}_{n,B} \in R \right)
- \Pro(\Lambda_H^{-1/2} Y_B \in R)
\right|
\lesssim \left( \widehat{\Delta}_B \log^2(d) \right)^{1/3}.
\end{align*}
Thus it suffices to show that with probability at least $1-C/n$,
$\widehat{\Delta}_B \log^2(d) \lesssim n^{-\zeta/2}$. Define

\begin{align*}
\widehat{\Delta}_{B,1} &:= \left\|{N}^{-1} \sum_{\iota \in I_{n,r}} (Z_{\iota}-p_n) \tilde{H}(X_{\iota},W_{\iota}) \tilde{H}(X_{\iota},W_{\iota})^T \right\|_{\infty}, \quad
\widehat{\Delta}_{B,2} := \left\|\widehat{\Gamma}_{\tilde{H}} - \Gamma_{\tilde{H}}\right\|_{\infty},\\
\widehat{\Delta}_{B,3} &:= |{N}/{\widehat{N}}-1| \left\|\Gamma_{\tilde{H}}\right\|_{\infty},
\quad
\widehat{\Delta}_{B,4} := \left\|N^{-1}\sum_{\iota \in I_{n,r} } Z_{\iota} \tilde{H}(X_{\iota}, W_{\iota}) \right\|_{\infty}^2.
\end{align*}
Then clearly $ \widehat{\Delta}_B 
\leq |N/\widehat{N}| \left( 
\widehat{\Delta}_{B,1} + 
\widehat{\Delta}_{B,2}\right)
+ 
\widehat{\Delta}_{B,3} + (N/\widehat{N})^2 \widehat{\Delta}_{B,4}$.

Without loss of generality, we can assume  $C_1 n^{-\zeta} \leq 1/16$, since we can always take $C$ to be large enough. Then by Lemma~\ref{lemma:N_hat}, $\Pro(|N/\widehat{N}| \leq C) \geq 1 -2n^{-1}$, and thus it suffices to show that
$$
\Pro\left( \widehat{\Delta}_{B,i} \log^2(d) \leq  C n^{-\zeta/2} \right) \geq 1 - C/n, \text{ for all }
i = 1,\ldots,4,
$$
on which we now focus.

\vspace{0.5cm}
\noindent{\underline{Step 1: bounding $\widehat{\Delta}_{B,1}$}.} 
Conditional on $\{X_{\iota},W_{\iota}: \iota \in I_{n,r}\}$, by Lemma~\ref{tail_Bernoulli},
\begin{align*}
\Pro\left(
N \widehat{\Delta}_{B,1} \leq C
\left(
\sqrt{N V_n \log(dn)}
+M_1\log(dn)
\right)
\right) \geq 1 - C/n,
\end{align*}
where
$$
V_n := \max_{1 \leq j,k \leq d}|I_{n,r}|^{-1} 
\sum_{\iota \in I_{n,r}} \tilde{H}_j^2(X_{\iota},W_{\iota})\tilde{H}_k^2(X_{\iota},W_{\iota}),
\;\;
M_1 := \max_{\iota \in I_{n,r}} \max_{1 \leq j \leq d} \tilde{H}_j^2(X_{\iota},W_{\iota}).
$$

First, by the maximal inequality (\cite[Lemma 2.2.2]{van1996weak} and Lemma \ref{triangle_max_inequ}) and due to~\eqref{H_exp_moment_assumption} and Lemma \ref{psi_norms} and \ref{triangle_max_inequ}, 
\begin{align*}
\|M_1\|_{\psi_{q/2}} \;\lesssim\; \underline{\sigma}_H^{-2}
r^{2/q}\log^{2/q}(dn) \max_{\iota \in I_{n,r}} \max_{1 \leq j \leq d}
\|H_j^2(X_{\iota},W_{\iota})\|_{\psi_{q/2}}
\; \lesssim\;
\underline{\sigma}_H^{-2}r^{2/q}\log^{2/q}(dn) D_n^2.
\end{align*}
As a result,  $\Pro\left( M_1 \leq C \underline{\sigma}_H^{-2} r^{2/q} D_n^2 \log^{2/q}(n)\log^{2/q}(dn) \right) \geq 1-2/n$.

Second, we will apply Lemma~\ref{lemma:nonneg_tail_random_kernel} to bound $V_n$ with $F_{jk}(\cdot) =\tilde{H}_j^2(\cdot)\tilde{H}_k^2(\cdot)$ and $\beta = q/4$. Note that by Lemma \ref{triangle_max_inequ}, for $1\leq j,k \leq d$, 
\begin{align*}
\sigma_{H,j}^{2} \sigma_{H,k}^{2} f_{jk}(x_1^r)& := \Exp\left[H_j^2(x_{1}^r,W)H_k^2(x_1^r,W) \right]
\lesssim \Exp\left[H_j^4(x_{1}^r,W) + H_k^4(x_1^r,W) \right] \\
& \lesssim 
h_j^4(x_{1}^r) + B_{n,j}^4(x_1^r)+
h_k^4(x_{1}^r) + B_{n,k}^4(x_1^r),\\
\sigma_{H,j}^{2} \sigma_{H,k}^{2} b_{jk}(x_1^r) &:= \|H_j^2(x_{1}^r,W)H_k^2(x_1^r,W) - \sigma_{H,j}^{2} \sigma_{H,k}^{2} f_{jk}(x_1^r)\|_{\psi_{q/4}} \\
&\lesssim  h_j^4(x_{1}^r) + B_{n,j}^4(x_1^r)+
h_k^4(x_{1}^r) + B_{n,k}^4(x_1^r) +\sigma_{H,j}^{2} \sigma_{H,k}^{2} f_{jk}(x_1^r).
\end{align*}
As a result, due to~\eqref{H_moments_assumption},~\eqref{h_exp_assumption} and~\eqref{exp_cond_distr}
\begin{align*}
\Exp[f_{jk}(X_1^r)] \lesssim (\underline{\sigma}_H^{-1} D_n)^2, \quad
\|f_{jk}(X_1^r)\|_{\psi_{q/4}} \lesssim (\underline{\sigma}_H^{-1} D_n)^4, \quad
\|b_{jk}(X_1^r)\|_{\psi_{q/4}} \lesssim (\underline{\sigma}_H^{-1} D_n)^4.
\end{align*}
Then by Lemma~\ref{lemma:nonneg_tail_random_kernel} and~\ref{Inr_large}, and due to~\eqref{U_B_assumption} and~\eqref{log_exp_bounded}
\begin{align*}
\Pro(V_n \leq C \underline{\sigma}_H^{-2} D_n^2) \geq 1-8/n.
\end{align*}

Finally, putting the two results together and again by~\eqref{log_exp_bounded}, we have
\begin{align*}
\Pro\left(\widehat{\Delta}_{B,1} \;\leq \; C\left(
N^{-1/2} \log^{1/2}(dn) \underline{\sigma}_H^{-1} D_n
+ N^{-1} r^{2/q}\log^{2/q}(n)\log^{2/q+1}(dn) \underline{\sigma}_H^{-2} D_n^2
\right)
\right) \geq 1-C/n.
\end{align*}
Then by~\eqref{U_B_assumption}, $\Pro\left(\widehat{\Delta}_{B,1} \;\leq \; C
\underline{\sigma}_H^{-1} N^{-1/2} r^{1/q} \log^{1/2}(dn)D_n
\right) \geq 1-C/n$, which implies that with probability at least $1-C/n$,
$$
\widehat{\Delta}_{B,1} \log^2(d) \;\leq \; C n^{-\zeta/2}.
$$

\vspace{0.5cm}
\noindent{\underline{Step 2: bounding $\widehat{\Delta}_{B,2}$}.} 
By Lemma~\ref{lemma:convariance_est} and~\ref{Inr_large}, and 
due to assumptions~\eqref{U_B_assumption} and~\eqref{log_exp_bounded} 
\begin{align*}
\Pro\left(
\widehat{\Delta}_{B,2} \leq C \underline{\sigma}_H^{-1} n^{-1/2}r^{1/2} \log^{1/2}(dn) D_n
\right) \geq 1 - 13/n,
\end{align*}
which implies $\Pro(\widehat{\Delta}_{B,2} \log^2(d) \leq C n^{-\zeta/2}) \geq 1-13/n$.

\vspace{0.5cm}

\noindent{\underline{Step 3: bounding $\widehat{\Delta}_{B,3}$}.} 
By definition, $\|\Gamma_{\tilde{H}}\|_{\infty} = 1$. Then by Lemma~\ref{lemma:N_hat} and~\eqref{U_B_assumption},
\begin{align*}
\widehat{\Delta}_{B,3} \log^2(d)\leq 4 N^{-1/2} \log^{1/2}(n)\log^2(d)\leq C n^{-\zeta/2},
\end{align*}
with probability at least $1 -2n^{-1}$.

\vspace{0.5cm}
\noindent{\underline{Step 4: bounding $\widehat{\Delta}_{B,4}$}.} 
Define
\begin{align*}
\widehat{\Delta}_{B,5} &:= 
\left\|N^{-1}\sum_{\iota \in I_{n,r} } (Z_{\iota}-p_n) \tilde{H}(X_{\iota}, W_{\iota}) \right\|_{\infty},\;\;
\widehat{\Delta}_{B,6} := 
\left\||I_{n,r}|^{-1}\sum_{\iota \in I_{n,r} }  \tilde{H}(X_{\iota}, W_{\iota}) \right\|_{\infty}. 
\end{align*}
Clearly, $\widehat{\Delta}_{B,4} \leq 2\left(\widehat{\Delta}^2_{B,5} +
\widehat{\Delta}^2_{B,6} \right).$ In the next two sub-steps, we will bound these two terms separately. 

\vspace{0.3cm}
\noindent{\underline{Step 4.1: bounding $\widehat{\Delta}_{B,5}^2$}.} 
Conditional on $\{X_{\iota},W_{\iota}: \iota \in I_{n,r}\}$, by Lemma~\ref{tail_Bernoulli},
\begin{align*}
\Pro\left(N\widehat{\Delta}_{B,5} \leq C 
\left(
\sqrt{N \widetilde{V}_n \log(dn)} + \widetilde{M}_1 \log(dn)
\right)
\right) \geq 1 - C/n,
\end{align*}
where $
\widetilde{V}_n := \max_{1 \leq j \leq d} |I_{n,r}|^{-1} \sum_{\iota \in I_{n,r}} \tilde{H}_j^2(X_{\iota}, W_{\iota}), \quad
\widetilde{M}_1 := \max_{\iota \in I_{n,r}} \max_{1 \leq j \leq d} |\tilde{H}_j(X_{\iota},W_{\iota})|$.

First, by the maximal inequality (\cite[Lemma 2.2.2]{van1996weak} and Lemma \ref{triangle_max_inequ}) and due to~\eqref{H_exp_moment_assumption}, 
\begin{align*}
\|\widetilde{M}_1\|_{\psi_{q}} \;\lesssim\;
\underline{\sigma}_{H}^{-1} r^{1/q}\log^{1/q}(dn) D_n.
\end{align*}
As a result,  $\Pro\left( \widetilde{M}_1 \leq C \underline{\sigma}_{H}^{-1} r^{1/q} D_n \log^{1/q}(n)\log^{1/q}(dn) \right) \geq 1-2/n$.

Second, we will apply Lemma~\ref{lemma:nonneg_tail_random_kernel} to bound $\widetilde{V}_n$ with $F_j(\cdot) = \tilde{H}_j^2(\cdot)$ and $\beta = q/2$. Define for $1\leq j \leq d$,
\begin{align*}
f_{j}(x_1^r)& := \Exp\left[\tilde{H}_j^2(x_{1}^r,W)\right],\quad
b_{j}(x_1^r) := \|\tilde{H}_j^2(x_{1}^r,W) - f_{j}(x_1^r)\|_{\psi_{q/2}}.
\end{align*}
By the similar argument as in Step 1, 
\begin{align*}
\Exp[f_{j}(X_1^r)] = 1, \quad
\|f_{j}(X_1^r)\|_{\psi_{q/2}} \lesssim (\underline{\sigma}_H^{-1} D_n)^2, \quad
\|b_{j}(X_1^r)\|_{\psi_{q/2}} \lesssim (\underline{\sigma}_H^{-1} D_n)^2.
\end{align*}
Then by Lemma~\ref{lemma:nonneg_tail_random_kernel} and~\ref{Inr_large}, and due to~\eqref{U_B_assumption} and~\eqref{log_exp_bounded}
we have
$\Pro(\widetilde{V}_n \leq C) \geq 1-8/n$.

Finally, putting the two results together, we have
\begin{align*}
\Pro\left(\widehat{\Delta}_{B,5}^2 \;\leq \; C\left(
N^{-1} \log(dn)
+ \underline{\sigma}_{H}^{-2} N^{-2} r^{2/q}\log^{2/q+2}(dn) \log^{2/q}(n) D_n^2
\right)
\right) \geq 1-C/n.
\end{align*}
Then by~\eqref{U_B_assumption}, $\Pro\left(\widehat{\Delta}_{B,5}^2 \;\leq \; C
N^{-1}  \log(dn)
\right) \geq 1-C/n$, which implies that with probability at least $1-C/n$,
$\widehat{\Delta}_{B,5}^2 \log^2(d) \;\leq \; C n^{-\zeta}$ holds.

\vspace{0.4cm}
\noindent{\underline{Step 4.2: bounding $\widehat{\Delta}_{B,6}^2$}.} 
Observe that
$\widehat{\Delta}_{B,6} \leq 
\widehat{\Delta}_{B,7} 
+
\widehat{\Delta}_{B,8}
$, where
$$
\widehat{\Delta}_{B,7} := \left\|
|I_{n,r}|^{-1}\sum_{\iota \in I_{n,r} }  \Lambda_H^{-1/2}
\left(H(X_{\iota}, W_{\iota}) - h(X_{\iota}) \right)
\right\|_{\infty}, 
\;\;
\widehat{\Delta}_{B,8} := \left\| |I_{n,r}|^{-1}\sum_{\iota \in I_{n,r} }  
 \Lambda_H^{-1/2} h(X_{\iota}) \right\|_{\infty} .
$$
By directly applying Lemma~\ref{lemma:tail_random_kernel} with $\beta = q$, due to~\eqref{U_B_assumption} and Lemma \ref{Inr_large},
\begin{align*}
\Pro\left(
\widehat{\Delta}_{B,7} \; \leq \; C \underline{\sigma}_{H}^{-1} D_n n^{-1} \log^{1/2}(dn)
\right) \geq 1 - 9/n.
\end{align*}
By directly applying Lemma~\ref{tail general U} with $\beta = q$ and due to~\eqref{U_B_assumption},
$$
\Pro\left( \widehat{\Delta}_{B,8} \leq C  \underline{\sigma}_{H}^{-1} n^{-1/2} r^{1/2} \log^{1/2}(dn)D_n
\right)\geq 1 -4/n.
$$
Thus $\Pro\left(\widehat{\Delta}_{B,6}^2 \log^2(d) \leq C n^{-\zeta}\right) \geq 1-C/n$.

Combining sub-step 4.1 and 4.2, we have $\Pro\left(\widehat{\Delta}_{B,4}^2 \log^2(d) \leq C n^{-\zeta}\right) \geq 1-C/n$.  And combining Step 0-4, we finish the proof.
\end{proof}

\subsubsection{Proof of Lemma~\ref{lemma:appr_Y_A}}\label{proof:appr_Y_A}
\begin{proof}
Without loss of generality, we can assume $\theta = \Exp[H(X_1^{r},W)] = 0$.
Recall the definition $\Lambda_g$ is~\eqref{def:Lambda_g_H}. By definition, $\Exp[(\sigma_{g,j}^{-1} Y_{A,j})^2]= 1$ for $1 \leq j \leq d$. Then by the Gaussian comparison inequality~\citep[Lemma C.5] {chen2017randomized},
\begin{align*}
\sup_{R \in \cR}
\left| 
\Pro_{\vert \cD_{n}} \left( U^{\#}_{n_1,A} \in R \right)
- \Pro( Y_A  \in R)
\right| 
&=
\sup_{R \in \cR}
\left| 
\Pro_{\vert \cD_{n}} \left( \Lambda_g^{-1/2} U^{\#}_{n_1,A} \in R \right)
- \Pro(\Lambda_g^{-1/2} Y_A  \in R) 
\right|  \\
&\lesssim (\widehat{\Delta}_A \log^2(d))^{1/3},
\end{align*}
where 
\begin{align*}
\widehat{\Delta}_A := \max_{1 \leq j,k\leq d}\left\vert
\frac{1}{\sigma_{g,j} \sigma_{g,k} n_1} \sum_{i_1 \in S_1}(G_{i_1,j} - \overline{G}_j)(G_{i_1,k} - \overline{G}_k)
- \frac{1}{\sigma_{g,j} \sigma_{g,k}} \Gamma_{g, jk}
\right\vert.
\end{align*}
By the same argument as in the proof of \cite[Theorem 4.2] {chen2017randomized},
\begin{align*}
\widehat{\Delta}_A \lesssim \widehat{\Delta}_{A,1}^{1/2} + \widehat{\Delta}_{A,1} + \widehat{\Delta}_{A,2} + \widehat{\Delta}_3^2,
\end{align*}
where $\widehat{\Delta}_{A,1}$ is defined in~\eqref{Delta_A_1}, and 
\begin{align*}
\widehat{\Delta}_{A,2} := \max_{1 \leq j,k \leq d} \left\vert
\frac{1}{\sigma_{g,j} \sigma_{g,k} n_1} \sum_{i_1 \in S_1}\left(g_j(X_{i_1})g_k(X_{i_1}) - \Gamma_{g,jk}\right)
\right\vert, \;\;
\widehat{\Delta}_{A,3} := \max_{1 \leq j,k \leq d} \left\vert
\frac{1}{\sigma_{g,j}  n_1} \sum_{i_1 \in S_1}g_j(X_{i_1})
\right\vert.
\end{align*}
\vspace{0.2cm}

\noindent{\underline{Step 1: bounding $\widehat{\Delta}_{A,1}$}.} By the second part of~\eqref{aux_Y_A}, we have
\begin{align*}
\Pro\left(\widehat{\Delta}_{A,1}^{1/2}\log^2(d) \leq C_1^{1/2} n^{-\zeta_2/2} \right) \leq 1 - Cn^{-1}, \quad
\Pro\left(\widehat{\Delta}_{A,1}\log^2(d) \leq C_1 n^{-\zeta_2} \right) \leq 1 - Cn^{-1}.
\end{align*}
\vspace{0.2cm}

\noindent{\underline{Step 2: bounding $\widehat{\Delta}_{A,2}$}.} 
We apply Lemma~\ref{tail_general_sum_ind} with $\beta = q/2$, $m = n_1$  and note that $n_1 \leq n$:
\begin{align*}
\Pro\left(
\widehat{\Delta}_{A,2} \geq C\left(
n_1^{-1} \sigma \log^{1/2}(dn) + n_1^{-1} u_n \log^{2/q+1}(dn_1) \log^{2/q}(n)
\right)
\right) \leq 4n^{-1}.
\end{align*}
where $
\sigma^2 = \max_{1\leq j,k \leq d} \sigma_{g,j}^{-2}\sigma_{g,k}^{-2}\sum_{i_1 \in S_1}
\Exp[(g_j(X_{i_1})g_k(X_{i_1}) - \Gamma_{g,jk})^2 ]$ and
$$u_n = \| \sigma_{g,j}^{-1}\sigma_{g,k}^{-1}\left(
g_j(X_{i_1})g_k(X_{i_1}) - \Gamma_{g,jk} \right)\|_{\psi_{q/2}}.
$$
By Lemma \ref{triangle_max_inequ}, \eqref{moments_assumption} and \eqref{H_exp_moment_assumption},
$
\sigma^2 \leq n_1 \left(\underline{\sigma}_g^{-1}  D_n\right)^2, \;
u_n\leq \left(\underline{\sigma}_g^{-1}  D_n\right)^2
$. Thus 
$$
\Pro\left(
\widehat{\Delta}_{A,2} \geq C\left(
n_1^{-1/2} \underline{\sigma}_g^{-1}  D_n \log^{1/2}(dn) + n_1^{-1} \underline{\sigma}_g^{-2}  D_n^2 \log^{2/q+1}(dn_1) \log^{2/q}(n)
\right)
\right) \leq 4n^{-1}.
$$
Then due to the first part of~\eqref{aux_Y_A} and~\eqref{log_exp_bounded}, $\Pro(\widehat{\Delta}_{A,2}\log^2(d) \geq C n^{-\zeta_1/2}) \leq Cn^{-1}$. 
\vspace{0.2cm}

\noindent{\underline{Step 3: bounding $\widehat{\Delta}_{A,3}$}.} 
We apply Lemma~\ref{tail_general_sum_ind} with $\beta = q$, $m = n_1$:
\begin{align*}
\Pro\left(
\widehat{\Delta}_{A,3} \geq C\left( n_1^{-1/2} \log^{1/2}(dn)
+ n_1^{-1} \underline{\sigma}_g^{-1}  D_n \log^2(d n_1) \log(n)
\right)
\right) \leq 4 n^{-1}.
\end{align*}
Then due to the first part of~\eqref{aux_Y_A} and~\eqref{log_exp_bounded}, $\Pro(\widehat{\Delta}_{A,3}^2\log^2(d) \geq C n^{-\zeta_1}) \leq Cn^{-1}$.
\end{proof}

\subsubsection{Proof of Theorem~\ref{thrm:appr_U_full}}\label{proof:appr_U_full}

\begin{proof} Without loss of generality, we can assume $\theta = \Exp[H(X_1^{r},W)] = 0$.

\noindent{\underline{Step 1}.} 
Let $\zeta_1 := \zeta$, $\zeta_2 := \zeta -1/\nu$. 
Due to Theorem~\ref{thrm:U_B_valid}, Lemma~\ref{lemma:appr_Y_A} and using the same argument as in the Step 3 of the proof of
~\cite[Theorem 4.2]{chen2017randomized}, it suffices to show the second part of~\eqref{aux_Y_A} holds. From the definition~\eqref{Delta_A_1},
$$
\widehat{\Delta}_{A,1} \leq \underline{\sigma}_{g}^{-2}\max_{1 \leq j \leq d} \frac{1}{n_1}\sum_{i_1 \in S_1} \left(
G_{i_1,j} - g_j(X_{i_1})
\right)^2 := \underline{\sigma}_{g}^{-2} \overline{\Delta}_{A,1}.
$$
In Step 2, we will show that
\begin{equation}\label{aux_full_U_appr}
\Exp\left[
\overline{\Delta}_{A,1}^{\nu} 
\right] 
\lesssim \left(n^{-1} r D_n^{2} \log^{2/q+1}(d)
\right)^{\nu}.
\end{equation}
Then by Markov inequality and \eqref{U_full_assumption},
\begin{align*}
\Pro\left(
\widehat{\Delta}_{A,1}\log^4(d) 
\geq C_1 n^{-\zeta_2}
\right) &\lesssim n^{\zeta_2 \nu} \underline{\sigma}_{g}^{-2\nu}\log^{4\nu}(d) \left(n^{-1} r D_n^{2} \log^{2/q+1}(d)
\right)^{\nu}\\
&= n^{-1}  \left(n^{\zeta} n^{-1}\underline{\sigma}_{g}^{-2}r D_n^2\log^{2/q+5}(d) \right)^{\nu}
\lesssim n^{-1},
\end{align*}
which completes the proof.

%
\vspace{0.2cm}

\noindent{\underline{Step 2}.} The goal is to show~\eqref{aux_full_U_appr}.
Define 
$$F(x_1^r, w) := \max_{1 \leq j \leq d}\vert H_j(x_1^r, w) \vert,\quad
{g}^{(i_1,k)}(X_{i_1}) := H( X_{\kindex}, W_{\kindex}) \text{ for }
i_1 \in S_1, k =1,\ldots,K.
$$ 
By Jensen's inequality,
\begin{align*}
\Exp[\overline{\Delta}_{A,1}^{\nu}] \leq \frac{1}{n_1} \sum_{i_1 \in S_1}
\Exp\left[
\max_{1 \leq j \leq d} \left|
G_{i_1,j} - g_j(X_{i_1})
\right|^{2\nu}
\right],
\end{align*}
and for each $i_1 \in S_1$, conditional on $X_{i_1}$, 
by Hoffmann-Jorgensen inequality~\cite[A.1.6.]{van1996weak},
\begin{align*}
&\Exp_{\vert_{X_{i_1}}}\left[
\max_{1 \leq j \leq d} \left|
G_{i_1,j} - g_j(X_{i_1})
\right|^{2\nu}
\right] \lesssim I_{i_1} + II_{i_1} := \\
&\left(
\Exp_{\vert X_{i_1}}\left[
\max_{1 \leq j \leq d} \left|
G_{i_1,j} - g_j(X_{i_1})
\right| \right]
\right)^{2\nu}
+ K^{-2\nu} 
\Exp_{\vert X_{i_1}}\left[\max_{1 \leq k \leq K}
\max_{1 \leq j \leq d} \left|
{g}_j^{({i_1},k)}(X_{i_1}) - g_j(X_{i_1})
\right|^{2\nu}
\right].
\end{align*}

\vspace{0.3cm}
\noindent{\underline{Step 2.1: bounding $II_{i_1}$.}} 
Observe that for each $1 \leq k \leq K$,
\begin{align*}
&\Exp_{\vert X_{i_1}} \left[
 \max_{1 \leq j \leq d}\left| g^{(i_1,k)}_j(X_{i_1}) - g_j(X_{i_1})\right|^{2\nu}
\right]
= \Exp_{\vert X_{i_1}} \left[
 \max_{1 \leq j \leq d}\left| g^{({i_1},k)}_j(X_{i_1}) - 
\Exp_{\vert X_i}\left[ g^{(i_1,k)}_j(X_{i_1}) \right] 
 \right|^{2\nu}
\right] \\
& \lesssim 
\Exp_{\vert X_{i_1}} \left[
 \max_{1 \leq j \leq d}\left| g^{({i_1},k)}_j(X_{i_1})  
 \right|^{2\nu} \right]= 
\Exp_{\vert X_{i_1}} \left[ F^{2\nu}(X_{\kindex}, W_{\kindex})  
 \right] \\
 &=
\Exp_{\vert X_{i_1}} \left[ F^{2\nu}(X_{\kindone}, W_{\kindone})   \right] := b(X_{i_1}).
\end{align*}
Thus $II_{i_1} \lesssim K^{-2\nu+1} b(X_{i_1})$.

\vspace{0.5cm}
\noindent{\underline{Step 2.2: bounding $I_{i_1}$.}} Observe that
for each $i_1 \in S_1$,
\begin{align*}
\max_{1 \leq j \leq d} \sum_{k=1}^{K} \Exp_{\vert X_{i_1}}\left[
\left(
g^{({i_1},k)}_j(X_{i_1}) - g_j(X_{i_1})
\right)^2
\right] \lesssim
K \Exp_{\vert X_{i_1}} \left[ F^{2}(X_{\kindone}, W_{\kindone})  
  \right] := K \widetilde{b}(X_{i_1}).
\end{align*}
Further, by Jensen's inequality,
\begin{align*}
&\Exp_{\vert X_{i_1}}\left[\max_{1 \leq k \leq K}
\max_{1 \leq j \leq d} \left|
{g}_j^{({i_1},k)}(X_{i_1}) - g_j(X_{i_1})
\right|^{2}
\right] 
\\
&\leq
\left(
\sum_{k=1}^{K}
\Exp_{\vert X_{i_1}}\left[
\max_{1 \leq j \leq d} \left|
{g}_j^{({i_1},k)}(X_{i_1}) - g_j(X_{i_1})
\right|^{2\nu}
\right] 
\right)^{1/\nu} \lesssim
K^{1/\nu} b^{1/\nu}(X_{i_1}),
\end{align*}
where $b(X_{i_1})$ is defined in Step 1. Then by the same argument as in the proof of~\cite[Proposition 4.4]{chen2017randomized},
\begin{align*}
I_{i_1} \lesssim K^{-\nu} \log^{\nu}(d) \widetilde{b}^{\nu}(X_{i_1})
+ K^{-2\nu+1} \log^{2\nu}(d) b(X_{i_1}).
\end{align*}

\vspace{0.5cm}
\noindent{\underline{Step 2.3: combining 2.1 and 2.2.}} By Jensen's inequality, assumption~\eqref{H_exp_moment_assumption} and 
by the maximal inequality (\cite[Lemma 2.2.2]{van1996weak} and Lemma \ref{triangle_max_inequ})
$$
\Exp\left[ \widetilde{b}^{\nu}(X_{i_1}) \right] \leq \log^{2\nu/q}(d) D_n^{2\nu},\quad
\Exp\left[
b(X_{i_1})
\right] \leq \log^{2\nu/q}(d) D_n^{2\nu}.
$$
Thus combining the results from 2.1 and 2.2, we have
\begin{align*}
\Exp\left[
\max_{1 \leq j \leq d} \left|
G_{{i_1},j}- g_j(X_{i_1})
\right|^{2\nu} 
\right] 
&\lesssim 
K^{-\nu} D_n^{2\nu} \log^{3\nu}(d)
\left(1 + K^{-\nu+1} \log^{\nu}(d)
\right)  \\
&\lesssim \left(n^{-1} r D_n^{2} \log^{2/q+1}(d)
\right)^{\nu},
\end{align*}
where the second inequality is due to~\eqref{U_full_assumption} and that $\nu \geq 7/6$ and $K = \lfloor (n-1)/(r-1) \rfloor$. 

\end{proof}

\subsubsection{Proof of Corollary~\ref{cor:comb_normalize}}\label{proof:comb_normalize}

\begin{proof}
We have shown in Step 0 of the proof (Subsection~\ref{proof:U_B}) for Theorem~\ref{thrm:U_B_valid} that
\begin{align*}
&\Pro\left(
\max_{1\leq j \leq d}|\widehat{\sigma}_{H,j}^2/\sigma_{H,j}^2 - 1| \log^2(d)
\lesssim C n^{-\zeta/2}
\right) \geq 1- C n^{-1}, 
\end{align*}
Further, if we take $\nu = 7/\zeta$ in Theorem~\ref{thrm:appr_U_full}, then in the proof for Theorem~\ref{lemma:appr_Y_A} and Theorem~\ref{thrm:appr_U_full}, we have shown that
\begin{align*}
&\Pro\left(
\max_{1\leq j \leq d}|\widehat{\sigma}_{g,j}^2/\sigma_{g,j}^2 - 1| \log^2(d)
\lesssim C n^{-3\zeta/7}
\right) \geq 1- C n^{-1}.
\end{align*}
The rest of the proof is  the same as the proof for~\cite[Corollary A.1]{chen2017randomized}, and thus omitted.
\end{proof}

\subsection{Proof of Lemma~\ref{lemma:lower_sigma_g}}\label{proof:lower_sigma_g}
\begin{proof}
Clearly, the inequality is for each dimension, and thus without loss of generality, we assume $d = 1$ and omit the dependence on $j$.

We denote $\Exp_{\beta}$ and $\Cov_{\beta}$ the expectation and covariance when $X_1,\ldots,X_r$ have densities $f_{\beta}$.
Further, define $g_{\beta}(x_1) = \Exp_{\beta}[h(x_1,X_2,\ldots,X_r)]$ for $x_1 \in S$ and by definition $g(\cdot) = g_0(\cdot)$.

First, note that by interchanging the order of integration and differentiation 
\[
\Exp_{\beta}\left[ \Psi(\beta) \right]
= \int  \left(\sum_{i=1}^{r} \nabla\ln f_{\beta}(x_i)\right)  \prod_{i=1}^{r} f_{\beta}(x_i) \mu(d x_i)
= \int \nabla\left(  \prod_{i=1}^{r} f_{\beta}(x_i) \right) \prod_{i=1}^{r} \mu(d x_i)
= 0.
\]
Further, by a similar argument,
\begin{align*}
\Cov_{\beta}(g_{\beta}(X_1), \Psi(\beta)) &= \int g_{\beta}(x_1)  \left(\sum_{i=1}^{r} \nabla\ln f_{\beta}(x_i)\right)  \prod_{i=1}^{r} f_{\beta}(x_i) \mu(d x_i) \\
&= \int g_{\beta}(x_1)   \nabla\ln f_{\beta}(x_1)   f_{\beta}(x_1) \mu(d x_1)\\
&= \int \left(\int h(x_1,x_2,\ldots,x_r) \prod_{i=2}^{r} f_{\beta}(x_i) \mu(d x_i) \right)  \left( \nabla\ln f_{\beta}(x_1) \right)  f_{\beta}(x_1) \mu(d x_1)\\
&= \int  h(x_1,x_2,\ldots,x_r)  \nabla\ln f_{\beta}(x_1)   \prod_{i=1}^{r} f_{\beta}(x_i) \mu(d x_i),
\end{align*}
which implies that
\begin{align*}
\sum_{i=1}^{r} \Cov_{\beta}(g_{\beta}(X_i), \Psi(\beta)) 
=\;& \int h(x_1,x_2,\ldots,x_r) \left(\sum_{i=1}^{r} \nabla\ln f_{\beta}(x_i) \right)  \prod_{i=1}^{r} f_{\beta}(x_i) \mu(d x_i) \\
=\;& \int h(x_1,x_2,\ldots,x_r) \nabla\left(  \prod_{i=1}^{r} f_{\beta}(x_i) \right) \prod_{i=1}^{r} \mu(d x_i)
= \nabla \theta(\beta).
\end{align*}
Finally, observe that
\begin{align*}
0 \leq\, &\Var_{\beta}\left(\sum_{i=1}^{r}g_{\beta}(X_i) - \nabla \theta(\beta)^T (r\cJ(\beta))^{-1} \Psi(\beta)\right) \\
=\,& \sum_{i=1}^{r}\Var_{\beta}\left(g_{\beta}(X_i)\right) - 2r^{-1}\Cov_{\beta}\left(\sum_{i=1}^{r}g_{\beta}(X_i), \nabla \theta(\beta)^T \cJ(\beta)^{-1} \Psi(\beta)\right) \\
\quad &+ r^{-2}\Var_{\beta}\left( \nabla \theta(\beta)^T \cJ(\beta)^{-1} \Psi(\beta)\right)\\
=\,&  r\Var_{\beta}\left(g_{\beta}(X_1)\right) -  r^{-1}\nabla \theta(\beta)^T \cJ(\beta)^{-1} \nabla \theta(\beta),
\end{align*}
which completes the proof.
\end{proof}

\subsection{Proofs of tail probabilities in Section \ref{subsec:tail_probabilities}}

\subsubsection{Proof of Lemma~\ref{tail_nonneg_sum}}
\label{proof:tail_nonneg_sum}
\begin{proof}
We first define 
$$
S :=\max_{1 \leq j \leq d}\sum_{i=1}^{m} Z_{ij}, \quad M := \max_{1\leq i \leq m} \max_{1\leq j \leq d} Z_{ij}.
$$ 
Then by the maximal inequality~\cite[Lemma 2.2.2]{van1996weak}, $\|M \|_{\psi_{\beta}} \leq C u_n \log^{1/\beta}(dm)$. By~\cite[Lemma E.4]{chernozhukov2017},
\begin{align*}
\Pro\left(
S \geq 2\Exp[S] + t
\right) \leq 3 \exp\left(
-
\left(\frac{t}{C \|M\|_{\psi_{\beta}}}
\right)^{\beta}
\right).
\end{align*}
The right hand side is $ 3/n$ if 
$$
t = C \|M\|_{\psi_{\beta}} \log^{1/\beta}(n) \leq 
C  u_n \log^{1/\beta}(n)\log^{1/\beta}(d m).
$$
Further by~\cite[Lemma E.3]{chernozhukov2017},
\begin{align*}
\Exp[S] &\lesssim \max_{1\leq j \leq d} \Exp\left[\sum_{i=1}^{m} Z_{ij} \right]
+ \log(d) \Exp[M] 
\lesssim \max_{1\leq j \leq d} \Exp\left[\sum_{i=1}^{m} Z_{ij} \right]
+ u_n \log^{1/\beta+1}(dm).
\end{align*}
Combining two parts finishes the proof.
\end{proof}
\vspace{0.2cm}

\subsubsection{Proof of Lemma~\ref{tail_general_sum_ind}}

\label{proof:tail_general_sum_ind}
\begin{proof}
We first define 
$$
S :=\max_{1 \leq j \leq d} \left \vert \sum_{i=1}^{m} Z_{ij} \right\vert, \quad M := \max_{1\leq i \leq m} \max_{1\leq j \leq d} \left\vert Z_{ij} \right\vert.
$$ 
Then by the maximal inequality~\cite[Lemma 2.2.2]{van1996weak}, $\|M \|_{\psi_{\beta}} \leq C u_n \log^{1/\beta}(dm)$. By~\cite[Lemma E.2]{chernozhukov2017},
\begin{align*}
\Pro\left( S \geq 2 \Exp[S] + t
\right) \leq \exp(-t^2/(3\sigma^2)) + 3\exp\left( 
-\left(\frac{t}{C \|M\|_{\psi_{\beta}}}
\right)^{\beta}
\right).
\end{align*}
The right hand side is $ 4/n$ if 
\begin{align*}
t &=\sqrt{3} \sigma \log^{1/2}(n) +  C \|M\|_{\psi_{\beta}} \log^{1/\beta}(n) \\
&\leq 
C \left( \sigma \log^{1/2}(n)  + \log^{1/\beta}(dm)\log^{1/\beta}(n) u_n
\right).
\end{align*}
Further by~\cite[Lemma E.1]{chernozhukov2017},
$$
\Exp[S] \lesssim \sigma \log^{1/2}(d)
+ \log(d) \sqrt{\Exp[M^2]}
\lesssim \sigma \log^{1/2}(d)
+ \log^{1/\beta+1}(dm)u_n.
$$
Combining two parts finishes the proof.

\end{proof}

\subsubsection{Proof of Lemma~\ref{tail_Bernoulli}}\label{proof:tail_Bernoulli}
\begin{proof}
We first define 
\begin{align*}
S &:=\max_{1 \leq j \leq d} \left \vert \sum_{i=1}^{m} (Z_{i}-p_n) a_{ij} \right\vert, \quad \tilde{M} := \max_{1\leq i \leq m} \max_{1\leq j \leq d} \left\vert (Z_{ij}-p_n)a_{ij} \right\vert \leq  \max_{1\leq i \leq m} \max_{1\leq j \leq d} \left\vert a_{ij} \right\vert  \\
\tilde{\sigma}^2 &:= \max_{1 \leq j \leq d} \sum_{i=1}^{m} \Exp[(Z_i-p_n)^2a_{ij}^2] \leq p_n(1-p_n) \max_{1 \leq j \leq d}\sum_{i=1}^{m} a_{ij}^2.
\end{align*}
By~\cite[Lemma E.2]{chernozhukov2017},
\begin{align*}
\Pro\left( S \geq 2 \Exp[S] + t
\right) \leq \exp(-t^2/(3\tilde{\sigma}^2)) + 3\exp\left( 
-\frac{t}{C \|\tilde{M}\|_{\psi_{1}}}
\right).
\end{align*}
The right hand side is $ 4/n$ if 
\begin{align*}
t &=\sqrt{3} \tilde{\sigma} \log^{1/2}(n) +  C \|\tilde{M}\|_{\psi_{1}} \log(n) \leq 
C \left( \sqrt{p_n(1-p_n)} \sigma \log^{1/2}(n)  + M \log(n)
\right).
\end{align*}
Further by~\cite[Lemma E.1]{chernozhukov2017},
$$
\Exp[S] \lesssim \tilde{\sigma} \log^{1/2}(d)
+ \log(d) \sqrt{\Exp[\tilde{M}^2]}
\lesssim \sqrt{p_n(1-p_n)} \sigma \log^{1/2}(d)
+ M \log(d).
$$
Combining two parts finishes the proof.
\end{proof}

\subsubsection{Proof of Lemma~\ref{tail U-non-negative}}\label{proof:tail_nonneg}

\begin{proof}
Let $m =\lfloor n/r \rfloor$, and
define the following quantity
$$Z_1:= \max_{1\leq j\leq d} \sum_{i=1}^{m} f_j(X^{ir}_{(i-1)r+1}),\;\;
M_1 := \max_{1\leq i \leq m}\max_{1\leq j\leq d} f_j(X^{ir}_{(i-1)r+1}).
$$
Then by the maximal inequality~\cite[Lemma 2.2.2]{van1996weak}, $\|M_1 \|_{\psi_{\beta}} \leq C u_n \log^{1/\beta}(dn)$. By~\cite[Lemma E.3]{chen2018gaussian},
\begin{align*}
\Pro\left( m \max_{1\leq j\leq d} U_{n,j} \geq 2\Exp[Z_1] + t
\right) \leq 3 \exp\left( -\left(\frac{t}{C \|M_1\|_{\psi_{\beta}}} \right)^{\beta}\right).
\end{align*}
The right hand side is $3/n$ if we set
$$
t = C \|M_1\|_{\psi_{\beta}} \log^{1/\beta}(n) \leq C u_n \log^{1/\beta}(dn)\log^{1/\beta}(n),
$$
Further, by~\cite[Lemma 9]{chernozhukov2015comparison},
\begin{align*}
\Exp[Z_1]&\leq C \left( \max_{1 \leq j \leq d} \Exp\left[
\sum_{i=1}^{m}  f_j(X^{ir}_{(i-1)r+1})\right]
 + \log(d) \Exp[M_1] \right)\\
&\leq C \left( m v_n  + u_n \log^{1/\beta+1}(dn)
\right).
\end{align*}
Putting two parts together, we have
\begin{align*}
\Pro\left( \max_{1\leq j\leq d} U_{n,j} \geq 
C \left(v_n + n^{-1} r u_n \log^{1/\beta+1}(dn)
+n^{-1} r u_n \log^{1/\beta}(dn)\log^{1/\beta}(n)
\right)
\right) \leq \frac{3}{n},
\end{align*}
which completes the proof.
\end{proof}

\subsubsection{Proof of Lemma~\ref{tail general U}}\label{proof:tail_general}

\begin{proof}
Let $m =\lfloor n/r \rfloor$, and define the following quantity
\begin{align*}
&Z_1:= \max_{1\leq j\leq d} \left\vert \sum_{i=1}^{m} f_j(X^{ir}_{(i-1)r+1}) \right\vert,\;\;
M_1 := \max_{1\leq i \leq m}\max_{1\leq j\leq d} \left\vert f_j(X^{ir}_{(i-1)r+1})\right\vert.
\end{align*}
Then by the maximal inequality~\cite[Lemma 2.2.2]{van1996weak}, $\|M_1 \|_{\psi_{\beta}} \leq C u_n \log^{1/\beta}(dn)$. By~\cite[Lemma C.3]{chen2017randomized}, 
\begin{align*}
\Pro\left( m \max_{1\leq j\leq d} \left\vert U_{n,j} 
\right\vert \geq 2 \Exp[Z_1] +t\right)
 \leq \exp\left(
\frac{-t^2}{3m \sigma^2} \right) + 3\exp\left(
-\left( \frac{t}{C \|M_1\|_{\psi_{\beta}}}\right)^{\beta}
\right).
\end{align*}
The right hand side is  $4/n$ if we take
\begin{align*}
t &= \sigma \sqrt{3m} \log^{1/2}(n) + C \|M_1 \|_{\psi_{\beta}} \log^{1/\beta}(n) \\
&\leq C \left( \sigma m^{1/2} \log^{1/2}(n) + u_n \log^{1/\beta}(dn)\log^{1/\beta}(n)
\right).
\end{align*}
Further, by~\cite[Lemma 8]{chernozhukov2015comparison},
\begin{align*}
\Exp[Z_1] \lesssim
\sqrt{\log(d) m \sigma^2} + \sqrt{\Exp[M_1^2]} \log(d)
\lesssim m^{1/2}\log^{1/2}(d) \sigma + u_n \log^{1/\beta+1}(dn).
\end{align*}
Putting two parts together completes the proof.
\end{proof}

\subsubsection{Proof of Lemma~\ref{lemma:nonneg_tail_random_kernel}}
\label{proof:nonneg_tail_random_kernel}
\begin{proof}
First, observe that 
$
\|F_j(x_1^r,W)\|_{\psi_{\beta}} \; \lesssim \;
f_j(x_1^r) + b_j(x_1^r)
$.
Denote 
$$
Z_1 := \max_{1 \leq j \leq d} \frac{1}{|I_{n,r}|}\sum_{\iota \in I_{n,r}}f_j(X_{\iota}),
\quad
M_1 :=  \max_{\iota \in I_{n,r}}\max_{1 \leq j \leq d} \left(f_j(X_{\iota})+ b_j(X_{\iota})\right),
$$
Then conditional on $X_1^n$, by Lemma~\ref{tail_nonneg_sum}, 
$$
\Pro_{\vert X_1^n}\left(Z \geq 
C\left(Z_1
+ |I_{n,r}|^{-1}M_1 r^{2/\beta}\log^{1/\beta+1}(dn)\log^{1/\beta-1}(n)
\right)
\right) \leq \frac{3}{|I_{n,r}|} \leq \frac{3}{n}.
$$
By Lemma~\ref{tail U-non-negative},
$$
\Pro\left(
Z_1 \geq C\left(
\max_{1 \leq j \leq d} \Exp[f_j(X_1^r)] +
n^{-1}r \log^{1/\beta+1}(dn)\log^{1/\beta-1}(n) u_n
\right)
\right) \leq  \frac{3}{n}.
$$
Further, by maximal inequality~\cite[Lemma 2.2.2]{van1996weak}
$$
\| M_1 \|_{\psi_\beta} \leq C r^{1/\beta} \log^{1/\beta}(dn) u_n \;\;
\Rightarrow \;\;
\Pro\left(M_1 \geq C r^{1/\beta} \log^{1/\beta}(n)\log^{1/\beta}(dn) u_n \right) \leq \frac{2}{n}.
$$
Then the proof is complete by combining above results.
\end{proof}
\vspace{0.3cm}

\subsubsection{Proof of Lemma~\ref{lemma:tail_random_kernel}}
\label{proof:tail_random_kernel}
\begin{proof}
First, we define
\begin{align*}
&\sigma^2 := \max_{1 \leq j \leq d} \sum_{\iota \in I_{n,r}} \Exp_{\vert X_1^n}
\left[
\left(F_j(X_{\iota}, W_{\iota}) - f_j(X_{\iota}) \right)^2
\right] 
\lesssim  \max_{1 \leq j \leq d} \sum_{\iota \in I_{n,r}}  b_j^2(X_{\iota}),
\\
&M := \max_{\iota \in I_{n,r}} \max_{1\leq j \leq d} b_j(X_{\iota}).
\end{align*}
Then by first conditional on $X_1^{n}$ and by Lemma~\ref{tail_general_sum_ind},
\begin{align*}
\Pro\left(
|I_{n,r}| Z \geq C(\sigma r^{1/2} \log^{1/2}(dn)  + M r^{2/\beta}\log^{1/\beta+1}(dn)\log^{1/\beta-1}(n)) \right) \leq \frac{4}{|I_{n,r}|} \leq \frac{4}{n}.
\end{align*}

Observe that 
$$\| b_j^2(X_1^{r}) \|_{\psi_{\beta/2}}
= \| b_j(X_1^{r}) \|_{\psi_{\beta}}^2 \leq u_n^2.
$$
Then by Lemma~\ref{tail U-non-negative} with $\psi_{\beta/2}$, 
\begin{align*}
\Pro\left(
\frac{\sigma^2}{|I_{n,r}|} \geq 
C u_n^2\left(
1 + n^{-1} r \log^{2/\beta+1}(dn)
\log^{2/\beta-1}(n)
\right)
\right) \leq \frac{3}{n}.
\end{align*}
Further, by maximal inequality~\cite[Lemma 2.2.2]{van1996weak}
$$
\| M \|_{\psi_\beta} \leq C r^{1/\beta} \log^{1/\beta}(dn) u_n \;\;
\Rightarrow \;\;
\Pro(M \geq C r^{1/\beta} \log^{1/\beta}(dn) \log^{1/\beta}(n)u_n) \leq \frac{2}{n}.
$$
Then the proof is complete by combining above results.
\end{proof}
\vspace{0.3cm}

\subsection{Proofs of additional lemmas}\label{additional_proofs}
The following lemma is similar to \cite[Lemma C.1]{chernozhukov2017}, and is needed in proving Lemma \ref{lemma:GAR_q}.

\begin{lemma}\label{third_moment_q}
Let $q \in (0,3]$, and $\xi$ be a non-negative random variable such that $\|\xi\|_{\psi_q} \leq D$. Then there exists a constant $C$, depending only on $q$, such that
\[
\Exp\left[ \xi^3; \xi > t\right] \leq C (t^3 + D^3) e^{-(t/D)^q}, \quad
 \text{ for } t > 0.
\]
\end{lemma}
\begin{proof}
Since $\|\xi\|_{\psi_q} \leq D$, we have for $x > 0$,
\[
\Pro(\xi > x) \leq e^{-(x/D)^q} \Exp\left[ e^{-(\xi/D)^q}\right] 
\leq 2  e^{-(x/D)^q}.
\]
By change of variable, we have
\begin{align*}
&\Exp\left[ \xi^3; \xi > t\right] \leq t^3 \Pro(\xi > t) + 3 \int_{t}^{\infty} \Pro(\xi > x) x^2 d\,x \\
&\lesssim t^3 e^{-(t/D)^q} + D^3  \int_{(t/D)^q}^{\infty} e^{-u} u^{3/q-1} d\,u \\
& \lesssim t^3 e^{-(t/D)^q} + D^3  e^{-(t/D)^q} \int_{0}^{\infty} e^{-u} \left(u+(t/D)^q \right)^{3/q-1} d\,u\\
& \lesssim t^3 e^{-(t/D)^q} + D^3  e^{-(t/D)^q} \int_{0}^{\infty} e^{-u} \left( u^{3/q-1}+(t/D)^{3-q} \right) d\,u \\
& \lesssim t^3 e^{-(t/D)^q} + D^3  e^{-(t/D)^q} \int_{0}^{\infty} e^{-u} \left( u^{3/q-1}+(t/D)^{3-q} \right) d\,u \\ 
& \lesssim \left( t^3  + D^3  + t^{3-q} D^q\right) e^{-(t/D)^q}
 \lesssim \left( t^3  + D^3 \right) e^{-(t/D)^q}.
\end{align*}

\end{proof}

\begin{proof}[Proof of Lemma \ref{lemma:GAR_q}]\label{proof:GAR_q}
For $q \geq 1$, it has been established by \cite[Proposition 2.1]{chernozhukov2017}. For $q < 1$, the proof is almost identical to that for \cite[Proposition 2.1]{chernozhukov2017}, except that we replace \cite[Lemma C.1]{chernozhukov2017} by Lemma \ref{third_moment_q}.
\end{proof}

\begin{proof}[Proof of Lemma \ref{triangle_max_inequ}]
(i). Without loss of generality, we assume $0 < x := \|X\|_{\psi_\beta} < \infty$, and $0 < y := \|Y\|_{\psi_\beta} < \infty$. Observe that
\begin{align*}
&\Exp\left[\exp\left( \frac{|X+Y|}{2^{1+1/\beta}(x+y)}
\right)^\beta
\right]
\leq 
\Exp\left[\exp\left( \frac{|X|^{\beta}+|Y|^{\beta}}{2(x+y)^{\beta}}
\right)
\right]\\
\leq 
&\Exp\left[\frac{1}{2}\exp\left( \frac{|X|^{\beta}}{(x+y)^{\beta}}
\right)
\right]
+
\Exp\left[\frac{1}{2}\exp\left( \frac{|Y|^{\beta}}{(x+y)^{\beta}}
\right)
\right] \leq 2.
\end{align*}

(ii).
From Lemma \ref{psi_prop}, for $1 \leq i \leq  n$,
\begin{align*}
\Exp\left[\widetilde{\psi}_\beta\left(\frac{|\xi_i|}{D} \right) \right]
\leq \Exp\left[{\psi}_\beta\left(\frac{|\xi_i|}{D} \right) \right]
+1 \leq 2,
\end{align*}
which, by the convexity of $\widetilde{\psi}_{\beta}$ and the fact $\widetilde{\psi}_{\beta}(0) = 0$, implies $\|\xi_i\|_{\widetilde{\psi}_{\beta}} \leq 2D$.
By the standard maximal inequality (e.g., see~\cite[Lemma 2.2.2]{van1996weak}) and Lemma \ref{psi_prop},
$\|\max_{1 \leq i \leq n} \xi_i\|_{\widetilde{\psi}_{\beta}} \leq C \log^{1/\beta}(n) D$. Thus by Lemma \ref{psi_prop},
\begin{align*}
\Exp\left[\exp
\left(\frac{\max_{1 \leq i \leq n} \xi_i}{C \log^{1/\beta}(n) D} \right)^{\beta}
\right] \leq 
\Exp\left[\psi_{\beta}
\left(\frac{\max_{1 \leq i \leq n} \xi_i}{C \log^{1/\beta}(n) D} \right)
\right] + e^{1/\beta}   \leq  1+e^{1/\beta}.
\end{align*}
Now we let $m \geq 1$ such that $\left(1 + e^{1/\beta} \right)^{1/m} \leq 2$. Then by Jensen's inequality ($\Exp[X^{1/m}] \leq \left(\Exp[X]\right)^{1/m}$ for $X > 0$ a.s.),
\begin{align*}
\Exp\left[\exp
\left(\frac{\max_{1 \leq i \leq n} \xi_i}{C m^{1/\beta}\log^{1/\beta}(n) D} \right)^{\beta}
\right] \leq 2,
\end{align*}
which implies that $\|\max_{1 \leq i \leq n} \xi_i\|_{\widetilde{\psi}_{\beta}} \lesssim \log^{1/\beta}(n) D$.
\end{proof}

\section*{Acknowledgements}
X. Chen is supported in part by NSF DMS-1404891, NSF CAREER Award DMS-1752614, and UIUC Research Board Awards (RB17092, RB18099).

\bibliographystyle{plain}
\bibliography{arXiv_v3}

\end{document}